\documentclass[reqno,12pt]{amsart}
\usepackage{amsmath,amssymb,latexsym,soul,cite,mathrsfs}
\usepackage{xcolor,enumitem,graphicx}
\usepackage[colorlinks=true,urlcolor=blue,
citecolor=red,linkcolor=blue,linktocpage,pdfpagelabels,
bookmarksnumbered,bookmarksopen]{hyperref}
\usepackage[english]{babel}
\usepackage[left=2.6cm,right=2.6cm,top=2.9cm,bottom=2.9cm]{geometry}
\usepackage[hyperpageref]{backref}
\newtheorem{theoremletter}{Theorem}

\newtheorem{theo}{Theorem}[section]

\newtheorem{definition}{Definition}[section]

\newtheorem{lemma}{Lemma}[section]

\newtheorem{prop}{Proposition}[section]

\newtheorem{cor}{Corollary}[section]

\numberwithin{equation}{section}

\title[Trudinger-Moser type inequalities for the Hessian equation]{Trudinger-Moser type inequalities for the Hessian equation with logarithmic weights}

\author[J.M. do \'O]{Jo\~ao Marcos do \'O*}
\author[J.F. de Oliveira]{Jos\'e Francisco de Oliveira}
\author[R.C. Ponciano]{Raon\'{\i} Cabral Ponciano}

\address[J.M. do \'O]{Department of Mathematics,
	Federal University of Para\'{\i}ba
	\newline\indent 
	58051-900, Jo\~ao Pessoa-PB, Brazil}
\email{\href{mailto:jmbo@mat.ufpb.br}{jmbo@mat.ufpb.br}}

\address[J.F. de Oliveira]{
Department of Mathematics,
Federal University of Piau\'{i}
 \newline\indent 
	64049-550, Teresina, PI, Brazil}
\email{\href{mailto:jfoliveira@ufpi.edu.br}{jfoliveira@ufpi.edu.br}}

\address[R.C. Ponciano]{Department of Mathematics,
	Federal University of ABC
	\newline\indent
	09280-560, Santo Andr\'e-SP, Brazil}
\email{\href{mailto:raoni.ponciano@ufabc.edu.br}{raoni.ponciano@ufabc.edu.br}}

\thanks{*Corresponding author.}

\subjclass[2020]{35J60, 35B33, 46E35, 26D10}
\keywords{$k$-Hessian operator, Hardy Inequalities, Trudinger-Moser inequality, Extremals, Attainability, Weighted Sobolev spaces}

\begin{document}
	\begin{abstract}
		We establish sharp Trudinger–Moser inequalities with logarithmic weights for the $k$-Hessian equation and investigate the existence of maximizers. Our analysis extends the classical results of Tian and Wang to $k$-admissible function spaces with logarithmic weights, providing a natural complement to the work of Calanchi and Ruf. Our approach relies on transforming the problem into a one-dimensional weighted Sobolev space, where we solve it using various techniques, including some radial lemmas and certain Hardy-type inequalities, which we establish in this paper, as well as a theorem due to Leckband.
\end{abstract}
	\maketitle	
	\begin{center}
		\footnotesize
		\tableofcontents
	\end{center}

\section{Introduction}
Since the pioneering works by Caffarelli-Nirenberg-Spruck~\cite{zbMATH04067434,MR739925,MR780073}, the $k$-Hessian equation (or $k$-Hessian operator) has emerged as a topic of great importance in nonlinear analysis, geometric analysis and the calculus of variations. Prototypical features such as being fully nonlinear for $k>1$, having a divergent form, potentially being either elliptic or non-elliptic (degenerate or not) depending on the context, and its strong connection with geometric analysis problems make it a highly relevant and interesting research subject, we refer \cite{Wang2009,MR1275451} for a comprehensive survey. As a result, there have been several studies in different directions, we recommend \cite{chou2001variational,de2020admissible,zbMATH06712355,chou1996critical,tso1990remarks} for the study of classical solutions via variational methods, \cite{trudinger1995dirichlet,chou2001variational} for regularity estimates, \cite{trudinger1999hessian,trudinger2002hessian,trudinger1997hessian} for Hessian measures, and \cite{phuc2008quasilinear,phuc2009singular,zbMATH01820865} for the existence of either classical or viscosity solutions based on Wolff's potentials.


 The existence of classical solutions typically necessitates the ellipticity condition. In this framework, it is essential to constrain the action of the operator to the space of $k$-admissible functions, whose structure aligns with both the divergence form and the ellipticity of the $k$-Hessian operator. This restriction not only facilitates a rigorous variational analysis but also enables the derivation of regularity estimates \cite{chou2001variational}.

Our goal in this paper is to obtain sharp estimates compatible with the variational study of the $k$-Hessian equation in extreme regimes, where there is a strong loss of compactness in the energy associated functional. 

In fact, we want to extend the Trudinger-Moser type inequality by Tian-Wang~\cite{zbMATH05782190} for the log-weighted $k$-admissible function spaces which represent the counterpart of the results due to Calanchi-Ruf~\cite{CALANCHI20151967,CALANCHI2015403} for the $k$-Hessian equation. In addition, we investigate the associated extremal problem providing an extension of the results of Roy~\cite{zbMATH07074659,zbMATH06562450} and Nguyen~\cite{zbMATH07122708} in the $n$-Laplacian framework and by de Oliveira, Do \'{O}, and Ruf~\cite{zbMATH07238535} for the $k$-Hessian equation. As a byproduct of our approach, we provide general log-weighted Hardy-type inequalities, which are of independent interest in the context of weighted Sobolev spaces, including fractional dimensions. These inequalities have applications in general classes of differential operators in the radial form and have been addressed in several works; see, for instance, without any claim of completeness \cite{zbMATH01127685,zbMATH07714682,zbMATH07972414,zbMATH07827601,zbMATH07946843,zbMATH06576572,zbMATH06654683,zbMATH06343604,zbMATH01530501,zbMATH07214288,XueZhangZhu2025,MR4790800,MR4870665} and the references therein.

To formalize our results, we begin with notation and key results on the 
$k$-Hessian equation. Let $ \Omega$ be a smooth bounded domain in $\mathbb{R}^n $. For $k=1,\ldots,n$ and $u\in C^2(\Omega)$, the $k$-th Hessian operator $S_k(D^2u)$ is defined by the sum of all $k \times k$ principal minors of the Hessian matrix $D^2u$. Hence $S_1(D^2u)=\Delta u$ and $S_n(D^2u)=\det(D^{2}u)$. The fully nonlinear operators $S_k(D^2u)$, $k=2,\ldots, n$ are not elliptic in the entire space $C^2(\Omega)$. To ensure the ellipticity, we must restrict ourselves to the space of $k$-admissible functions. Following \cite{zbMATH04067434,MR1275451}, we say a function $u\in C^2(\Omega)\cap C^0(\overline{\Omega})$ is $k$-admissible on $\Omega$ if $S_j(D^2u)\ge 0$ for $j=1,\ldots,k$. We denote by $\Phi^k(\Omega)$ the space of all $k$-admissible functions on $\Omega$ and by $\Phi^{k}_{0}(\Omega)$ the subspace of functions $u\in\Phi^{k}(\Omega)$ that satisfy the boundary condition $u_{\mid_{\partial\Omega}}=0$. Acting on the space $\Phi^{k}_{0}(\Omega)$, the $k$-Hessian operators for $1\le k\le n$ are elliptic and have a divergent form, making them compatible with the variational study (cf.\cite{chou2001variational}) in $\Phi^{k}_{0}(\Omega)$ endowed with the norm
\begin{equation*}
\|u\|_{\Phi_0^k}= \Big(\int_{\Omega}(-u)S_{k}(D^2 u)\, \mathrm{d}x\Big)^{\frac{1}{k+1}}, \text{ for all } u\in \Phi^{k}_{0}(\Omega
),
\end{equation*}
see for instance \cite{MR1275451,Wang2009} for more details.

The classical Sobolev embedding for $\Phi^k_0(\Omega)$ was established in \cite[Theorem 5.2]{MR1275451}. Actually, under the Sobolev condition $1 \leq k < \frac {n}{2}$, it was shown that the following continuous embedding holds
\begin{equation*}
\Phi_0^k(\Omega)\hookrightarrow L^p(\Omega), \; \text{ for all }\; p \in [1,k^*], \text{ with } k^* = \frac{n(k+1)}{n-2k}.
\end{equation*}
For the borderline case $n = 2k$, Tian-Wang~\cite{zbMATH05782190} proved the following sharp inequality of Trudinger-Moser type \cite{zbMATH03261965,zbMATH03323360}
\begin{equation}\label{TMinequality} 
\sup_{\|u\|_{\Phi^{k}_{0}}=1} \int_{\Omega}\mathrm{e}^{\alpha_n |u|^{\frac{n+2}{n}}}\, \mathrm{d} x \leq C,
\end{equation}
where 
\begin{equation} \label{critical-constant}
\alpha_n=nc^{\frac{2}{n}}_n, \ \mbox{ with } c_n=\frac{\omega_{n-1}}{k}\binom{n-1}{k-1},
\end{equation}
$\omega_{n-1}$ is the area of the unit sphere in $\mathbb{R}^{n}$, and $C$ is a positive constant depending only on $n$ and $\mathrm{diam}(\Omega)$. The existence of maximizers for \eqref{TMinequality} in the case $\Omega=B$, the unit ball in $\mathbb{R}^n$, was recently investigated in \cite{zbMATH07238535}. The inequality \eqref{TMinequality} and its extensions, play a central role in the variational study of the $k$-Hessian equation under the critical exponential growth condition; see, for instance, \cite{de2020admissible,zbMATH07931768} and the references therein.  This paper investigates Trudinger–Moser-type inequalities for the Hessian equation in the presence of a logarithmic weight. We further examine the existence of maximizers for the associated variational problem within the framework of $k$-admissible function spaces.

\subsection*{Description of the results}

Since we are dealing with inequalities and maximization problems in non-rearrangement-invariant Banach spaces \cite{zbMATH05782190} and given the nature of the results in \cite{CALANCHI20151967,CALANCHI2015403}, we will carry out our analysis in $\Phi^{k}_{0,\mathrm{rad}}(B)$, the subspace of radially symmetric functions in $\Phi^{k}_0(B)$, where $B$ the unit ball in $\mathbb{R}^n$ centered at the origin $0\in\mathbb{R}^n$.

For any $u\in\Phi^{k}_{0,\mathrm{rad}}(B)$ with $u(x)=v(r)$, where $r=|x|$, we have 
\begin{equation}\label{fullnorm}
\|u\|_{\Phi^{k}_{0}}=\left(c_{n}\int_{0}^{1}r^{n-k}|v^{\prime}|^{k+1}\,\mathrm{d} r\right)^{\frac{1}{k+1}},
\end{equation}
where $c_n$ is given by \eqref{critical-constant}. Motivated by \eqref{fullnorm}, for any $w(x)=w(r)$ positive radial weight function on $B$ and $u\in\Phi^{k}_{0,\mathrm{rad}}(B)$, we set
\begin{equation}\label{normPhi}
 \|u\|_{\Phi,w}=\Big(c_{n}\int_{0}^{1}r^{n-k}|v^{\prime}|^{k+1}w \, \mathrm{d} r\Big)^{\frac{1}{k+1}}.  
\end{equation}
Then, we define
\begin{equation}\label{defPhi}
\Phi^{k}_{0,\mathrm{rad}}(B, w)=\mathrm{cl}\left\{u\in \Phi^{k}_{0,\mathrm{rad}}(B)\;:\; \|u\|_{\Phi,w}<\infty \right\},
\end{equation}
where the closure is taken with respect to the norm \eqref{normPhi}. Here, we are mainly interested in the prototypes 
\begin{equation}\label{w-types}
w_0(x)=\left(\ln\frac{1}{|x|}\right)^{\frac{\beta n}{2}} \quad \text{ or } \quad  w_1(x)=\left(\ln\frac{\mathrm{e}}{|x|}\right)^{\frac{\beta n}{2}},\;\;\mbox{with}\;\; \beta\in [0,1).
\end{equation}
By a slight abuse of notation, we will denote the functions $w_0(r):=w_0(x)$ and $w_1(r):=w_1(x)$, where $r=|x|$, using the same letter $w$.
\\

\subsection{Log-weighted Trudinger-Moser inequalities for \texorpdfstring{$k$}{}-Hessian and maximizers}
In this section we will present Trudinger-Moser type inequalities in the $k$-admissible functions spaces $\Phi^{k}_{0,\mathrm{rad}}(B, w)$ for both log-weights $w=w_0$ and $w=w_1$ in \eqref{w-types}, which represent the counterpart of the results due to Calanchi-Ruf~\cite{CALANCHI20151967,CALANCHI2015403} for $k$-Hessian operators. In addition, we will investigate the existence of maximizers for the corresponding extremal problem providing an extension of previous results by Roy~\cite{zbMATH07074659,zbMATH06562450} and Nguyen~\cite{zbMATH07122708} for $n$-Laplace equation and by de Oliveira, do \'{O}, and Ruf~\cite{zbMATH07238535} even in the context of $k$-Hessian equation.  \\
\begin{theo}\label{thm1}
Assume $k=n/2$ and let $\beta\in [0,1)$ and let $w$ be given by \eqref{w-types}. 
\begin{enumerate}
\item [$(a)$ ] Then for any $u\in \Phi^{k}_{0,\mathrm{rad}}(B, w)$,
\begin{equation*}
\int_{B}\mathrm{e}^{|u|^{\gamma}}\, \mathrm{d} x < \infty  \Longleftrightarrow \gamma\leq \gamma_{n,\beta}:=\frac{n+2}{n(1-\beta)}.
\end{equation*}
\item [$(b)$ ] Let $\Sigma:=\left\{u\in \Phi^{k}_{0,\mathrm{rad}}(B, w)\,:\,\|u\|_{\Phi, w}\leq 1 \right\}$, then
\begin{equation}\label{S-attainable}
\mathrm{MT}(n,\alpha,\beta)=\sup_{u\in \Sigma}\int_{B}\mathrm{e}^{\alpha|u|^{\gamma_{n,\beta}}}\, \mathrm{d} x < \infty
\Longleftrightarrow  \alpha\leq \alpha_{n,\beta}=n\left[c^{\frac{2}{n}}_n(1-\beta)\right]^{\frac{1}{1-\beta}}.
\end{equation}
\end{enumerate}
\end{theo} 
By the monotonicity formula \cite[Lemma 3.1]{zbMATH05782190} the proof of \eqref{TMinequality} can be reduced to radial functions $\Phi^{k}_{0,\mathrm{rad}}(B)$. Thus, for $\beta=0$, Theorem~\ref{thm1} recovers the Trudinger-Moser inequality for the Hessian equation proved by Tian-Wang~\cite{zbMATH05782190}. 

The next result concerns embeddings for the $\Phi^{k}_{0,\mathrm{rad}}(B, w)$ logarithmic weight $w=w_1$.

\begin{theo}\label{thm1B} Suppose $k=n/2$ and $w_1(x)=\big(\ln\frac{\mathrm{e}}{|x|}\big)^{\frac{\beta n}{2}}$ with $\beta>1$. Then we have the continuous embedding
    \begin{align*}
      \Phi^{k}_{0,\mathrm{rad}}(B, w_1) \hookrightarrow L^{\infty}(B).
    \end{align*}
\end{theo}
Note that, for $\beta > 1$, the functions of $\Phi^{k}_{0,\mathrm{rad}}(B, w_1)$ with $n=2k$ are bounded. This asserts that the optimal growth \cite{zbMATH03325896}, i.e. the maximal growth for a real function $\varphi$ such that $\varphi(u)\in L^{1}(B)$ for any $u\in \Phi^{k}_{0,\mathrm{rad}}(B, w_1)$, is very slow and shows no signs of a criticality phenomenon. In fact, in the next result we show that
$\beta=1$ is a threshold with double-exponential optimal growth.
\begin{theo}\label{thm2} Assume $k=n/2$ and let $w_1(x)=\big(\ln\frac{\mathrm{e}}{|x|}\big)^{\frac{n}{2}}$. 
\begin{enumerate}
\item [$(a)$ ] Then, for any $u\in \Phi^{k}_{0,\mathrm{rad}}(B, w_1)$
\begin{equation*}
\int_{B}\mathrm{e}^{\mathrm{e}^{|u|^{\frac{n+2}{n}}}}\,\mathrm{d} x < \infty.
\end{equation*}
\item [$(b)$ ] Let $\Sigma_1:=\left\{u\in \Phi^{k}_{0,\mathrm{rad}}(B, w_1)\,:\,\|u\|_{\Phi, w_1}\leq 1 \right\}$, then
\begin{equation*}
\sup_{u\in \Sigma_1}\int_{B}\mathrm{e}^{a\mathrm{e}^{c^{\frac{2}{n}}_n|u|^{\frac{n+2}{n}}}}\,\mathrm{d} x < \infty\;\;\Longleftrightarrow \;\; a\leq n.
\end{equation*}
\end{enumerate}
\end{theo}
Theorems~\ref{thm1}, \ref{thm1B}, and \ref{thm2} extend previous results of Calanchi-Ruf~\cite{CALANCHI20151967,CALANCHI2015403} and Xue-Zhang-Zhu~\cite{XueZhangZhu2025} for the space of $k$-admissible functions $\Phi^{k}_{0,\mathrm{rad}}(B,w)$.

Our next result concerns the attainability of the Trudinger–Moser type \eqref{S-attainable} in the critical regime $\alpha=\alpha_{n,\beta}$ and $w=w_0(x)=(\ln \frac{1}{|x|})^{\frac{\beta n}{2}}$.
\begin{theo}\label{thm-extremal} Let $w=w_0(x)=(\ln \frac{1}{|x|})^{\frac{\beta n}{2}}$ in \eqref{S-attainable}. Then, there exists $\beta_0\in (0,1)$ such that $\mathrm{MT}(n,\alpha_{n,\beta},\beta)$ is attained for any $\beta\in [0,\beta_0)$.
\end{theo} 
Theorem~\ref{thm-extremal} extends previous results by Roy\cite{zbMATH07074659,zbMATH06562450} and Nguyen \cite{zbMATH07122708} for the $k$-Hessian equation framework and improves \cite{zbMATH07238535} even within the $k$-Hessian scenario. 

\subsection{Log-weighted Hardy-type inequalities} In this section, we will present Hardy-type inequalities with logarithmic weights. Notably, the results stated below emerge as a byproduct of a refined characterization of log-weighted Hardy-type inequalities, which improve and complement previous findings in \cite{zbMATH00046945}. See Propositions \ref{hardy1}, \ref{propOK2}, \ref{propOK3}, \ref{hardy1e}, \ref{propOK2e}, \ref{propOK3e} in Section~\ref{section2} for a more general result that still incorporates a logarithmic weight in the left-hand integral.

For $R>0$, let us denote by $AC_R(0,R)$ as the set of all locally absolutely continuous functions $v\colon(0,R)\to\mathbb R$ satisfying the right-hand homogeneous boundary condition $\lim_{r\to R}v(r)=0$. Our next two theorems completely characterize the conditions under which the weighted Hardy-type inequality holds for functions in $AC_R(0,R)$, covering both cases $w=w_0$ and $w=w_1$. 

\begin{theo}\label{theo21}
Let $v\in AC_R(0,R)$ with $0<R<\infty$. Given $\alpha,\beta,n,k,p\in\mathbb R$ with $k\geq0$ and $p\geq1$, the inequality
\begin{equation}\label{aosfnas}
\left(\int_0^Rr^\alpha|v|^p\,\mathrm{d} r\right)^{\frac1p}\leq C\left(\int_0^Rr^{n-k}\left(\ln\frac{1}{r}\right)^{\frac{\beta n}2}|v^{\prime}|^{k+1}\,\mathrm{d} r\right)^{\frac{1}{k+1}}
\end{equation}
holds for some constant $C=C(\alpha,\beta,n,k,p,R)>0$ if, and only if, one of the following conditions is fulfilled:
\begin{flushleft}
    $\mathrm{(i)}$ $\alpha>-1$, $\beta\geq0$, $n<0$, and $k=0$;\\
    $\mathrm{(ii)}$ $\alpha>-1$, $n=0$, and $k=0$;\\
    $\mathrm{(iii)}$ $\alpha>-1$, $\beta=0$, $n>0$, $k=0$, and $p\leq\frac{\alpha+1}n$;\\
    $\mathrm{(iv)}$ $\alpha>-1$, $\beta n<2k$, $n\leq 2k$, and $k>0$;\\
    $\mathrm{(v)}$ $\alpha>-1$, $\beta n<2k$, $n>2k$, $k>0$, and $p<\frac{(\alpha+1)(k+1)}{n-2k}$;\\
    $\mathrm{(vi)}$ $\alpha>-1$, $0\leq\beta n<2k$, $n>2k$, $k>0$, and $k+1\leq p=\frac{(\alpha+1)(k+1)}{n-2k}$;\\
    $\mathrm{(vii)}$ $\alpha>-1$, $\frac{2k+2}p-2<\beta n<2k$, $n>2k$, $k>0$, and $p=\frac{(\alpha+1)(k+1)}{n-2k}<k+1$.
\end{flushleft}
\end{theo}
\begin{theo}\label{theo21e}
Let $v\in AC_R(0,R)$ with $0<R<\infty$. Given $\alpha,\beta,n,k,p\in\mathbb R$ with $k\geq0$ and $p\geq1$, the inequality
\begin{equation}\label{aosfnase}
\left(\int_0^Rr^\alpha|v|^p\,\mathrm{d} r\right)^{\frac1p}\leq C\left(\int_0^Rr^{n-k}\left(\ln\frac{\mathrm{e}}{r}\right)^{\frac{\beta n}2}|v^{\prime}|^{k+1}\,\mathrm{d} r\right)^{\frac{1}{k+1}}
\end{equation}
holds for some constant $C=C(\alpha,\beta,n,k,p,R)>0$ if, and only if, one of the following conditions is fulfilled:
\begin{flushleft}
    $\mathrm{(i)}$ $\alpha>-1$, $n\leq0$, and $k=0$;\\
    $\mathrm{(ii)}$ $\alpha>-1$, $n\leq2k$, and $k>0$;\\
    $\mathrm{(iii)}$ $\alpha>-1$, $\beta\leq0$, $0<n<\frac{\alpha+1}p$ and $k=0$;\\
    $\mathrm{(iv)}$ $\alpha>-1$, $\beta=0$, $n=\frac{\alpha+1}p$ and $k=0$;\\
    $\mathrm{(v)}$ $\alpha>-1$, $\beta>0$, $n>0$, and $k=0$;\\
    $\mathrm{(vi)}$ $\alpha>-1$, $n>2k$, $k>0$, and $p<\frac{(\alpha+1)(k+1)}{n-2k}$;\\
    $\mathrm{(vii)}$ $\alpha>-1$, $\beta\geq0$, $n>2k$, $k>0$, and $p=\frac{(\alpha+1)(k+1)}{n-2k}$.
\end{flushleft}
\end{theo}

For $p\ge 1$ and $\alpha>-1$, let us denote by $L^p_\alpha=L^p_\alpha(0,R)$ the Lebesgue weighted space of the Lebesgue measurable functions $v:(0,R)\to\mathbb{R}$ such that $$\|v\|_{L^{p}_{\alpha}}=\left(\int_{0}^{R}r^{\alpha}|v|^p\mathrm{d}r\right)^{\frac{1}{p}}<\infty.$$
Now, we define the weighted Sobolev space $X^{1,k+1}_{R,w}$ given by
\begin{equation*}
X^{1,k+1}_{R,w}=\mathrm{cl}\Big\{v\in AC_{R}(0,R) \;:\; \|v\|_{w}=\Big(c_{n}\int_{0}^{R}r^{n-k}|v^{\prime}|^{k+1}w\,\mathrm{d} r\Big)^{\frac{1}{k+1}}<\infty \Big\}
\end{equation*}
where the closure is taken with respect to the norm $\|\cdot\|_w$, $w:(0,R)\to\mathbb{R}$ is a weight function, and the constant $c_n$ is defined in \eqref{critical-constant}. 

Our next main result establishes embeddings from $X^{1,k+1}_{R,w}$ the logarithmic weighted Sobolev space into $L^p_\alpha$ the weighted Lebesgue space, provided that the following technical condition holds when $p$ is the critical exponent:
\begin{equation}\label{hip:criticalcase}
\left\{\begin{array}{ll}
\left\{\begin{array}{ll}
\beta\geq0,&\mbox{if }\alpha+1\geq n-2k,\\
\frac{\beta n}2>\frac{n-2k}{\alpha+1}-1,&\mbox{if }\alpha+1<n-2k,
\end{array}\right.,&\mbox{if }w(r)=\left(\ln\frac{1}{r}\right)^{\frac{\beta n}{2}},\\
\beta\geq0,&\mbox{if }w(r)=\left(\ln\frac{\mathrm{e}}{r}\right)^{\frac{\beta n}{2}}.
\end{array}\right.
\end{equation}
In fact, we have the following:
\begin{cor}\label{thm0}
Let $v\in X^{1,k+1}_{R,w}$ and $\alpha,\beta,n,p,k\in\mathbb R$ with $0<R<\infty$, $\alpha>-1$, $n\geq1$, $p\geq1$, and $k\geq1$. Additionally, in the case $w=w_0$, we assume $\beta n<2k$.
\begin{flushleft}
\noindent$\mathrm{(i)}$ If $k<\frac{n}2$, then the following continuous embedding holds
\begin{equation*}
X^{1,k+1}_{R,w}\hookrightarrow L^p_\alpha(0,R)\quad\mbox{for all }1\leq p\leq k^*:=\frac{(\alpha+1)(k+1)}{n-2k},
\end{equation*}
where we need to assume \eqref{hip:criticalcase} for $p=k^*$. Moreover, assuming $\alpha\geq0$, the embedding is compact for $1\leq p<k^*$;\\
$\mathrm{(ii)}$ If $k=\frac{n}2$, then the following continuous embedding holds
\begin{equation*}
X^{1,k+1}_{R,w}\hookrightarrow L^p_\alpha(0,R)\quad\mbox{for all }1\leq p<\infty.
\end{equation*}
Moreover, the embedding is compact if $\alpha\geq0$.
\end{flushleft}
\end{cor}

\subsection{General strategy of the proofs}
Our strategy for proving the main results consists of transporting the problem from the space of $k$-admissible functions $\Phi^k_{0,\mathrm{rad}}(B,w)$ to the weighted Sobolev spaces $X^{1,k+1}_{1,w}$. In this new framework, we establish the desired results (see Propositions~\ref{prop-thm1}, \ref{prop-thm2}, and \ref{extremalX}) by employing a theorem from Leckband \cite{zbMATH03867635}, along with certain radial lemmas and the Sobolev embedding for the space $X^{1,k+1}_{R,w}$. Specifically, Leckband's theorem is stated in Theorem~\ref{Leckband}, the radial lemmas in Corollary~\ref{corollary-radial}, and the Sobolev embeddings in Corollary~\ref{thm0}. After obtaining these intermediate results, we return to the space $\Phi^k_{0,\mathrm{rad}}(B,w)$.

The main challenge lies in the process of recovering the results from $X^{1,k+1}_{1,w}$ back to $\Phi^k_{0,\mathrm{rad}}(B,w)$. If $u\in \Phi^k_{0,\mathrm{rad}}(B,w)$, then it is immediate that $ v = u(|x|)$  belongs to $X^{1,k+1}_{1,w}$. However, the converse is not true in general. In fact, the reverse process involves regularity results, which were addressed by using an appropriate regularizing sequence to establish the optimality of the constants $\gamma=\gamma_{n,\beta}$, $\alpha=\alpha_{n,\beta}$ and $a=n$ (see Propositions \ref{prop-thm1} and \ref{prop-thm2} and Theorems \ref{thm1} and \ref{thm2}), as well as regularity estimates to recover the extremal problem (see Proposition~\ref{extremalX} and Theorem~\ref{thm-extremal}).

\subsection{Outline of the paper} 
The rest of the paper is divided as follows. Section \ref{section2} contains the demonstration of the Hardy inequalities, proving Theorems~\ref{theo21}, \ref{theo21e}, and Corollary~\ref{thm0}. In Section~\ref{section3}, we reformulate our problem by transferring it from the space of $k$-admissible functions $\Phi^k_{0,\mathrm{rad}}(B,w)$ to the weighted Sobolev spaces $X^{1,k+1}_{1,w}$ in which we prove the transported Trudinger-Moser inequalities in Propositions~\ref{prop-thm1} and \ref{prop-thm2}. Section~\ref{section4} is dedicated to proving Theorems \ref{thm1}, \ref{thm1B}, and \ref{thm2}, based on the results obtained in the previous section. Section~\ref{section5} focuses on proving the existence of maximizers for the transported problem in Proposition~\ref{extremalX}. Finally, in Section~\ref{section6}, we establish regularity results for the maximizers obtained in the previous section and prove Theorem~\ref{thm-extremal}.

\section{Hardy-type Inequalities}\label{section2}
Our main objective in this section is to provide a comprehensive analysis of Hardy inequalities for the weight functions $w$, considering both cases $w=w_0$ and $w=w_1$. In particular, we prove Theorems~\ref{theo21} and \ref{theo21e} using the results established in this section. The key approach involves applying \cite[Theorem 6.2 and 6.3]{zbMATH00046945} and developing an asymptotic behavior of the involved weights to precisely determine when the Hardy-type inequality holds. Before proceeding, let us first show that Corollary~\ref{thm0} is a direct consequence of Theorems~\ref{theo21} and \ref{theo21e}.
\begin{proof}[Proof of Corollary \ref{thm0}]
The continuous embeddings follow directly from items (iv), (v), (vi), and (vii) of Theorem \ref{theo21} for the case $w=w_0$, and from items (ii), (vi), and (vii) of Theorem \ref{theo21e} for the case $w=w_1$. 

To conclude the proof, we now demonstrate the compact embedding. Let $(v_m)$ be a bounded sequence in $X^{1,k+1}_{R,w}$. Since the embedding $W^{1,k+1}(0,R)\hookrightarrow L^1(0,R)$ is compact, it follows that, up to a subsequence, $(v_m)$ is Cauchy in $L^1(0,R)$. By applying an interpolation inequality, and noting that for any $p<\widetilde p<k^*$, there exists $\eta\in(0,1]$ such that
\begin{equation*}
\|v_m-v_{\widetilde m}\|_{L^p_{\alpha}}\leq\|r^{\frac{\alpha}{p}}(v_m-v_{\widetilde m})\|^\eta_{L^1}\|r^{\frac{\alpha}{p}}(v_m-v_{\widetilde m})\|^{1-\eta}_{L^{\widetilde p}}.
\end{equation*}
Dividing by $R^{\frac{\alpha}{p}}$ and using the fact that $\left(\frac{r}{R}\right)^{\frac{\alpha}{p}}\leq\left(\frac{r}{R}\right)^{\frac{\alpha}{\widetilde p}}\leq1$ for all $r\in(0,R)$, we obtain
\begin{align*}
\dfrac{\|v_m-v_{\widetilde m}\|_{L^p}}{R^{\frac{\alpha}{p}}}&\leq\|v_m-v_{\widetilde m}\|^\eta_{L^1}\|r^{\frac{\alpha}{\widetilde p}}(v_m-v_{\widetilde m})\|^{1-\eta}_{L^{\widetilde p}}.
\end{align*}
Here, we used the continuous embedding $X^{1,k+1}_{R,w}\hookrightarrow L^{\widetilde p}_\alpha(0,R)$. Therefore, $(v_m)$ is Cauchy in $L^p_\alpha(0,R)$, and consequently, it converges.
\end{proof}

\subsection{Case logarithmic with 1}

Remember that $AC_R(0,R)$ is the space of all locally absolutely continuous functions $v\colon(0,R)\to\mathbb R$ satisfying
\begin{equation*}
\lim_{r\to R}v(r)=0.
\end{equation*}

Let us assume $v\in AC_R(0,R)$, where $R\in(0,\infty)$. The main goal of this subsection is to check when
\begin{equation}\label{eqOK}
    \left(\int_0^R|v|^p\left(\ln\dfrac{R}{t}\right)^\theta t^\alpha \,\mathrm{d}t\right)^{\frac1p}\leq C\left(\int_0^R|v^{\prime}|^q\left(\ln\frac{R}{t}\right)^\mu t^\nu \,\mathrm{d}t\right)^{\frac1q}
\end{equation}
holds for some $C=C(\alpha,\theta,\nu,\mu,R,p,q)>0$, where $\alpha,\theta,\nu,\mu\in\mathbb R$ and $p,q\in[1,\infty)$.

\begin{prop}\label{hardy1}
Assume $v\in AC_{R}(0,R)$, $q=1$, and $1\leq p<\infty$. Given $\alpha,\theta,\nu,\mu\in\mathbb R$, the inequality \eqref{eqOK} holds if and only if one of the following conditions is fulfilled:
\begin{flushleft}
    $\mathrm{(i)}$ $\alpha=-1$, $\theta<-1$, $\nu=0$, and $\mu=\frac{\theta+1}p$;\\
    $\mathrm{(ii)}$ $\alpha\geq-1$, $\theta<-1$, $\nu<0$, and $\mu\leq\frac{\theta+1}p$;\\
    $\mathrm{(iii)}$ $\alpha>-1$, $\theta<-1$, $0\leq\nu\leq\frac{\alpha+1}p$, and $\frac{\theta}{p}\leq\mu\leq\frac{\theta+1}p$;\\
    $\mathrm{(iv)}$ $\alpha>-1$, $\theta=-1$, $0\leq\nu\leq\frac{\alpha+1}p$, and $-\frac{1}{p}\leq\mu<0$;\\
    $\mathrm{(v)}$ $\alpha>-1$, $\theta=-1$, $\nu<0$, and $\mu<0$;\\
    $\mathrm{(vi)}$ $\alpha>-1$, $-1<\theta\leq0$, $0\leq\nu\leq\frac{\alpha+1}p$, and $\frac{\theta}p\leq\mu\leq0$;\\
    $\mathrm{(vii)}$ $\alpha>-1$, $-1<\theta\leq0$, $\nu<0$, and $\mu\leq0$.
\end{flushleft}
\end{prop}
\begin{proof}
    By \cite[Theorem 6.2]{zbMATH00046945}, \eqref{eqOK} holds if and only if
\begin{equation}\label{eq0011}
\sup_{x\in(0,R)}\|f^{\frac1p}\|_{L^p(0,x)}\|g^{-1}\|_{L^{\infty}(x,R)}<\infty,
\end{equation}
where $f(t)=(\ln\frac{R}{t})^\theta t^\alpha$ and $g(t)=(\ln\frac{R}{t})^\mu t^\nu$. We emphasize that $g^{-1}$ denotes the function $1/g$, not the inverse of $g$. Making the change of variables $r=\ln\frac{R}{t}$, we obtain
\begin{equation*}
\int_0^x\left(\ln\frac{R}{t}\right)^\theta t^\alpha \,\mathrm{d}t=R^{\alpha+1}\int_{\ln\frac{R}{x}}^{\infty}r^\theta \mathrm{e}^{-(\alpha+1)r}\,\mathrm{d} r.
\end{equation*}
Thus, applying the change of variables $s=(\alpha+1)r$ in the case of $\alpha>-1$,
\begin{equation}\label{eqr1}
\int_0^x\left(\ln\frac{R}{t}\right)^\theta t^\alpha \,\mathrm{d}t=\left\{\begin{array}{lll}
     -\frac{1}{\theta+1}\left(\ln\frac{R}{x}\right)^{\theta+1}&\mbox{if }\alpha=-1\mbox{ and }\theta<-1 \\
     \frac{R^{\alpha+1}}{(\alpha+1)^{\theta+1}}\displaystyle\int_{(\alpha+1)\ln\frac{R}{x}}^\infty s^{\theta}\mathrm{e}^{-s}\,\mathrm{d}s&\mbox{if }\alpha>-1\\
     \infty&\mbox{otherwise}.
\end{array}\right.
\end{equation}
To analyze the asymptotic behavior of the previous integral as $x\to0$ and $x\to R$, let us consider the upper incomplete gamma function $\Gamma\colon\mathbb R\times(0,\infty)\to\mathbb R$ defined by
\begin{equation*}
    \Gamma(\eta,y)=\int_y^\infty t^{\eta-1}\mathrm{e}^{-t}\,\mathrm{d}t.
\end{equation*}
The asymptotic behavior of $\Gamma(\eta,\cdot)$ is known to be 
    \begin{equation*}
        \Gamma(\eta,y)=O(y^{\eta-1}\mathrm{e}^{-y})\mbox{ as }y\to\infty
    \end{equation*}
    and
\begin{equation*}
\Gamma(\eta,y)=\left\{\begin{array}{lll}
     O(y^\eta)&\mbox{if }\eta<0  \\
     O(-\ln y)&\mbox{if }\eta=0 \\
     O(1)&\mbox{if }\eta>0.
\end{array}\right.\mbox{ as }y\to0.
\end{equation*} 
Using this information, we can derive the asymptotic behavior of \eqref{eqr1} in different scenarios. As $x\to R$:
\begin{equation}\label{eqr2}
\int_0^x\left(\ln\frac{R}{t}\right)^\theta t^\alpha \,\mathrm{d}t=\left\{\begin{array}{llll}
     O\left[\left(\ln\frac{R}{x}\right)^{\theta+1}\right]&\mbox{if }\alpha\geq-1\mbox{ and }\theta<-1  \\
     O\left[-\ln\left((\alpha+1)\ln\frac{R}{x}\right)\right]&\mbox{if }\alpha>-1\mbox{ and }\theta=-1\\
     O(1)&\mbox{if }\alpha>-1\mbox{ and }\theta>-1 \\
\infty&\mbox{otherwise}.
\end{array}\right.
\end{equation}
As $x\to 0$:
\begin{equation}\label{eqr3}
\int_0^x\left(\ln\frac{R}{t}\right)^\theta t^\alpha \,\mathrm{d}t=\left\{\begin{array}{lll}
     O\left[\left(\ln\frac{R}{x}\right)^{\theta+1}\right]&\mbox{if }\alpha=-1\mbox{ and }\theta<-1  \\
     O\left[\left(\ln\frac{R}{x}\right)^\theta x^{\alpha+1}\right]&\mbox{if }\alpha>-1 \\
\infty&\mbox{otherwise}.
\end{array}\right.
\end{equation}

Given $x\in(0,R)$ and $\nu,\mu\in\mathbb R$, we aim to determine the behavior of $\|g^{-1}\|_{L^{\infty}(x,R)}$ as $x\to0$ and as $x\to R$. Note that $\|g^{-1}\|_{L^{\infty}(x,R)}=\infty$ if $\mu>0$. Therefore, we only need to consider the case $\mu\leq0$. Consider the case when $\nu\geq0$. Here, we have that
\begin{equation*}
    g^{-1}(t)=\left(\ln\frac{R}{t}\right)^{-\mu}t^{-\nu}
\end{equation*}
is a non-increasing function. Thus, for $\nu\geq0$,
\begin{equation*}
\|g^{-1}\|_{L^{\infty}(x,R)}=g^{-1}(x).
\end{equation*}
Next, the case $\nu<0$ and $\mu=0$ is straightforward since:
\begin{equation*}
\|g^{-1}\|_{L^{\infty}(x,R)}=g^{-1}(R).
\end{equation*}
Hence, we can now focus on the situation where $\nu,\mu<0$. The derivative of $g^{-1}$ is given by
\begin{equation*}
    \left(g^{-1}\right)^{\prime}(t)=\left(\ln\frac{R}{t}\right)^{-\mu}t^{-\nu-1}\left(\dfrac{\mu}{\ln\frac{R}{t}}-\nu\right).
\end{equation*}
Setting the derivative to zero, we find that $t_0=R\mathrm{e}^{-\frac{\mu}{\nu}}\in(0,R)$ is the only critical point of $g^{-1}$ and it is a global maximum. Thus, assuming $\nu,\mu<0$, we have
\begin{equation*}
\|g^{-1}\|_{L^{\infty}(x,R)}=\left\{\begin{array}{ll}
     g^{-1}(x)&\mbox{if }x>R\mathrm{e}^{-\frac{\mu}{\nu}}  \\
     g^{-1}(R\mathrm{e}^{-\frac{\mu}{\nu}})&\mbox{if }x\leq R\mathrm{e}^{-\frac{\mu}{\nu}}. 
\end{array}\right.
\end{equation*}
Combining all the cases, we conclude:
\begin{equation}\label{eqr4}
\|g^{-1}\|_{L^{\infty}(x,R)}=\left\{\begin{array}{llll}
     \infty&\mbox{if }\mu>0  \\
     g^{-1}(x)&\mbox{if }\mu\leq0\mbox{ and }\nu\geq0\\
     g^{-1}(R)&\mbox{if }\mu=0\mbox{ and }\nu<0\\
     g^{-1}(x)&\mbox{if }\mu,\nu<0\mbox{ and }x>R\mathrm{e}^{-\frac{\mu}{\nu}}\\
     g^{-1}(R\mathrm{e}^{-\frac{\mu}{\nu}})&\mbox{if }\mu,\nu<0\mbox{ and }x\leq R\mathrm{e}^{-\frac{\mu}{\nu}}.
\end{array}\right.
\end{equation}

To study \eqref{eq0011} using \eqref{eqr2}, \eqref{eqr3}, and \eqref{eqr4}, we need to analyse the following twelve cases:
\begin{center}
\begin{tabular}{ |c|c|c|c| } 
 \hline
  & $\mu\leq0,\nu\geq0$ & $\mu=0,\nu<0$ & $\mu<0,\nu<0$\\
  \hline
 $\alpha=-1,\theta<-1$ & Case 1 & Case 2 & Case 3\\ 
 \hline
 $\alpha>-1,\theta<-1$ & Case 4 & Case 5 & Case 6\\ 
 \hline
 $\alpha>-1,\theta=-1$ & Case 7 & Case 8 & Case 9\\ 
 \hline
 $\alpha>-1,\theta>-1$ & Case 10 & Case 11 & Case 12\\ 
 \hline
\end{tabular}
\end{center}
We need to check the finiteness of \eqref{eq0011} in each one of these cases. To evaluate the supremum, we note that
\begin{equation*}
\eqref{eq0011}\mbox{ holds}\Leftrightarrow\left\{\begin{array}{l}
     \displaystyle\lim_{x\to R}\left(\displaystyle\int_0^x\left(\ln\frac{R}{t}\right)^\theta t^\alpha\,\mathrm{d}t\right)^{\frac{1}{p}}\|g^{-1}\|_{L^\infty(x,R)}<\infty  \\
     \displaystyle\lim_{x\to 0}\left(\displaystyle\int_0^x\left(\ln\frac{R}{t}\right)^\theta t^\alpha\,\mathrm{d}t\right)^{\frac{1}{p}}\|g^{-1}\|_{L^\infty(x,R)}<\infty.
\end{array}\right.
\end{equation*}
Summarizing the conclusions for all cases, we obtain:
\begin{center}
\begin{tabular}{|c|c|}
\hline
     Case 1& \eqref{eq0011} holds $\Leftrightarrow\alpha=-1$, $\theta<-1$, $\nu=0$, $\mu=\frac{\theta+1}q$ \\
     \hline
     Case 2& \eqref{eq0011} does not hold\\
     \hline
     Case 3& \eqref{eq0011} holds $\Leftrightarrow\alpha=-1$, $\theta<-1$, $\nu<0$, $\mu\leq\frac{\theta+1}p$ \\
     \hline
     Case 4& \eqref{eq0011} holds $\Leftrightarrow\alpha>-1$, $\theta<-1$, $0\leq\nu\leq\frac{\alpha+1}p$, $\frac{\theta}p\leq\mu\leq\frac{\theta+1}p$ \\
     \hline
     Case 5& \eqref{eq0011} does not hold\\
     \hline
     Case 6& \eqref{eq0011} holds $\Leftrightarrow\alpha>-1$, $\theta<-1$, $\nu<0$, $\mu\leq\frac{\theta+1}{p}$ \\
     \hline
     Case 7& \eqref{eq0011} holds $\Leftrightarrow\alpha>-1$, $\theta=-1$, $0\leq\nu\leq\frac{\alpha+1}p$, $-\frac{1}{p}\leq\mu<0$ \\
     \hline
     Case 8& \eqref{eq0011} does not hold\\
     \hline
     Case 9& \eqref{eq0011} holds $\Leftrightarrow\alpha>-1$, $\theta=-1$, $\nu<0$, $\mu<0$ \\
     \hline
     Case 10& \eqref{eq0011} holds $\Leftrightarrow\alpha>-1$, $-1<\theta\leq0$, $0\leq\nu\leq\frac{\alpha+1}p$, $\frac{\theta}{p}\leq\mu\leq0$ \\
     \hline
     Case 11& \eqref{eq0011} holds $\Leftrightarrow\alpha>-1$, $-1<\theta\leq0$, $\nu<0$, $\mu=0$ \\
     \hline
     Case 12& \eqref{eq0011} holds $\Leftrightarrow\alpha>-1$, $-1<\theta\leq0$, $\nu<0$, $\mu<0$ \\
     \hline
\end{tabular}
\end{center}
\bigskip 

Note the implications of each case:
\begin{itemize}
    \item Case 1 infers item (i);
    \item Case 3 and 6 infer item (ii);
    \item Case 4 infers item (iii);
    \item Case 7 infers item (iv);
    \item Case 9 infers item (v);
    \item Case 10 infers item (vi);
    \item Cases 11 and 12 infer item (vii).
\end{itemize}
This concludes the proof.
\end{proof}
\begin{prop}\label{propOK2}
Assume $v\in AC_{R}(0,R)$ and $1<q\leq p<\infty$. Given $\alpha,\theta,\nu,\mu\in\mathbb R$, the inequality \eqref{eqOK} holds if and only if $\mu<q-1$ and one of the following conditions is fulfilled:
\begin{flushleft}
    $\mathrm{(i)}$ $\alpha=-1$, $\theta<-1$, $\nu<q-1$, and $p\geq-\frac{q(\theta+1)}{q-1-\mu}$;\\
    $\mathrm{(ii)}$ $\alpha=-1$, $\theta\leq\mu-q$, $\nu=q-1$, and $p=-\frac{q(\theta+1)}{q-1-\mu}$;\\
    $\mathrm{(iii)}$ $\alpha>-1$, $\theta<-1$, $\nu\leq q-1$, and $p\geq-\frac{q(\theta+1)}{q-1-\mu}$;\\
    $\mathrm{(iv)}$ $\alpha>(\nu-q+1)\frac{p}{q}-1$, $\theta<-1$, $\nu>q-1$, and $p\geq-\frac{q(\theta+1)}{q-1-\mu}$;\\
    $\mathrm{(v)}$ $\alpha=(\nu-q+1)\frac{p}{q}-1$, $\theta<-1$, $\nu>q-1$, $p\geq-\frac{q(\theta+1)}{q-1-\mu}$, and $\frac{\theta}{p}\leq\frac{\mu}{q}$;\\
    $\mathrm{(vi)}$ $\alpha>-1$, $\theta\geq-1$, and $\nu\leq q-1$;\\
    $\mathrm{(vii)}$ $\alpha>(\nu-q+1)\frac{p}{q}-1$, $\theta\geq-1$, and $\nu>q-1$;\\
    $\mathrm{(viii)}$ $\alpha=(\nu-q+1)\frac{p}{q}-1$, $\theta\geq-1$, $\nu>q-1$, and $\frac{\theta}{p}\leq\frac{\mu}{q}$.
\end{flushleft}
\end{prop}
\begin{proof}
    By \cite[Theorem 6.2]{zbMATH00046945}, \eqref{eqOK} holds if and only if
\begin{equation}\label{eq0012}
\sup_{x\in(0,R)}\|f^{\frac1p}\|_{L^p(0,x)}\|g^{-\frac{1}{q}}\|_{L^{\frac{q}{q-1}}(x,R)}<\infty,
\end{equation}
where $f(t)=(\ln\frac{R}{t})^\theta t^\alpha$ and $g(t)=(\ln\frac{R}{t})^\mu t^\nu$. In the proof of Proposition \ref{hardy1}, we obtained the asymptotic behavior of $\|f^{\frac{1}{p}}\|_{L^p(0,x)}$ in \eqref{eqr2} and \eqref{eqr3}. Thus, now we will study the asymptotic behavior of $\|g^{-\frac{1}q}\|_{L^{\frac{q}{q-1}}(x,R)}^{\frac{q}{q-1}}$. Making the change of variables $r=\ln\frac{R}{t}$, we obtain
\begin{equation*}
\int_x^R\left(\ln\frac{R}{t}\right)^{-\frac{\mu}{q-1}} t^{-\frac{\nu}{q-1}}\,\mathrm{d}t=R^{\frac{q-1-\nu}{q-1}}\int_{0}^{\ln\frac{R}{x}}r^{-\frac{\mu}{q-1}} \mathrm{e}^{\frac{\nu-q+1}{q-1}r}\,\mathrm{d} r.
\end{equation*}
Note that, for any $\mu\geq q-1$ and $x\in(0,R)$,
\begin{equation}\label{eqr5}
\int_x^R\left(\ln\frac{R}{t}\right)^{-\frac{\mu}{q-1}} t^{-\frac{\nu}{q-1}}\,\mathrm{d}t\geq R^{\frac{q-1-\nu}{q-1}}\inf_{r\in(0,\ln\frac{R}{x})}\left(\mathrm{e}^{\frac{\nu-q+1}{q-1}r}\right)\int_{0}^{\ln\frac{R}{x}}r^{-\frac{\mu}{q-1}} \,\mathrm{d} r=\infty.
\end{equation}
Then, applying the change of variables $s=\frac{\nu-q+1}{q-1}r$ in the case $\nu>q-1$ and $s=\frac{q-1-\nu}{q-1}r$ in the case $\nu<q-1$, we have, for any $\mu<q-1$,
\begin{equation*}
\int_x^R\left(\ln\frac{R}{t}\right)^{-\frac{\mu}{q-1}}t^{-\frac{\nu}{q-1}}\mathrm{d}t\!=\!\left\{\begin{array}{lll}
     R^{\frac{q-1-\nu}{q-1}}\left(\frac{q-1}{q-1-\nu}\right)^{\frac{q-1-\mu}{q-1}}\displaystyle\int_0^{\frac{q-1-\nu}{q-1}\ln\frac{R}{x}}s^{-\frac{\mu}{q-1}}\mathrm{e}^{-s}\mathrm ds&\mbox{if }\nu<q-1  \\
          \frac{q-1}{q-1-\mu}\left(\ln\frac{R}{x}\right)^{\frac{q-1-\mu}{q-1}}&\mbox{if }\nu=q-1\\
     R^{-\frac{\nu-q+1}{q-1}}\left(\frac{q-1}{\nu-q+1}\right)^{\frac{q-1-\mu}{q-1}}\displaystyle\int_0^{\frac{\nu-q+1}{q-1}\ln\frac{R}{x}}s^{-\frac{\mu}{q-1}}\mathrm{e}^s\mathrm ds&\mbox{if }\nu>q-1.
\end{array}\right.
\end{equation*}
To analyze the asymptotic behavior of the previous integral as $x\to0$ and $x\to R$, let us consider the lower incomplete gamma function $\gamma\colon(0,\infty)\times(0,\infty)\to\mathbb R$ and the function $h\colon(0,\infty)\times(0,\infty)\to\mathbb R$ defined by
\begin{equation*}
\gamma(\eta,y)=\int_0^yt^{\eta-1}\mathrm{e}^{-t}\,\mathrm{d}t\mbox{ and }h(\eta,y)=\int_0^yt^{\eta-1} \mathrm{e}^t\,\mathrm{d}t.
\end{equation*}
The asymptotic behavior of $\gamma(\eta,\cdot)$ and $h(\eta,\cdot)$ is given by the following. As $y\to\infty$:
\begin{equation*}
    \gamma(\eta,y)=O(1)\mbox{ and }h(\eta,y)=O(y^{\eta-1} \mathrm{e}^y).
\end{equation*}
As $y\to0$:
\begin{equation*}
    \gamma(\eta,y)=O(y^{\eta})\mbox{ and }h(\eta,y)=O(y^{\eta}).
\end{equation*}
Applying those definitions, note that, for any $\mu<q-1$,
\begin{equation*}
\int_x^R\left(\ln\frac{R}{t}\right)^{-\frac{\mu}{q-1}}t^{-\frac{\nu}{q-1}}\mathrm{d}t=\!\left\{\begin{array}{lll}
     R^{\frac{q-1-\nu}{q-1}}\left(\frac{q-1}{q-1-\nu}\right)^{\frac{q-1-\mu}{q-1}}\!\gamma\left(\frac{q-1-\mu}{q-1},\frac{q-1-\nu}{q-1}\ln\frac{R}{x}\right)&\mbox{if }\nu<q-1  \\
          \frac{q-1}{q-1-\mu}\left(\ln\frac{R}{x}\right)^{\frac{q-1-\mu}{q-1}}&\mbox{if }\nu=q-1\\
     R^{-\frac{\nu-q+1}{q-1}}\left(\frac{q-1}{\nu-q+1}\right)^{\frac{q-1-\mu}{q-1}}\!h\left(\frac{q-1-\mu}{q-1},\frac{\nu-q+1}{q-1}\ln\frac{R}{x}\right)&\mbox{if }\nu>q-1.
\end{array}\right.
\end{equation*}
Applying those and \eqref{eqr5}, we obtain the asymptotic behavior of $\|g^{-\frac1q}\|_{L^{\frac{q}{q-1}}(x,R)}^{\frac{q}{q-1}}$ given as follows. As $x\to R$:
\begin{equation}\label{eqr6}
\int_x^R\left(\ln\frac{R}{t}\right)^{-\frac{\mu}{q-1}}t^{-\frac{\nu}{q-1}}\,\mathrm{d}t=\left\{\begin{array}{ll}
     O\left[\left(\ln\frac{R}{x}\right)^{\frac{q-1-\mu}{q-1}}\right]&\mbox{if }\mu<q-1  \\
     \infty&\mbox{if }\mu\geq q-1. 
\end{array}\right.
\end{equation}
As $x\to0$:
\begin{equation}\label{eqr7}
\int_x^R\left(\ln\frac{R}{t}\right)^{-\frac{\mu}{q-1}}t^{-\frac{\nu}{q-1}}\,\mathrm{d}t=\left\{\begin{array}{ll}
     O(1)&\mbox{if }\mu<q-1,\nu<q-1\\
     O\left[\left(\ln\frac{R}{x}\right)^{\frac{q-1-\mu}{q-1}}\right]&\mbox{if }\mu<q-1,\nu=q-1\\
     O\left[\left(\ln\frac{R}{x}\right)^{-\frac{\mu}{q-1}}x^{\frac{q-1-\nu}{q-1}}\right]&\mbox{if }\mu<q-1,\nu>q-1\\
     \infty&\mbox{if }\mu\geq q-1. 
\end{array}\right.
\end{equation}

We already know that \eqref{eq0012} is infinite for $\mu \geq q-1$. To examine the finiteness of \eqref{eq0012} using \eqref{eqr2}, \eqref{eqr3}, \eqref{eqr6}, and \eqref{eqr7}, we need to analyze the following twelve cases (assuming $\mu<q-1$ in all scenarios):

\bigskip

\begin{center}
\begin{tabular}{ |c|c|c|c| } 
 \hline
  & $\nu<q-1$ & $\nu=q-1$ & $\nu>q-1$\\
  \hline
 $\alpha=-1,\theta<-1$ & Case 1 & Case 2 & Case 3\\ 
 \hline
 $\alpha>-1,\theta<-1$ & Case 4 & Case 5 & Case 6\\ 
 \hline
 $\alpha>-1,\theta=-1$ & Case 7 & Case 8 & Case 9\\ 
 \hline
 $\alpha>-1,\theta>-1$ & Case 10 & Case 11 & Case 12\\ 
 \hline
\end{tabular}
\end{center}

\bigskip 

We need to check the finiteness of \eqref{eq0012} in each one of these cases. To evaluate the supremum, we note that
\begin{equation*}
\eqref{eq0012}\mbox{ holds}\Leftrightarrow\left\{\begin{array}{l}
     \displaystyle\lim_{x\to R}\left(\displaystyle\int_0^x\left(\ln\frac{R}{t}\right)^\theta t^\alpha \,\mathrm{d}t\right)^{\frac{1}{p}}\left(\int_x^R\left(\ln\frac{R}{t}\right)^{-\frac{\mu}{q-1}}t^{-\frac{\nu}{q-1}}\,\mathrm{d}t\right)^{\frac{q-1}{q}}<\infty  \\
     \displaystyle\lim_{x\to 0}\left(\displaystyle\int_0^x\left(\ln\frac{R}{t}\right)^\theta t^\alpha \,\mathrm{d}t\right)^{\frac{1}{p}}\left(\int_x^R\left(\ln\frac{R}{t}\right)^{-\frac{\mu}{q-1}}t^{-\frac{\nu}{q-1}}\,\mathrm{d}t\right)^{\frac{q-1}{q}}<\infty.
\end{array}\right.
\end{equation*}
Summarizing the conclusions for all cases, assuming $\mu<q-1$ in each instance, we obtain:

\bigskip

\begin{center}
\begin{tabular}{|c|c|}
\hline
     Case 1& \eqref{eq0012} holds $\Leftrightarrow\alpha=-1$, $\theta<-1$, $\nu<q-1$, $p\geq-\frac{q(\theta+1)}{q-1-\mu}$ \\
     \hline
     Case 2& \eqref{eq0012} holds $\Leftrightarrow \alpha=-1$, $\theta\leq\mu-q$, $\nu=q-1$, $p=-\frac{q(\theta+1)}{q-1-\mu}$ \\
     \hline
     Case 3& \eqref{eq0012} does not hold \\
     \hline
     Case 4& \eqref{eq0012} holds $\Leftrightarrow\alpha>-1$, $\theta<-1$, $\nu<q-1$, $p\geq-\frac{q(\theta+1)}{q-1-\mu}$ \\
     \hline
     Case 5& \eqref{eq0012} holds $\Leftrightarrow\alpha>-1$, $\theta<-1$, $\nu=q-1$, $p\geq-\frac{q(\theta+1)}{q-1-\mu}$\\
     \hline
     Case 6& \eqref{eq0012} holds $\Leftrightarrow\theta<-1$, $\nu>q-1$, $p\geq-\frac{q(\theta+1)}{q-1-\mu}$, and\\ &($\alpha>(\nu-q+1)\frac{p}{q}-1$ or ($\alpha=(\nu-q+1)\frac{p}{q}-1$ and $\frac{\theta}{p}\leq\frac{\mu}q$)) \\
     \hline
     Case 7& \eqref{eq0012} holds $\Leftrightarrow\alpha>-1$, $\theta=-1$, $\nu<q-1$\\
     \hline
     Case 8& \eqref{eq0012} holds $\Leftrightarrow\alpha>-1$, $\theta=-1$, $\nu=q-1$\\
     \hline
     Case 9& \eqref{eq0012} holds $\Leftrightarrow\theta=-1$, $\nu>q-1$, and\\ &($\alpha>(\nu-q+1)\frac{p}{q}-1$ or ($\alpha=(\nu-q+1)\frac{p}{q}-1$ and $\frac{\theta}{p}\leq\frac{\mu}q$)) \\
     \hline
     Case 10& \eqref{eq0012} holds $\Leftrightarrow\alpha>-1$, $\theta>-1$, $\nu<q-1$ \\
     \hline
     Case 11& \eqref{eq0012} holds $\Leftrightarrow\alpha>-1$, $\theta>-1$, $\nu=q-1$\\
     \hline
     Case 12& \eqref{eq0012} holds $\Leftrightarrow\theta>-1$, $\nu>q-1$, and\\ &($\alpha>(\nu-q+1)\frac{p}{q}-1$ or ($\alpha=(\nu-q+1)\frac{p}{q}-1$ and $\frac{\theta}{p}\leq\frac{\mu}q$)) \\
     \hline
\end{tabular}
\end{center}

\bigskip

Note the implications of each case:
\begin{itemize}
    \item Case 1 infers item (i);
    \item Case 2 infers item (ii);
    \item Case 4 and 5 infer item (iii);
    \item Case 6 infer items (iv) and (v);
    \item Cases 7, 8, 10, and 11 infer item (vi);
    \item Cases 9 and 12 infer items (vii) and (viii).
\end{itemize}
This concludes the proof.
\end{proof}
\begin{prop}\label{propOK3}
Assume $v\in AC_{R}(0,R)$ and $1\leq p<q<\infty$. Given $\alpha,\theta,\nu,\mu\in\mathbb R$, the inequality \eqref{eqOK} holds if and only if $\mu<q-1$ and one of the following conditions is fulfilled:
\begin{flushleft}
    $\mathrm{(i)}$ $\alpha\geq-1$, $\theta<-1$, $\nu<q-1$, and $p>-\frac{q(\theta+1)}{q-1-\mu}$;\\
    $\mathrm{(ii)}$ $\alpha>-1$, $\theta<-1$, $\nu=q-1$, and $p>-\frac{q(\theta+1)}{q-1-\mu}$;\\
    $\mathrm{(iii)}$ $\alpha>-1$, $\nu>q-1$, and $-\frac{q(\theta+1)}{q-1-\mu}<p<\frac{(\alpha+1)q}{\nu-q+1}$;\\
    $\mathrm{(iv)}$ $\alpha>-1$, $\nu>q-1$, and $-\frac{q(\theta+1)}{q-1-\mu}<p=\frac{(\alpha+1)q}{\nu-q+1}$, and $p(\mu+1)>q(\theta+1)$;\\
    $\mathrm{(v)}$ $\alpha>-1$, $\theta\geq-1$, and $\nu\leq q-1$.
\end{flushleft}
\end{prop}
\begin{proof}
By \cite[Theorem 6.3]{zbMATH00046945}, \eqref{eqOK} holds if and only if the following integral is finite
\begin{equation}\label{eq311}
\int_0^R\left(\int_0^x\left(\ln\frac{R}{t}\right)^\theta t^\alpha \,\mathrm{d}t\right)^{\frac{q}{q-p}}\left(\int_x^R\left(\ln\frac{R}{t}\right)^{-\frac{\mu}{q-1}}t^{-\frac{\nu}{q-1}}\,\mathrm{d}t\right)^{q\frac{p-1}{q-p}}\left(\ln\frac{R}{x}\right)^{-\frac{\mu}{q-1}}x^{-\frac{\nu}{q-1}}\, \mathrm{d} x,
\end{equation}
Making the change of variables $y=\ln\frac{R}{x}$, we note that, given $a,b\in\mathbb R$,
\begin{equation*}
\int_0^{\frac{R}2}\left(\ln\frac{R}{x}\right)^ax^b\, \mathrm{d} x=R^{b+1}\int_{\ln2}^\infty y^a\mathrm{e}^{-(b+1)y}\,\mathrm{d}y\left\{\begin{array}{ll}
    =\infty&\mbox{if }b<-1\\
     <\infty&\mbox{if }b>-1\\
     \left\{\begin{array}{ll}
        <\infty&\mbox{if }a<-1  \\
          =\infty&\mbox{if }a\geq-1 
     \end{array}\right.&\mbox{if }b=-1.
\end{array}\right.
\end{equation*}
Then
\begin{equation}\label{eq692}
\int_0^{\frac{R}2}\left(\ln\frac{R}{x}\right)^ax^b\, \mathrm{d} x<\infty\Leftrightarrow b>-1\mbox{ or }\left(b=-1\mbox{ and }a<-1\right).
\end{equation}
Analogously, we note that
\begin{equation*}
    \int_{\frac{R}2}^R\left(\ln\frac{R}{x}\right)^ax^b\, \mathrm{d} x=R^{b+1}\int_0^{\ln2}y^a\mathrm{e}^{-(b+1)y}\,\mathrm{d}y\left\{\begin{array}{ll}
    =\infty&\mbox{if }a\leq-1\\
     <\infty&\mbox{if }a>-1.
\end{array}\right.
\end{equation*}
Thus,
\begin{equation*}
\int_{\frac{R}2}^R\left(\ln\frac{R}{x}\right)^ax^b\, \mathrm{d} x<\infty\Leftrightarrow a>-1.
\end{equation*}

We need to analyze the following twelve cases (assuming $\mu<q-1$ in all scenarios):
\begin{center}
\begin{tabular}{ |c|c|c|c| } 
 \hline
  & $\nu<q-1$ & $\nu=q-1$ & $\nu>q-1$\\
  \hline
 $\alpha=-1,\theta<-1$ & Case 1 & Case 2 & Case 3\\
 \hline
 $\alpha>-1,\theta<-1$ & Case 4 & Case 5 & Case 6\\
 \hline
 $\alpha>-1,\theta=-1$ & Case 7 & Case 8 & Case 9\\
 \hline
 $\alpha>-1,\theta>-1$ & Case 10 & Case 11 & Case 12\\ 
 \hline
\end{tabular}
\end{center}
In each of these cases, we proceed in the following way. Dividing the integral in half we note that
\begin{equation*}
\mbox{\eqref{eq311} holds}\Leftrightarrow\int_0^{\frac{R}{2}}A(x)\, \mathrm{d} x<\infty\mbox{ and }\int_{\frac{R}{2}}^RA(x)\, \mathrm{d} x<\infty,
\end{equation*}
where $A(x)=A_{p,q,\alpha,\theta,\nu,\mu,x,R}(x)>0$ is the integrand of \eqref{eq311}. In order to determinate when $\int_0^{\frac{R}{2}}A(x)\, \mathrm{d} x<\infty$ we use \eqref{eqr3} and \eqref{eqr7} and to determinate when $\int_{\frac{R}{2}}^RA(x)\, \mathrm{d} x<\infty$ we use \eqref{eqr2} and \eqref{eqr6}. Summarizing the conclusions for all cases, assuming $\mu<q-1$ in each instance, we obtain:

\begin{center}
\begin{tabular}{|c|c|}
\hline
     Case 1& \eqref{eq311} is finite $\Leftrightarrow$ $\alpha=-1$, $\theta<-1$, $\nu<q-1$, and $p>-\frac{q(\theta+1)}{q-1-\mu}$\\
     \hline
     Case 2& \eqref{eq311} is infinite \\
     \hline
     Case 3& \eqref{eq311} is infinite \\
     \hline
     Case 4& \eqref{eq311} is finite $\Leftrightarrow$ $\alpha>-1$, $\theta<-1$, $\nu<q-1$, and $p>-\frac{q(\theta+1)}{q-1-\mu}$\\
     \hline
     Case 5& \eqref{eq311} is finite $\Leftrightarrow$ $\alpha>-1$, $\theta<-1$, $\nu=q-1$, and $p>-\frac{q(\theta+1)}{q-1-\mu}$\\
     \hline
     Case 6& \eqref{eq311} is finite $\Leftrightarrow$ $\alpha>-1$, $\theta<-1$, $\nu>q-1$, $p>-\frac{q(\theta+1)}{q-1-\mu}$, and\\ &($p<\frac{(\alpha+1)q}{\nu-q+1}$ or ($p=\frac{(\alpha+1)q}{\nu-q+1}$ and $p(\mu+1)>q(\theta+1)$))\\
     \hline
     Case 7& \eqref{eq311} is finite $\Leftrightarrow$ $\alpha>-1$, $\theta=-1$, and $\nu<q-1$\\
     \hline
     Case 8& \eqref{eq311} is finite $\Leftrightarrow$ $\alpha>-1$, $\theta=-1$, and $\nu=q-1$\\
     \hline
     Case 9& \eqref{eq311} is finite $\Leftrightarrow$ $\alpha>-1$, $\theta=-1$, $\nu>q-1$, and\\&$(p<\frac{(\alpha+1)q}{\nu-q+1}$ or ($p=\frac{(\alpha+1)q}{\nu-q+1}$ and $\mu>-1$))\\
     \hline
     Case 10& \eqref{eq311} is finite $\Leftrightarrow$ $\alpha>-1$, $\theta>-1$, and $\nu<q-1$\\
     \hline
     Case 11& \eqref{eq311} is finite $\Leftrightarrow$ $\alpha>-1$, $\theta>-1$, and $\nu=q-1$\\
     \hline
     Case 12& \eqref{eq311} is finite $\Leftrightarrow$ $\alpha>-1$, $\theta>-1$, $\nu>q-1$, and\\&($p<\frac{(\alpha+1)q}{\nu-q+1}$ or ($p=\frac{(\alpha+1)q}{\nu-q+1}$ and $p(\mu+1)>q(\theta+1)$))\\
     \hline
\end{tabular}
\end{center}

Here, we provide a detailed explanation of Case 7 as an example. Assume that $\alpha>-1$, $\theta=-1$, and $\nu<q-1$. Note that the inequality $\frac{(\alpha+1)q}{q-p}-\frac{\nu}{q-1}>-1$, together with \eqref{eq692}, implies that
\begin{equation*}
\int_0^{\frac{R}2}\left(\ln\frac{R}{x}\right)^{\frac{\theta q}{q-p}}x^{\frac{(\alpha+1)q}{q-p}}\left(\ln\frac{R}{x}\right)^{-\frac{\mu}{q-1}}x^{-\frac{\nu}{q-1}}\mathrm \, \mathrm{d} x<\infty.
\end{equation*}
By applying \eqref{eqr3} and \eqref{eqr7}, we get that $\int_0^{\frac{R}{2}}A(x)\, \mathrm{d} x<\infty$ without any additional assumptions on the constants. We are left to check when $\int_{\frac{R}2}^RA(x)\, \mathrm{d} x<\infty$. We proceed by fixing $\epsilon=\exp(\frac{1}{\mathrm{e}(\alpha+1)})-1>0$ and then making the change of variables $y=\ln\frac{R}{x}$ and $z=-\ln\left[(\alpha+1)y\right]$ to obtain
\begin{align*}
\int_{\frac{R}2}^RA(x)\, \mathrm{d} x<\infty&\Leftrightarrow\int_{\frac{R}{1+\epsilon}}^RA(x)\, \mathrm{d} x<\infty\\
&\Leftrightarrow\int_{\frac{R}{1+\epsilon}}^R\left[-\ln\left((\alpha+1)\ln\frac{R}{x}\right)\right]^{\frac{q}{q-p}}\!\left(\ln\frac{R}{x}\right)^{\frac{q-1-\mu}{q-1}q\frac{p-1}{q-p}-\frac{\mu}{q-1}}x^{-\frac{\nu}{q-1}}\, \mathrm{d} x<\infty\\
&\Leftrightarrow\int_{0}^{\frac{1}{\mathrm{e}(\alpha+1)}}\left[-\ln\left(\alpha+1)y\right)\right]^{\frac{q}{q-p}}y^{\frac{q-1-\mu}{q-1}q\frac{p-1}{q-p}-\frac{\mu}{q-1}}\mathrm{e}^{(\frac{\nu}{q-1}-1)y}\,\mathrm{d}y<\infty\\
&\Leftrightarrow\int_{0}^{\frac{1}{\mathrm{e}(\alpha+1)}}\left[-\ln\left(\alpha+1)y\right)\right]^{\frac{q}{q-p}}y^{\frac{q-1-\mu}{q-1}q\frac{p-1}{q-p}-\frac{\mu}{q-1}}\,\mathrm{d}y<\infty\\
&\Leftrightarrow\int_1^\infty z^{\frac{q}{q-p}}\mathrm{e}^{(\frac{\mu}{q-1}-\frac{q-1-\mu}{q-1}q\frac{p-1}{q-p}-1)z}\mathrm dz<\infty.
\end{align*}
The last inequality always happens because
\begin{equation*}
\frac{\mu}{q-1}-\frac{q-1-\mu}{q-1}q\frac{p-1}{q-p}-1<0.
\end{equation*}
This concludes that in Case 7 we have that \eqref{eq311} is finite if, and only if, $\alpha>-1$, $\theta=-1$, $\nu<q-1$, and $\mu<q-1$.

Note the implications of each case:
\begin{itemize}
    \item Cases 1 and 4 infer item (i);
    \item Case 5 infers item (ii);
     \item Cases 6, 9, and 12 infer items (iii) and (iv);
    \item Cases 7, 8, 10, and 11 infer item (v).
\end{itemize}
This concludes the proof.
\end{proof}

Since Propositions \ref{hardy1}, \ref{propOK2}, and \ref{propOK3} have been demonstrated, we are now prepared to prove Theorem \ref{theo21}.

\begin{proof}[Proof of Theorem \ref{theo21}]
Here, we apply the previously established propositions to the following constants:
\begin{equation*}
q=k+1,\ \theta=0,\ \mu=\frac{\beta n}2,\mbox{ and }\nu=n-k.
\end{equation*}
The proof is divided into four cases:

\smallskip

\noindent\underline{Case 1:} When $p\geq k+1$ and $k\geq\frac{n}{2}$, we apply items (vi) and (vii) of Proposition \ref{hardy1} alongside item (vi) of Proposition \ref{propOK2}. This establishes that inequality \eqref{aosfnas} holds if and only if one of the conditions in items (i), (ii), or (iv) of Theorem \ref{theo21} is satisfied.

\smallskip

\noindent\underline{Case 2:} When $p \geq k+1$ and $k < \frac{n}{2}$, we use item (vi) of Proposition \ref{hardy1} together with items (vii) and (viii) of Proposition \ref{propOK2}. This shows that inequality \eqref{aosfnas} holds if and only if one of the conditions in items (iii), (v), or (vi) of Theorem \ref{theo21} occurs.

\smallskip

\noindent\underline{Case 3:} When $p < k+1$ and $k \geq \frac{n}{2}$, we rely on item (v) of Proposition \ref{propOK3}, demonstrating that inequality \eqref{aosfnas} holds if and only if item (iv) of Theorem \ref{theo21} is satisfied.

\smallskip

\noindent\underline{Case 4:} When $p < k+1$ and $k < \frac{n}{2}$, we apply items (iii) and (iv) of Proposition \ref{propOK3}, establishing that inequality \eqref{aosfnas} holds if and only if one of the conditions in items (v) or (vii) of Theorem \ref{theo21} occurs.
\end{proof}

\subsection{Case logarithmic with \texorpdfstring{$\mathrm{e}$}{}}

Let us assume $v\in AC_R(0,R)$, where $R\in(0,\infty)$. The main goal of this subsection is to check when
\begin{equation}\label{eqOKe}
    \left(\int_0^R|v|^p\left(\ln\dfrac{\mathrm{e}R}{t}\right)^\theta t^\alpha \,\mathrm{d}t\right)^{\frac1p}\leq C\left(\int_0^R|v^{\prime}|^q\left(\ln\frac{\mathrm{e}R}{t}\right)^\mu t^\nu \,\mathrm{d}t\right)^{\frac1q}
\end{equation}
holds for some $C=C(\alpha,\theta,\nu,\mu,R,p,q)>0$, where $\alpha,\theta,\nu,\mu\in\mathbb R$ and $p,q\in[1,\infty)$.

\begin{prop}\label{hardy1e}
Assume $v\in AC_{R}(0,R)$, $q=1$, and $1\leq p<\infty$. Given $\alpha,\theta,\nu,\mu\in\mathbb R$, the inequality \eqref{eqOKe} holds if and only if one of the following conditions is fulfilled:
\begin{flushleft}
    $\mathrm{(i)}$ $\alpha=-1$, $\theta<-1$, $\frac{\theta+1}p\leq\mu<0$, and $\nu=0$;\\
    $\mathrm{(ii)}$ $\alpha=-1$, $\theta<-1$, $\mu<0$, and $\nu<0$;\\
    $\mathrm{(iii)}$ $\alpha=-1$, $\theta<-1$, $\mu\geq0$, and $\nu\leq0$;\\
    $\mathrm{(iv)}$ $\alpha=-1$, $\theta<-1$, and $\mu>0$, and $\nu>0$;\\
    $\mathrm{(v)}$ $\alpha>-1$ and $\nu\leq0$;\\
    $\mathrm{(vi)}$ $\alpha>-1$, $\mu\leq0$, and $0<\nu<\frac{\alpha+1}p$;\\
    $\mathrm{(vii)}$ $\alpha>-1$, $\theta\leq\mu p$, $\mu\leq0$, and $\nu=\frac{\alpha+1}p$;\\
    $\mathrm{(viii)}$ $\alpha>-1$, $\mu>0$, and $\nu>0$.
\end{flushleft}
\end{prop}

\begin{proof}
The proof follows the approach of Proposition \ref{hardy1}. We will only present the results derived by applying the same arguments. According to \cite[Theorem 6.2]{zbMATH00046945}, \eqref{eqOKe} holds if and only if
\begin{equation}\label{eq0011e}
\sup_{x\in(0,R)}\|f^{\frac1p}\|_{L^p(0,x)}\|g^{-1}\|_{L^{\infty}(x,R)}<\infty,
\end{equation}
where $f(t)=(\ln\frac{\mathrm{e}R}{t})^\theta t^\alpha$ and $g(t)=(\ln\frac{\mathrm{e}R}{t})^\mu t^\nu$. The following asymptotic behavior is obtained as $x\to R$.
\begin{equation}\label{eqr2e}
\int_0^x\left(\ln\frac{\mathrm{e}R}{t}\right)^\theta t^\alpha \,\mathrm{d}t=\left\{\begin{array}{lll}
     O(1)&\mbox{if }\alpha=-1\mbox{ and }\theta<-1 \\
     O(1)&\mbox{if }\alpha>-1\\
     \infty&\mbox{otherwise}.
\end{array}\right.
\end{equation}
Moreover, as $x\to0$, we have
\begin{equation}\label{eqr3e}
\int_0^x\left(\ln\frac{\mathrm{e}R}{t}\right)^\theta t^\alpha \,\mathrm{d}t=\left\{\begin{array}{lll}
     O\left[\left(\ln\frac{\mathrm{e}R}{x}\right)^{\theta+1}\right]&\mbox{if }\alpha=-1\mbox{ and }\theta<-1  \\
     O\left[\left(\ln\frac{\mathrm{e}R}{x}\right)^\theta x^{\alpha+1}\right]&\mbox{if }\alpha>-1 \\
\infty&\mbox{otherwise}.
\end{array}\right.
\end{equation}

To analyze the norm $\|g^{-1}\|_{L^\infty(x,R)}$, we obtain the following condition
\begin{equation}\label{eqr4e}
\|g^{-1}\|_{L^{\infty}(x,R)}=\left\{\begin{array}{llll}
     g^{-1}(x)&\mbox{if }\mu\leq0\mbox{ and }\nu\geq0  \\
     g^{-1}(R)&\mbox{if }\mu\geq0\mbox{ and }\nu\leq0\\
     g^{-1}(x)&\mbox{if }\mu,\nu<0\mbox{ and }x>R\mathrm{e}^{1-\frac{\mu}{\nu}}\\
     g^{-1}(R\mathrm{e}^{1-\frac{\mu}{\nu}})&\mbox{if }\mu,\nu<0\mbox{ and }x\leq R\mathrm{e}^{1-\frac{\mu}{\nu}}\\
     g^{-1}(x)&\mbox{if }\mu,\nu>0\mbox{ and }x>R\mathrm{e}^{1-\frac{\mu}{\nu}}\\
     g^{-1}(\min\{R,R\mathrm{e}^{1-\frac{\mu}{\nu}}\})&\mbox{if }\mu,\nu>0\mbox{ and }x\leq R\mathrm{e}^{1-\frac{\mu}{\nu}}.
\end{array}\right.
\end{equation}

To examine \eqref{eq0011e} using \eqref{eqr2e}, \eqref{eqr3e}, and \eqref{eqr4e}, we must consider the following six cases:
\begin{center}
\begin{tabular}{ |c|c|c|c| } 
 \hline
  & $\mu\leq0,\nu\geq0$ & $\mu\geq0,\nu\leq0$ & $\mu<0,\nu<0$ or $\mu>0,\nu>0$\\
  \hline
 $\alpha=-1,\theta<-1$ & Case 1 & Case 2 & Case 3 \\ 
 \hline
 $\alpha>-1$ & Case 4 & Case 5 & Case 6 \\ 
 \hline
\end{tabular}
\end{center}
We must verify the finiteness of \eqref{eq0011e} in each of these cases. From \eqref{eqr2e} and \eqref{eqr4e}, we observe that in each of the six cases:
\begin{equation*}
\eqref{eq0011e}\mbox{ holds}\Leftrightarrow\lim_{x\to 0}\left(\displaystyle\int_0^x\left(\ln\frac{\mathrm{e}R}{x}\right)^\theta t^\alpha \,\mathrm{d}t\right)^{\frac{1}{p}}\|g^{-1}\|_{L^\infty(x,R)}<\infty.
\end{equation*}
Summarizing the conclusions for all cases:
\begin{center}
\begin{tabular}{|c|c|}
\hline
     Case 1& \eqref{eq0011e} holds $\Leftrightarrow\alpha=-1$, $\theta<-1$, $\mu\leq 0$, $\nu=0$, and $p\mu\leq\theta+1$\\
     \hline
     Case 2& \eqref{eq0011e} holds $\Leftrightarrow\alpha=-1$, $\theta<-1$, $\mu\geq0$, $\nu\leq0$\\
     \hline
     Case 3& \eqref{eq0011e} holds $\Leftrightarrow\alpha=-1$, $\theta<-1$, and ($\mu,\nu<0$ or $\mu,\nu>0$) \\
     \hline
     Case 4& \eqref{eq0011e} holds $\Leftrightarrow\alpha>-1$, $\mu\leq0$, $\nu\geq0$,\\
     &and ($p\nu<\alpha+1$ or ($p\nu=\alpha+1$ and $\theta\leq\mu p$))\\
     \hline
     Case 5& \eqref{eq0011e} holds $\Leftrightarrow\alpha>-1$, $\mu\geq0$, $\nu\leq0$\\
     \hline
     Case 6& \eqref{eq0011e} holds $\Leftrightarrow\alpha>-1$ and ($\mu,\nu<0$ or $\mu,\nu>0$)\\
     \hline
\end{tabular}
\end{center}
The implications of each case are as follows:
\begin{itemize}
    \item Case 1 infers items (i) and (iii) (for $\mu=\nu=0)$;
    \item Case 2 infers item (iii);
    \item Case 3 infers items (ii) and (iv);
    \item Case 4 infers items (v) (for $\mu\leq 0$ and $\nu=0$), (vi), and (vii);
    \item Case 5 infers item (v) (for $\mu\geq0$);
    \item Case 6 infers items (v) (for $\mu<0$ and $\nu<0$) and (viii).
\end{itemize}
This concludes the proof.
\end{proof}

\begin{prop}\label{propOK2e}
Assume $v\in AC_{R}(0,R)$ and $1<q\leq p<\infty$. Given $\alpha,\theta,\nu,\mu\in\mathbb R$, the inequality \eqref{eqOKe} holds if and only if one of the following conditions is fulfilled:
\begin{flushleft}
    $\mathrm{(i)}$ $\alpha=-1$, $\theta<-1$, and $\nu<q-1$;\\
    $\mathrm{(ii)}$ $\alpha=-1$, $\theta<-1$, $\nu=q-1$, $\mu<q-1$, and $p\geq-\frac{q(\theta+1)}{q-1-\mu}$;\\
    $\mathrm{(iii)}$ $\alpha=-1$, $\theta<-1$, $\nu=q-1$, and $\mu\geq q-1$;\\
    $\mathrm{(iv)}$ $\alpha>-1$ and $\nu\leq q-1$;\\
    $\mathrm{(v)}$ $\alpha>(\nu-q+1)\frac{p}{q}-1$ and $\nu>q-1$;\\
    $\mathrm{(vi)}$ $\alpha=(\nu-q+1)\frac{p}{q}-1$, $\nu>q-1$, and $\frac{\theta}{p}\leq\frac{\mu}{q}$.
\end{flushleft}
\end{prop}
\begin{proof}
The proof follows the approach of Proposition \ref{propOK2}. We will only present the results derived by applying the same arguments. According to \cite[Theorem 6.2]{zbMATH00046945}, \eqref{eqOKe} holds if and only if
\begin{equation}\label{eq0012e}
\sup_{x\in(0,R)}\|f^{\frac1p}\|_{L^p(0,x)}\|g^{-\frac{1}{q}}\|_{L^{\frac{q}{q-1}}(x,R)}<\infty,
\end{equation}
where $f(t)=(\ln\frac{\mathrm{e}R}{t})^\theta t^\alpha$ and $g(t)=(\ln\frac{\mathrm{e}R}{t})^\mu t^\nu$. We already have the asymptotic behavior of $\|f^{\frac{1}{p}}\|_{L^p(0,x)}$ in \eqref{eqr2e} and \eqref{eqr3e}. Thus, now we will examine the asymptotic behavior of $\|g^{-\frac{1}q}\|_{L^{\frac{q}{q-1}}(x,R)}^{\frac{q}{q-1}}$. The following asymptotic behavior is obtained as $x\to R$.
\begin{equation}\label{eqr6e}
\int_x^R\left(\ln\frac{\mathrm{e}R}{t}\right)^{-\frac{\mu}{q-1}}t^{-\frac{\nu}{q-1}}\,\mathrm{d}t=O\left(\ln\frac{R}{x}\right).
\end{equation}
Moreover, as $x\to0$, we have
\begin{equation}\label{eqr7e}
\int_x^R\left(\ln\frac{\mathrm{e}R}{t}\right)^{-\frac{\mu}{q-1}}t^{-\frac{\nu}{q-1}}\,\mathrm{d}t=\left\{\begin{array}{ll}
     O(1)&\mbox{if }\nu<q-1  \\
     O\left[\left(\ln\frac{\mathrm{e}R}{x}\right)^{-\frac{\mu}{q-1}}x^{\frac{q-1-\nu}{q-1}}\right]&\mbox{if }\nu>q-1\\
     O\left[\left(\ln\frac{\mathrm{e}R}{x}\right)^{\frac{q-1-\mu}{q-1}}\right]&\mbox{if }\nu=q-1,\mu<q-1\\
     O\left[\ln\left(\ln\frac{\mathrm{e}R}{x}\right)\right]&\mbox{if }\nu=q-1,\mu=q-1\\
     O(1)&\mbox{if }\nu=q-1,\mu>q-1.
\end{array}\right.
\end{equation}

We need to analyze the following ten cases:
\begin{center}
\begin{tabular}{ |c|c|c|c|c|c| } 
 \hline
  & $\nu<q-1$ & $\nu>q-1$ & $\nu=q-1>\mu$ & $\nu=\mu=q-1$ & $\nu=q-1<\mu$\\
  \hline
 $\alpha=-1,\theta<-1$ & Case 1 & Case 2 & Case 3 & Case 4 & Case 5\\ 
 \hline
 $\alpha>-1$ & Case 6 & Case 7 & Case 8 & Case 9 & Case 10\\ 
 \hline
\end{tabular}
\end{center}
We need to check the finiteness of \eqref{eq0012e} in each one of these cases. To evaluate the supremum, we note that
\begin{equation*}
\eqref{eq0012e}\mbox{ holds}\Leftrightarrow\left\{\begin{array}{l}
     \displaystyle\lim_{x\to R}\left(\displaystyle\int_0^x\left(\ln\frac{\mathrm{e}R}{t}\right)^\theta t^\alpha \,\mathrm{d}t\right)^{\frac{1}{p}}\left(\int_x^R\left(\ln\frac{\mathrm{e}R}{t}\right)^{-\frac{\mu}{q-1}}t^{-\frac{\nu}{q-1}}\,\mathrm{d}t\right)^{\frac{q-1}{q}}<\infty  \\
     \displaystyle\lim_{x\to 0}\left(\displaystyle\int_0^x\left(\ln\frac{\mathrm{e}R}{t}\right)^\theta t^\alpha \,\mathrm{d}t\right)^{\frac{1}{p}}\left(\int_x^R\left(\ln\frac{\mathrm{e}R}{t}\right)^{-\frac{\mu}{q-1}}t^{-\frac{\nu}{q-1}}\,\mathrm{d}t\right)^{\frac{q-1}{q}}<\infty.
\end{array}\right.
\end{equation*}
From \eqref{eqr2e} and \eqref{eqr6e}, we have already established that the limit as $x\to R$ is finite. Therefore, it remains to verify whether the limit as $x\to0$ is finite, using \eqref{eqr3e} and \eqref{eqr7e}. Summarizing the conclusions for all cases, we obtain:
\begin{center}
\begin{tabular}{|c|c|}
\hline
     Case 1& \eqref{eq0012e} holds $\Leftrightarrow\alpha=-1$, $\theta<-1$, $\nu<q-1$\\
     \hline
     Case 2& \eqref{eq0012e} does not hold\\
     \hline
     Case 3& \eqref{eq0012e} holds $\Leftrightarrow\alpha=-1$, $\theta<-1$, $\nu=q-1>\mu$, and $p\leq-\frac{q(\theta+1)}{q-1-\mu}$ \\
     \hline
     Case 4& \eqref{eq0012e} holds $\Leftrightarrow\alpha=-1$, $\theta<-1$, $\nu=\mu=q-1$\\
     \hline
     Case 5& \eqref{eq0012e} holds $\Leftrightarrow\alpha=-1$, $\theta<-1$, $\nu=q-1<\mu$\\
     \hline
     Case 6& \eqref{eq0012e} holds $\Leftrightarrow\alpha>-1$, $\nu<q-1$\\
     \hline
     Case 7& \eqref{eq0012e} holds $\Leftrightarrow\alpha>-1$, $\nu>q-1$, and \\ &($\alpha>(\nu-q+1)\frac{p}{q}-1$ or ($\alpha=(\nu-q+1)\frac{p}{q}-1$ and $\frac{\theta}{p}\leq\frac{\mu}q$))\\
     \hline
     Case 8& \eqref{eq0012e} holds $\Leftrightarrow\alpha>-1$, $\nu=q-1>\mu$\\
     \hline
     Case 9& \eqref{eq0012e} holds $\Leftrightarrow\alpha>-1$, $\nu=\mu=q-1$\\
     \hline
     Case 10& \eqref{eq0012e} holds $\Leftrightarrow\alpha>-1$, $\nu=q-1<\mu$ \\
     \hline
\end{tabular}
\end{center}
Note the implications of each case:
\begin{itemize}
    \item Case 1 infers item (i);
    \item Case 3 infers item (ii);
    \item Case 4 and 5 infer item (iii);
    \item Case 6, 8, 9, and 10 infer items (iv);
    \item Cases 7 infers items (v) and (vi).
\end{itemize}
This concludes the proof.
\end{proof}

\begin{prop}\label{propOK3e}
Assume $v\in AC_{R}(0,R)$ and $1\leq p<q<\infty$. Given $\alpha,\theta,\nu,\mu\in\mathbb R$, the inequality \eqref{eqOKe} holds if and only if one of the following conditions is fulfilled:
\begin{flushleft}
    $\mathrm{(i)}$ $\alpha=-1$, $\theta<-1$, and $\nu<q-1$;\\
    $\mathrm{(ii)}$ $\alpha=-1$, $\theta<-1$, $\nu=q-1$, $\mu<q-1$, and $p<-\frac{(\theta+1)q}{q-1-\mu}$;\\
    $\mathrm{(iii)}$ $\alpha=-1$, $\theta<-1$, $\nu=q-1$, and $\mu\geq q-1$;\\
    $\mathrm{(iv)}$ $\alpha>-1$ and $\nu\leq q-1$;\\
    $\mathrm{(v)}$ $\alpha>-1$, $\nu>q-1$, and $p<\frac{(\alpha+1)q}{\nu-q+1}$;\\
    $\mathrm{(vi)}$ $\alpha>-1$, $\nu>q-1$, $p=\frac{(\alpha+1)q}{\nu-q+1}$, and $p(\mu+1)>q(\theta+1)$.
\end{flushleft}
\end{prop}
\begin{proof}
This proof is based on the same technique presented in Proposition \ref{propOK3}. By \cite[Theorem 6.3]{zbMATH00046945}, \eqref{eqOKe} holds if and only if the following integral is finite
\begin{equation}\label{eq311e}
\int_0^R\left(\int_0^x\left(\ln\frac{\mathrm{e}R}{t}\right)^\theta t^\alpha \,\mathrm{d}t\right)^{\frac{q}{q-p}}\left(\int_x^R\left(\ln\frac{\mathrm{e}R}{t}\right)^{-\frac{\mu}{q-1}}t^{-\frac{\nu}{q-1}}\,\mathrm{d}t\right)^{q\frac{p-1}{q-p}}\left(\ln\frac{\mathrm{e}R}{x}\right)^{-\frac{\mu}{q-1}}x^{-\frac{\nu}{q-1}}\, \mathrm{d} x.
\end{equation}
Making the change of variables $y=\ln\frac{\mathrm{e}R}{x}$, we note that, given $a,b\in\mathbb R$.
\begin{equation*}
\int_0^{\frac{R}2}\left(\ln\frac{\mathrm{e}R}{x}\right)^ax^b\, \mathrm{d} x=(\mathrm{e}R)^{b+1}\int_{1+\ln2}^\infty y^a\mathrm{e}^{-(b+1)y}\,\mathrm{d}y\left\{\begin{array}{ll}
    =\infty&\mbox{if }b<-1\\
     <\infty&\mbox{if }b>-1\\
     \left\{\begin{array}{ll}
        <\infty&\mbox{if }a<-1  \\
          =\infty&\mbox{if }a\geq-1 
     \end{array}\right.&\mbox{if }b=-1.
\end{array}\right.
\end{equation*}
Then
\begin{equation}\label{eq692e}
\int_0^{\frac{R}2}\left(\ln\frac{\mathrm{e}R}{x}\right)^ax^b\, \mathrm{d} x<\infty\Leftrightarrow b>-1\mbox{ or }\left(b=-1\mbox{ and }a<-1\right).
\end{equation}

The analysis is divided into ten cases:
\begin{center}
\begin{tabular}{ |c|c|c|c|c|c| } 
 \hline
  & $\nu<q-1$ & $\nu>q-1$ & $\nu=q-1>\mu$ & $\nu=\mu=q-1$ & $\nu=q-1<\mu$\\
  \hline
 $\alpha=-1,\theta<-1$ & Case 1 & Case 2 & Case 3 & Case 4 & Case 5\\ 
 \hline
 $\alpha>-1$ & Case 6 & Case 7 & Case 8 & Case 9 & Case 10\\ 
 \hline
\end{tabular}
\end{center}
In each case, we proceed as follows. By splitting the integral in half and observing that the upper half is bounded, we obtain
\begin{equation*}
\mbox{\eqref{eq311e} holds}\Leftrightarrow\int_0^{\frac{R}{2}}A(x)\, \mathrm{d} x<\infty,
\end{equation*}
where $A(x)=A_{p,q,\alpha,\theta,\nu,\mu,x,R}(x)>0$ is the integrand of \eqref{eq311e}. To determine when $\int_0^{\frac{R}{2}}A(x)\, \mathrm{d} x$ is finite, we apply \eqref{eqr3e}, \eqref{eqr7e}, and \eqref{eq692e} in each instance. Summarizing the results across all cases, we obtain:

\begin{center}
\begin{tabular}{|c|c|}
\hline
     Case 1& \eqref{eq311e} is finite $\Leftrightarrow$ $\alpha=-1$, $\theta<-1$, and $\nu<q-1$\\
     \hline
     Case 2& \eqref{eq311e} is infinite \\
     \hline
     Case 3& \eqref{eq311e} is finite $\Leftrightarrow$ $\alpha=-1$, $\theta<-1$, $\nu=q-1$, $\mu<q-1$, and $p<-\frac{(\theta+1)q}{q-1-\mu}$\\
     \hline
     Case 4& \eqref{eq311e} is finite $\Leftrightarrow$ $\alpha=-1$, $\theta<-1$, $\nu=q-1$, and $\mu=q-1$\\
     \hline
     Case 5& \eqref{eq311e} is finite $\Leftrightarrow$ $\alpha=-1$, $\theta<-1$, $\nu=q-1$, and $\mu>q-1$\\
     \hline
     Case 6& \eqref{eq311e} is finite $\Leftrightarrow$ $\alpha>-1$ and $\nu<q-1$\\
     \hline
     Case 7& \eqref{eq311e} is finite $\Leftrightarrow$ $\alpha>-1$, $\nu>q-1$, and\\ &($p<\frac{(\alpha+1)q}{\nu-q+1}$ or ($p=\frac{(\alpha+1)q}{\nu-q+1}$ and $p(\mu+1)>q(\theta+1)$))\\
     \hline
     Case 8& \eqref{eq311e} is finite $\Leftrightarrow$ $\alpha>-1$, $\nu=q-1$, and $\mu<q-1$\\
     \hline
     Case 9& \eqref{eq311e} is finite $\Leftrightarrow$ $\alpha>-1$, $\nu=q-1$, and $\mu=q-1$\\
     \hline
     Case 10& \eqref{eq311e} is finite $\Leftrightarrow$ $\alpha>-1$, $\nu=q-1$, and $\mu>q-1$\\
     \hline
\end{tabular}
\end{center}

Note the implications of each case:
\begin{itemize}
    \item Case 1 infers item (i);
    \item Case 3 infers item (ii);
     \item Cases 4 and 5 infer item (iii);
    \item Cases 6, 8, 9, and 10 infer item (iv);
    \item Case 7 infers items (v) and (vi).
\end{itemize}
This concludes the proof of the proposition.
\end{proof}

Now that Propositions \ref{hardy1e}, \ref{propOK2e}, and \ref{propOK3e} have been established, we can proceed to prove Theorem \ref{theo21e}.

\begin{proof}[Proof of Theorem \ref{theo21e}]
Here, we apply the previously established propositions to the following constants:
\begin{equation*}
q=k+1,\ \theta=0,\ \mu=\frac{\beta n}2,\mbox{ and }\nu=n-k.
\end{equation*}
The proof is divided into four cases:

\smallskip

\noindent\underline{Case 1:} When $p\geq k+1$ and $k\geq\frac{n}{2}$, we apply item (v) of Proposition \ref{hardy1e} alongside item (iv) of Proposition \ref{propOK2e}. This establishes that inequality \eqref{aosfnase} holds if and only if one of the conditions in items (i) or (ii) of Theorem \ref{theo21e} is satisfied.

\smallskip

\noindent\underline{Case 2:} When $p \geq k+1$ and $k < \frac{n}{2}$, we use items (vi), (vii), and (viii) of Proposition \ref{hardy1e} together with items (v) and (vi) of Proposition \ref{propOK2e}. This shows that inequality \eqref{aosfnase} holds if and only if one of the conditions in items (iii), (iv), (v), (vi) or (vii) of Theorem \ref{theo21e} occurs.

\smallskip

\noindent\underline{Case 3:} When $p < k+1$ and $k \geq \frac{n}{2}$, we rely on item (iv) of Proposition \ref{propOK3e}, demonstrating that inequality \eqref{aosfnase} holds if and only if item (ii) of Theorem \ref{theo21e} is satisfied.

\smallskip

\noindent\underline{Case 4:} When $p < k+1$ and $k < \frac{n}{2}$, we apply items (v) and (vi) of Proposition \ref{propOK3e}, establishing that inequality \eqref{aosfnase} holds if and only if one of the conditions in items (vi) or (vii) of Theorem \ref{theo21e} occurs.
\end{proof}

\section{Transported log-weighted Trudinger-Moser inequalities}\label{section3}
In this section, we will discuss log-weighted Trudinger-Moser type inequalities in the setting of weighted Sobolev spaces $X^{1,k+1}_{1,w}$. We recall that
\begin{equation}\label{XR1}
X^{1,k+1}_{1,w}=\mathrm{cl}\Big\{v\in AC_{R}(0,1) \;:\; \|v\|_{w}=\Big(c_{n}\int_{0}^{1}r^{n-k}|v^{\prime}|^{k+1}w\,\mathrm{d} r\Big)^{\frac{1}{k+1}}<\infty \Big\},
\end{equation}
where the closure is taken with respect to the norm $\|\cdot\|_w$, and $c_n$ is defined in \eqref{critical-constant}. 

We note that, from \eqref{normPhi} and \eqref{XR1}, for any $u\in \Phi^{k}_{0,\mathrm{rad}}(B, w)$ by setting $v(r)=u(|x|)$ with $r=|x|$, we obtain $v\in X^{1,k+1}_{1,w} $ and, in addition
\begin{equation}\label{gate-norm}
\|u\|_{\Phi, w}=\|v\|_{w}
\end{equation}
and 
\begin{equation}\label{gate-fuctional}
\int_{B}F(u)\,\mathrm{d} r=\omega_{n-1}\int_{0}^{1}r^{n-1}F(v)\,\mathrm{d} r
\end{equation}
where $s\mapsto F(s)$ can be any of the prototypes
$$
F(s)=\mathrm{e}^{|s|^{\gamma}}, \;\; F(s)=\mathrm{e}^{\alpha |s|^{\gamma_{n,\beta}}},\;\; F(s)=\mathrm{e}^{\mathrm{e}^{|s|^{\frac{n+2}{n}}}}\;\;\text {or} \;\; F(s)=\mathrm{e}^{a\mathrm{e}^{c_n|s|^{\frac{n+2}{n}}}},
$$ 
provided that the integral is well-defined.

The identities \eqref{gate-norm} and \eqref{gate-fuctional} will serve as bridges, allowing us to investigate the estimates in $\Phi^{k}_{0,\mathrm{rad}}(B, w)$ through their corresponding ones in $X^{1,k+1}_{1,w}$.
\begin{lemma}\label{r-estimate} Assume that $k=n/2$ and let $w$ be a positive weight function on $(0,1)$. Then, for any $v \in X^{1,k+1}_{1, w}$ holds
\begin{equation*}
|v(r)-v(t)| \leq A^{\frac{n}{n+2}}_{w}(r;t)\left(c_n\int_{r}^{t}s^{n-k}|v^{\prime}(s)|^{k+1}w(s)\,\mathrm{d}s\right)^{\frac{2}{n+2}},\;\; \mbox{for any}\;\; 0<r<t\leq 1
\end{equation*}
where $A_{w}:[0,1]\times[0,1]\to \mathbb{R}$ is given by
\begin{equation*}
    A_{w}(r;t)=\frac{1}{c^{\frac{2}{n}}_{n}}\int_{r}^{t}\frac{1}{sw^{\frac{2}{n}}(s)}\, \mathrm{d}s.
\end{equation*}
\end{lemma}
\begin{proof}
For any $v\in X^{1,k+1}_{1,w}$, the H\"{o}lder inequality yields 
\begin{equation*}
\begin{aligned}
 |v(r)-v(t)| &\leq \int_{r}^{t}|v^{\prime}(s)|\,\mathrm{d}s\\
        & = \int_{r}^{t}\left(s^{\frac{n-k}{k+1}}(c_nw(s))^{\frac{1}{k+1}}|v^{\prime}(s)|\right)\left(s^{-\frac{n-k}{k+1}}(c_nw(s))^{-\frac{1}{k+1}}\right)\,\mathrm{d}s\\
        &\leq \left(\frac{1}{c^{\frac{2}{n}}_n}\int_{r}^{t}\frac{1}{s w^{\frac{2}{n}}(s)}\,\mathrm{d}s\right)^{\frac{n}{n+2}}\left(c_n\int_{r}^{t}s^{n-k}|v^{\prime}(s)|^{k+1}w(s)\,\mathrm{d}s\right)^{\frac{2}{n+2}},
\end{aligned}
\end{equation*}
for any $0<r<t\leq 1$.
\end{proof}
\begin{cor}\label{corollary-radial} Suppose $k=n/2$ and let $w_0, w_1: (0,1)\to \mathbb{R}$ be given by $w_0(r)=\left(\ln \frac{1}{r} \right)^{\frac{\beta n}{2}}$ for $\beta<1$ and $w_1(r)=\left(\ln \frac{\mathrm{e}}{r} \right)^{\frac{\beta n}{2}}$ for $\beta\leq 1$. Then,  
\begin{enumerate}
    \item [$(a)$] For $v \in X^{1,k+1}_{1,w_0}$ we have \begin{equation*}
|v(r)| \leq \frac{1}{c^{\frac{2}{n+2}}_n(1-\beta)^{\frac{n}{n+2}}}\left(\ln \frac{1}{r}\right)^{(1-\beta)\frac{n}{n+2}}\|v\|_{w_0},\;\; \mbox{for all}\;\; 0<r<1
\end{equation*}
\item [$(b)$] For $v \in X^{1,k+1}_{1,w_1}$ we have  
\begin{equation*}
    |v(r)|\leq \left\{ \begin{aligned}
        & \frac{1}{c^{\frac{2}{n+2}}_n|1-\beta|^{\frac{n}{n+2}}}\left|\left(\ln \frac{\mathrm{e}}{r}\right)^{1-\beta}-1\right|^{\frac{n}{n+2}}\|v\|_{w_1}\;\;&\mbox{if}\;\;\beta\neq 1\\
        & \frac{1}{c^{\frac{2}{n+2}}_n}\left(\ln(\ln\frac{\mathrm{e}}{r})\right)^{\frac{n}{n+2}}\|v\|_{w_1}\;\;&\mbox{if}\;\;\beta=1
    \end{aligned}\right.
\end{equation*}
for any $r\in(0,1)$.
\end{enumerate}
\end{cor}
\begin{proof}
From Lemma~\ref{r-estimate} with $t=1$, we only need to calculate $A_{w_0}(r;1)$ and $A_{w_1}(r;1)$. Firstly, 
\begin{equation*}
    \begin{aligned}
        c^{\frac{2}{n}}_{n}A_{w_0}(r;1)=\int_{r}^{1}\frac{1}{s}\left(-\ln{s}\right)^{-\beta}\,\mathrm{d}s=\frac{1}{\beta-1}\left(-\ln s\right)^{1-\beta}\Big|_{s=r}^{s=1}=\frac{1}{1-\beta}\left(\ln \frac{1}{r}\right)^{1-\beta}.
    \end{aligned}
\end{equation*}
Analogously, 
\begin{equation*}
    c^{\frac{2}{n}}_{n}A_{w_1}(r;1)=\int_{r}^{1}\frac{1}{s}\left(\ln{\frac{\mathrm{e}}{s}}\right)^{-\beta}\,\mathrm{d}s=\left\{ \begin{aligned}
        & \frac{1}{1-\beta}\left(\left(\ln \frac{\mathrm{e}}{r}\right)^{1-\beta}-1\right)\;\;&\mbox{if}\;\;\beta\neq 1\\
        & \ln(\ln\frac{\mathrm{e}}{r})\;\;&\mbox{if}\;\;\beta=1.
    \end{aligned}\right.
\end{equation*}
\end{proof}
In order to deal with either \textit{critical case} $\alpha=\alpha_{n,\beta}$ in Theorem~\ref{thm1}-$(b)$ or \textit{critical exponent} $a=n$ in Theorem~\ref{thm2}-$(b)$, we will make use of an integral inequality due to Leckband~\cite[Theorem~3]{zbMATH03867635}, which we will describe below.
\begin{definition}
    A function $N: [0, \infty) \rightarrow [0, \infty)$ will be called a $C^*$-convex function if $N(0) = 0$, $N$ is convex, $N \in C^1[0, \infty)$, and $\rho$ defined by the differential equation $\rho(N(t)) = N^{\prime}(t)$ is such that there exists a constant $C_\rho < \infty$ such that for $0 < d < \infty$, we have a constant $C_d < \infty$ with
\[
\rho((l + d)s) \leq C_\rho \cdot \rho((l - 1)s),
\]
for all $l > C_d$ and $0 < s < \infty$.
\end{definition}
 Some examples of $C^*$-convex functions are $N(t) = \mathrm{e}^t - 1$, $N(t) = \mathrm{e}^{t^2} - 1$, and $N(t) = t^p$, $p \geq 1$.

The following result of Leckband~\cite{zbMATH03867635} is a general integral inequality that extends a previous result by Neugebauer~\cite{zbMATH03758615} and Moser~\cite{zbMATH03323360}; see also \cite{zbMATH04132525}.
\begin{theoremletter}[Leckband\cite{zbMATH03867635}]\label{Leckband} Let $f \in L^p(0, \infty)$ such that $\|f\|_p \leq 1$ with $1\leq p<\infty$, $\varphi : [0,\infty) \rightarrow [0,\infty)$ locally integrable, and set
\[
G(x) = \left( \int_0^x \varphi^{\frac{p}{p-1}}(y)\,\mathrm{d}y \right)^{\frac{p-1}{p}}\;\;\mbox{and}\;\; F(x) = \int_0^x f(y) \varphi(y)\,\mathrm{d}y.
\]
Let $\Phi \geq 0$ be a nonincreasing function on $[0, \infty)$, and $N(t)$ a $C^*$-convex function. Then there exists a constant $C > 0$
\[
\int_0^{\infty} \Phi \left( N(G(t)) - N(F(t)) \right)\,\mathrm{d}m^{*} \leq C \|\Phi\|_1,
\]
where $m^{*}$ is the measure induced by $N(G(t))$.
\end{theoremletter}
To treat Theorem~\ref{thm1}, with help of \eqref{gate-norm} and \eqref{gate-fuctional}, we will first deduce the following transported Log-weighted sharp Trudinger-Moser inequality.
\begin{prop}\label{prop-thm1}
Assume $k=n/2$ and let $\beta\in [0,1)$ and let $w=w_0$ or $w=w_1$ as given in Corollary~\ref{corollary-radial}. 
\begin{enumerate}
\item [$(a)$ ] Then
\begin{equation*}
\int_{0}^{1}r^{n-1}\mathrm{e}^{|v|^{\gamma}}\mathrm \,\mathrm{d} r < \infty, \;\forall\, v\in X^{1,k+1}_{1,w}\;\;\Longleftrightarrow\;\;\gamma\leq \gamma_{n,\beta}=\frac{n+2}{n(1-\beta)}. 
\end{equation*}
\end{enumerate}
\begin{enumerate}
\item [$(b)$ ] Let $\widetilde{\Sigma}=\{v\in X^{1,k+1}_{1,w}\,:\,\|v\|_{w}\leq 1\}$, then
\begin{equation}\label{PI-ub}
\mathcal{MT}(n,\alpha,\beta)=\sup_{v\in\widetilde{\Sigma}  }\int_{0}^{1}r^{n-1}\mathrm{e}^{\alpha|v|^{\gamma_{n,\beta}}}\,\mathrm \,\mathrm{d} r < \infty \;\Longleftrightarrow\; \alpha\leq \alpha_{n,\beta}=n\left[c^{\frac{2}{n}}_n(1-\beta)\right]^{\frac{1}{1-\beta}}.
\end{equation}
\end{enumerate}
\end{prop} 
\begin{proof}
 Firstly, the inequality $w_1\ge w_0$ on $(0,1)$ yields the continuous embedding
 \begin{equation}\label{mergulho-w1w0}
     X^{1,k+1}_{1,w_1}\hookrightarrow X^{1,k+1}_{1,w_0}.
 \end{equation}
Thus, for the sufficiency part it is enough to consider the case $w(r)=w_0(r)=\left(\ln\frac{1}{r}\right)^{\frac{\beta n}{2}}$ and $\gamma=\gamma_{n,\beta}$. Further, by density we can assume that $v\in X^{1,k+1}_{1,w_0}$ is such that $v\ge0$, $v\in AC_{\mathrm{loc}}(0, 1)$ and $\lim_{r\to 1}v(r)=0.$ The change of variables
 \begin{equation}\label{changeV}
     t=n\ln\frac{1}{r},\quad \psi(t)=c_n^{\frac{2}{n+2}}n^{\frac{n}{n+2}(1-\beta)}(1-\beta)^{\frac{n}{n+2}}v(r)
 \end{equation}
 implies
 \begin{equation}\label{portal-norm}
     c_n\int_{0}^{1}r^{n-k}|\ln r|^{\frac{\beta n}{2}}|v^{\prime}(r)|^{k+1}\,\mathrm{d} r=\frac{1}{(1-\beta)^{\frac{n}{2}}}\int_{0}^{\infty}t^{\frac{\beta n}{2}}|\psi^{\prime}(t)|^{k+1}\,\mathrm{d}t
 \end{equation}
 and 
 \begin{equation}\label{portal-funcional}
     \int_{0}^{1}r^{n-1}\mathrm{e}^{\alpha|v|^{\gamma}}\,\mathrm{d} r=\int_{0}^{\infty}\mathrm{e}^{\overline{\alpha}|\psi|^{\gamma}-t}\,\mathrm{d}t
 \end{equation}
 where
 \begin{equation*}\gamma=\gamma_{n,\beta}\;\;\; \mbox{and}\;\; \; \overline{\alpha}=\frac{\alpha}{\alpha_{n,\beta}}.\end{equation*}
 To prove the integral in $(a)$ is finite, from \eqref{portal-norm} and \eqref{portal-funcional} we only need to show that 
 \begin{equation}\label{psi-Int}
     \int_{0}^{\infty}\mathrm{e}^{\overline{\alpha}|\psi|^{\gamma}-t}\,\mathrm{d}t<\infty\quad\mbox{whenever}\quad \frac{1}{(1-\beta)^{\frac{n}{2}}}\int_{0}^{\infty}t^{\frac{\beta n}{2}}|\psi^{\prime}(t)|^{k+1}\,\mathrm{d}t<\infty.
 \end{equation}
 In order to get this, we can apply the Moser's argument \cite{zbMATH03323360}. In fact, for all $\epsilon>0$ there exists $T=T(\epsilon)$ such that 
 \begin{equation*}
  \frac{1}{(1-\beta)^{\frac{n}{2}}}\int_{T}^{\infty}t^{\frac{\beta n}{2}}|\psi^{\prime}(t)|^{k+1}\,\mathrm{d}t<\epsilon^{\frac{n+2}{2}}.   
 \end{equation*}
By Hölder's inequality and $k=\frac{n}{2}$ 
\begin{equation}\label{Holder-radial}
\begin{aligned}
\psi(t) & = \psi(T) + \int_T^t \psi^{\prime}(s) s^{\frac{\beta n}{n+2}} s^{-\frac{\beta n}{n+2}} \,\mathrm{d}s \\
&\leq \psi(T) + \left( \int_T^t s^{\frac{\beta n}{2}}|\psi^{\prime}(s)|^{k+1} \,\mathrm{d}s \right)^{\frac{2}{n+2}} \left( \frac{t^{1-\beta} - T^{1-\beta}}{1-\beta} \right)^{\frac{n}{n+2}} \\
&\leq \psi(T) + \epsilon t^{\frac{1}{\gamma_{n,\beta}}}\left[ \frac{1}{1-\beta}\left(1 - \left(\frac{T}{t}\right)^{1-\beta}\right) \right]^{\frac{n}{n+2}} 
\end{aligned}
\end{equation}
for all $ t \geq T.$ Then, $\lim_{t\to\infty}\psi(t)/t^{1/\gamma_{n,\beta}}=0$ and, thus there exists $\overline{T}$ such that
\[
\overline{\alpha} \psi^{\gamma_{n,\beta}}(t)\leq \frac{t}{2}, \quad \text{for all } t \geq \overline{T}.
\]
From \eqref{changeV}, we also have $\lim_{t\to 0}\psi(t)=0$. Then the above estimate ensures the existence of the integral in \eqref{portal-funcional}. The proof of the uniform estimate in \eqref{PI-ub} in the subcritical case $\alpha<\alpha_{n,\beta}$ i.e. $\overline{\alpha}<1$ follows from Corollary~\ref{corollary-radial}. In fact, suppose that $\|v\|_{X,w_0}\leq 1$, then from \eqref{changeV}, \eqref{portal-norm} and Corollary~\ref{corollary-radial}-$(a)$ we can write
\begin{equation}\label{psi-radial}
    \psi^{\gamma_{n,\beta}}(t)\leq t, \;\;\forall\, t>0.
\end{equation}
Hence, if $\overline{\alpha}<1$, the estimate \eqref{psi-radial} and \eqref{portal-funcional} yield
\begin{equation*}
    \int_{0}^{1}r^{n-1}\mathrm{e}^{\alpha|v|^{\gamma_{n,\beta}}}\,\mathrm{d} r=\int_{0}^{\infty}\mathrm{e}^{\overline{\alpha}\psi^{\gamma_{n,\beta}}(t)-t}\,\mathrm{d}t\leq  \int_{0}^{\infty}\mathrm{e}^{\overline{\alpha}t-t}\,\mathrm{d}t=\frac{1}{1-\overline{\alpha}}<\infty.
\end{equation*}
The critical case $\overline{\alpha}=1$, however, is more delicate and we will need to use the integral inequality in Theorem~\ref{Leckband}. Indeed, we set
\begin{equation*}
    f(t)=\psi^{\prime}(t)\left(\frac{t^{\beta}}{1-\beta}\right)^{\frac{n}{n+2}}\quad\mbox{and}\quad \varphi(t)=\left(\frac{t^{\beta}}{1-\beta}\right)^{-\frac{n}{n+2}}
\end{equation*}
where $\psi$ is given by \eqref{changeV}. Then, $f$ satisfies 
\begin{equation*}
    \int_{0}^{\infty}|f(t)|^{k+1}\,\mathrm{d}t=\frac{1}{(1-\beta)^{\frac{n}{2}}}\int_{0}^{\infty}t^{\frac{\beta n}{2}}|\psi^{\prime}(t)|^{k+1}\,\mathrm{d}t=\|v\|^{k+1}_{w_0}\leq 1.
\end{equation*}
In addition, from $k=n/2$
\begin{equation*}
   G(x)=\left(\int_{0}^{x}\varphi^{\frac{k+1}{k}}(t)\,\mathrm{d}t\right)^{\frac{k}{k+1}} =\left((1-\beta)\int_{0}^{x} t^{-\beta}\,\mathrm{d}t\right)^{\frac{n}{n+2}}=x^{(1-\beta)\frac{n}{n+2}}
\end{equation*}
and
\begin{equation*}
    F(x)=\int_{0}^{x}f(t)\varphi(t)\,\mathrm{d}t=\int_{0}^{x}\psi^{\prime}(t)\,\mathrm{d}t=\psi(x).
\end{equation*}
Lastly, by setting
\begin{equation*}
    N(s)=s^{\frac{n+2}{n(1-\beta)}}\quad\mbox{and}\quad \Phi(s)=\mathrm{e}^{-s}
\end{equation*}
we obtain
\begin{equation*}
    N(G(t))=t\quad\mbox{and}\quad N(F(t))=\psi^{\gamma_{n,\beta}}(t).
\end{equation*}
Applying Theorem~\ref{Leckband} with the above choices we have
\begin{equation*}
  \int_{0}^{1}r^{n-1}\mathrm{e}^{\alpha_{n,\beta}|v|^{\gamma_{n,\beta}}}\,\mathrm{d} r=   \int_{0}^{\infty}\mathrm{e}^{\psi^{\gamma_{n,\beta}}(t)-t}\,\mathrm{d} r\leq C\int_{0}^{\infty}\mathrm{e}^ {-t}\,\mathrm{d}t=C.
\end{equation*}
\paragraph{\textit{Sharpness (a).}} In view of \eqref{mergulho-w1w0}, it is sufficient to show sharpness of $\gamma_{n,\beta}$ for $w=w_1(r)=(\ln\frac{\mathrm{e}}{r})^{\frac{\beta n}{2}}$ with $\beta\in [0,1)$. Suppose that $\gamma=\gamma_{n,\beta}+\epsilon$ for some $\epsilon>0$. Let $\eta>0$ be taken small enough such that $\overline{\eta}:=\frac{\epsilon}{\gamma_{n,\beta}}-\eta(\epsilon+\gamma_{n,\beta})>0$. Thus,
\begin{equation}\label{eta-choice}
(\gamma_{n,\beta}+\epsilon)\left(\frac{1}{\gamma_{n,\beta}}-\eta\right)=1+\frac{\epsilon}{\gamma_{n,\beta}}-\eta(\epsilon+\gamma_{n,\beta})=1+\overline{\eta}>1.
\end{equation}
Let us take $v:(0,1]\to\mathbb{R}$ be given by
\begin{equation}\label{v-optimalMoser}
    v(r)=-\left\{\begin{aligned}
    & \left(n\ln\frac{1}{r}\right)^{\frac{1}{\gamma_{n,\beta}}-\eta}\quad&\mbox{if}&\quad 0<r\leq \mathrm{e}^{-\frac{1}{n}}\\
    & n\ln\frac{1}{r}\quad &\mbox{if}&\quad \mathrm{e}^{-\frac{1}{n}}\leq r\leq 1.
    \end{aligned}\right.
\end{equation}
Note that
\begin{align*}
    \int_{0}^{1}r^{n-k}w_1(r)|v^{\prime}|^{k+1}\,\mathrm{d} r&=\left(\frac{n}{\gamma_{n,\beta}}-\eta n\right)^{k+1}\int_{0}^{\mathrm{e}^{-\frac{1}{n}}}r^{-1}\left(\ln\frac{\mathrm{e}}{r}\right)^{\frac{\beta n}{2}}\left(n\ln\frac{1}{r}\right)^{\left(\frac{1}{\gamma_{n,\beta}}-\eta-1\right)(k+1)}\,\mathrm{d} r\\
    &\quad+n^{k+1}\int_{\mathrm{e}^{-\frac{1}{n}}}^{1}r^{n-k}\left(\ln\frac{\mathrm{e}}{r}\right)^{\frac{\beta n}{2}}r^{-k-1}\,\mathrm{d} r\\
    &=\left(\frac{n}{\gamma_{n,\beta}}-\eta n\right)^{k+1}\int_{0}^{\mathrm{e}^{-\frac{1}{n}}}r^{-1}\left(\ln\frac{\mathrm{e}}{r}\right)^{\frac{\beta n}{2}}\left(n\ln\frac{1}{r}\right)^{\left(\frac{1}{\gamma_{n,\beta}}-\eta-1\right)(k+1)}\,\mathrm{d} r\\
    &\quad+n^{k+1}\int_{\mathrm{e}^{-\frac{1}{n}}}^{1}r^{-1}\left(\ln\frac{\mathrm{e}}{r}\right)^{\frac{\beta n}{2}}\,\mathrm{d} r    
\end{align*}
and, by setting $t=n\ln\frac{1}{r}$ we obtain
\begin{align*}
    \int_{0}^{\mathrm{e}^{-\frac{1}{n}}}r^{-1}\left(\ln\frac{\mathrm{e}}{r}\right)^{\frac{\beta n}{2}}\left(n\ln\frac{1}{r}\right)^{\left(\frac{1}{\gamma_{n,\beta}}-\eta-1\right)(k+1)}&\,\mathrm{d} r=\frac{1}{n}\int_{1}^{\infty}\left(1+\frac{t}{n}\right)^{\frac{\beta n}{2}}t^{\left(\frac{1}{\gamma_{n,\beta}}-\eta-1\right)(k+1)}\,\mathrm{d}t\\
    &\leq \frac{1}{n}\left(1+\frac{1}{n}\right)^{\frac{\beta n}{2}}\int_{1}^{\infty}t^{\left(\frac{1}{\gamma_{n,\beta}}-\eta-1\right)(k+1)+\frac{\beta n}{2}}\,\mathrm{d}t\\
     &=\frac{1}{n}\left(1+\frac{1}{n}\right)^{\frac{\beta n}{2}}\int_{1}^{\infty}t^{-\eta(k+1)-1}\,\mathrm{d}t<\infty.
\end{align*}
Thus, we get $v\in X^{1,k+1}_{1,w_1}$. In addition, from \eqref{eta-choice}
\begin{align*}
    \int_{0}^{1}r^{n-1}\mathrm{e}^{|v|^{\gamma}}\,\mathrm{d} r &\ge \int_{0}^{\mathrm{e}^{-\frac{1}{n}}}r^{n-1}\mathrm{e}^{(n\ln\frac{1}{r})^{(\frac{1}{\gamma_{n,\beta}}-\eta)(\gamma_{n,\beta}+\epsilon)}}\,\mathrm{d} r\\
    &=\int_{0}^{\mathrm{e}^{-\frac{1}{n}}}r^{n-1}\mathrm{e}^{(n\ln\frac{1}{r})^{1+\overline{\eta}}}\,\mathrm{d} r\\
    &=\frac{1}{n}\int_{1}^{\infty}\mathrm{e}^{t^{1+\overline{\eta}}-t}\,\mathrm{d}t=\infty.
\end{align*}

\paragraph{\textit{Sharpness (b).}} Firstly, we will treat the case $w=w_0(r)=(\ln\frac{1}{r})^{\frac{\beta n}{2}}$. For each $\ell\in\mathbb{N}$, let us define $v_{\ell}:(0,1]\to\mathbb{R}$ given by
\begin{equation}\label{velloptimalMoser}
    v_{\ell}(r)=-\left(\frac{\ell}{\alpha_{n,\beta}}\right)^{\frac{1}{\gamma_{n,\beta}}}\left\{\begin{aligned}
    & 1\quad &\mbox{if}&\quad 0<r\leq  \mathrm{e}^{-\frac{\ell}{n}}\\
    & \frac{\left(n\ln\frac{1}{r}\right)^{1-\beta}}{\ell^{1-\beta}}\quad&\mbox{if}&\quad  \mathrm{e}^{-\frac{\ell}{n}}\leq r\leq 1.
    \end{aligned}\right.
\end{equation}
We claim that
\begin{equation}\label{eq3.20}
    c_n\int_{0}^{1}r^{n-k}w_0(r)|v^{\prime}_{\ell}|^{k+1}\mathrm{d}r=1.
\end{equation}
Indeed, by a straightforward calculation, we get
\begin{align*}
c_n\int_{0}^{1}r^{n-k}w_0(r)|v^{\prime}_{\ell}|^{k+1}\,\mathrm{d} r&=c_n\left(\frac{\ell}{\alpha_{n,\beta}}\right)^{\frac{k+1}{\gamma_{n,\beta}}}(1-\beta)^{k+1}\left(\dfrac{n}\ell\right)^{(k+1)(1-\beta)}\int_{\mathrm{e}^{-\frac{\ell}{n}}}^1r^{-1}\left(\ln\frac{1}{r}\right)^{-\beta}\,\mathrm{d} r\\
&=c_n\left(\frac{\ell}{\alpha_{n,\beta}}\right)^{\frac{k+1}{\gamma_{n,\beta}}}(1-\beta)^{k+1}\left(\dfrac{n}\ell\right)^{(k+1)(1-\beta)}\dfrac{1}{1-\beta}\left(\dfrac{n}\ell\right)^{\beta-1}\\
&=c_n\left(\frac{\ell}{\alpha_{n,\beta}}\right)^{\frac{k+1}{\gamma_{n,\beta}}}(1-\beta)^{k}\left(\dfrac{n}\ell\right)^{k(1-\beta)}=1.
\end{align*}
This concludes \eqref{eq3.20}. In addition, for $\alpha>\alpha_{n,\beta}$, we have 
\begin{align*}
    \int_{0}^{1}r^{n-1}\mathrm{e}^{\alpha |v_{\ell}|^{\gamma_{n,\beta}}}\,\mathrm{d} r &\ge \int_{0}^{\mathrm{e}^{-\frac{\ell}{n}}}r^{n-1}\mathrm{e}^{\alpha |v_{\ell}|^{\gamma_{n,\beta}}}\,\mathrm{d} r= \int_{0}^{\mathrm{e}^{-\frac{\ell}{n}}}r^{n-1}\mathrm{e}^{\alpha\frac{\ell}{\alpha_{n,\beta}}}\,\mathrm{d} r= \frac{1}{n}\mathrm{e}^{\ell(\frac{\alpha}{\alpha_{n,\beta}}-1)}\rightarrow \infty
\end{align*}
as $\ell$ goes to $\infty$. This concludes the sharpness of the item \textit{(b)} for $w=w_0(r)$. 

For the case $w=w_1(r)=(\ln\frac{\mathrm{e}}{r})^{\frac{\beta n}{2}}$, we define $\overline v_\ell\colon(0,1]\to\mathbb R$, for each $\ell>n$, given by
\begin{equation}\label{vellbar}
    \overline v_{\ell}(r)=-\left(\frac{\ell}{\alpha_{n,\beta}}\right)^{\frac{1}{\gamma_{n,\beta}}}\left(\dfrac{\ell^{1-\beta}}{\ell^{1-\beta}-n^{1-\beta}}\right)^{\frac{1}{k+1}}\left\{\begin{aligned}
    & \dfrac{\ell^{1-\beta}-n^{1-\beta}}{\ell^{1-\beta}}\quad &\mbox{if}&\quad 0<r\leq  \mathrm{e}^{1-\frac{\ell}{n}}\\
    & \frac{\left(n\ln\frac{\mathrm{e}}{r}\right)^{1-\beta}-n^{1-\beta}}{\ell^{1-\beta}}\quad&\mbox{if}&\quad  \mathrm{e}^{1-\frac{\ell}{n}}\leq r\leq 1.
    \end{aligned}\right.
\end{equation}
By applying a similar argument as used in the proof of \eqref{eq3.20}, we obtain
\begin{equation*}
    c_n\int_{0}^{1}r^{n-k}w_0(r)|\overline v^{\prime}_{\ell}|^{k+1}\,\mathrm{d} r=1.
\end{equation*}
Moreover, for $\alpha>\alpha_{n,\beta}$, we have
\begin{align*}
\int_0^1r^{n-1}\mathrm{e}^{\alpha|\overline v_\ell|^{\gamma_{n,\beta}}}\,\mathrm{d} r&\ge \int_{0}^{\mathrm{e}^{1-\frac{\ell}{n}}}r^{n-1}\mathrm{e}^{\alpha |\overline v_{\ell}|^{\gamma_{n,\beta}}}\,\mathrm{d} r\\
&= \int_{0}^{\mathrm{e}^{1-\frac{\ell}{n}}}r^{n-1}\exp\left[\alpha\frac{\ell}{\alpha_{n,\beta}}\left(\dfrac{\ell^{1-\beta}}{\ell^{1-\beta}-n^{1-\beta}}\right)^{-\frac{k\gamma_{n,\beta}}{k+1}}\right]\,\mathrm{d} r\\
&= \frac{\mathrm{e}^n}{n}\exp\left[\ell\left(\frac{\alpha}{\alpha_{n,\beta}}\left(\dfrac{\ell^{1-\beta}}{\ell^{1-\beta}-n^{1-\beta}}\right)^{-\frac{k\gamma_{n,\beta}}{k+1}}-1)\right)\right]\overset{\ell\to\infty}\longrightarrow\infty.
\end{align*}
This concludes the sharpness of \textit{(b)} for $w=w_1(r)$, thereby concluding the proof of the proposition.
\end{proof}
In order to treat Theorem~\ref{thm2} we will use the following transported log-weighted sharp double-exponential Trudinger-Moser inequality. 
\begin{prop}\label{prop-thm2}
Assume $k=n/2$ and let $w_1(x)=\left(\ln\frac{\mathrm{e}}{|x|}\right)^{\frac{n}{2}}$. 
\begin{enumerate}
\item [$(a)$ ] Then, for any $ v\in X^{1,k+1}_{1,w_1}$
\begin{equation*}
\int_{0}^{1}r^{n-1}\mathrm{e}^{\mathrm{e}^{|v|^{\frac{n+2}{n}}}}\,\mathrm{d} r< \infty.
\end{equation*}
\end{enumerate}
\begin{enumerate}
\item [$(b)$ ] Let $\widetilde{\Sigma}_1=\{v\in X^{1,k+1}_{1,w_1}\,:\, \|v\|_{w_1}\leq 1\}$, then
\begin{equation*}
\sup_{v\in \widetilde{\Sigma}_1}\int_{0}^{1}r^{n-1}\mathrm{e}^{a\mathrm{e}^{c^{\frac{2}{n}}_n|v|^{\frac{n+2}{n}}}}\,\mathrm{d} r < \infty\;\;\Longleftrightarrow\;\; a\leq n.
\end{equation*}
\end{enumerate}
\end{prop}
\begin{proof}
To prove $(a)$, we can proceed analogously to Proposition~\ref{prop-thm1}. Indeed, for $v\in X^{1,k+1}_{1,w_1}$ is such that $v\ge0$, $v\in AC_{\mathrm{loc}}(0, 1)$ and $\lim_{r\to 1}v(r)=0$, the change of variables
 \begin{equation}\label{changeVP2}
     t=\ln\frac{1}{r},\quad \psi(t)=\sigma^{\frac{2}{n+2}}v(r),\quad \sigma>0
 \end{equation}
yields
 \begin{equation}\label{portal-normP2}
     \sigma\int_{0}^{1}r^{n-k}\left|\ln \frac{\mathrm{e}}{r}\right|^{\frac{n}{2}}v^{\prime}(r)|^{k+1}\,\mathrm{d} r=\int_{0}^{\infty}(1+t)^{\frac{ n}{2}}|\psi^{\prime}(t)|^{k+1}\,\mathrm{d}t
 \end{equation}
 and, for any $a>0$ 
 \begin{equation}\label{portal-funcionalP2}
     \int_{0}^{1}r^{n-1}\mathrm{e}^{a\mathrm{e}^{\sigma^{\frac{2}{n}}|v|^{\frac{n+2}{n}}}}\,\mathrm{d} r=\int_{0}^{\infty}\mathrm{e}^{a \mathrm{e}^{|\psi|^{\frac{n+2}{n}}}-nt}\,\mathrm{d}t.
 \end{equation}
By using \eqref{portal-normP2} and \eqref{portal-funcionalP2} with $a=1$ and $\sigma=1$, we can conclude $(a)$ provided that 
 \begin{equation}\label{psi-IntP2}
     \int_{0}^{\infty}\mathrm{e}^{\mathrm{e}^{|\psi|^{\frac{n+2}{n}}}-nt}\,\mathrm{d}t<\infty\quad\mbox{whenever}\quad \int_{0}^{\infty}(1+t)^{\frac{ n}{2}}|\psi^{\prime}(t)|^{k+1}\,\mathrm{d}t<\infty
 \end{equation}
 holds. To obtain \eqref{psi-IntP2}, we can apply Moser's argument \cite{zbMATH03323360} again. In fact, for all $\epsilon>0$ there exists $T=T(\epsilon)$ such that 
 \begin{equation*}
  \int_{T}^{\infty}(1+t)^{\frac{n}{2}}|\psi^{\prime}(t)|^{k+1}\,\mathrm{d}t<\epsilon^{k+1}  
 \end{equation*}
and the Hölder's inequality yields
\begin{align*}
\psi(t) & = \psi(T) + \int_T^t \psi^{\prime}(s) (1+s)^{\frac{n}{2(k+1)}} (1+s)^{-\frac{n}{2(k+1)}} \,\mathrm{d}s \\
&\leq \psi(T) + \left( \int_T^t (1+s)^{\frac{n}{2}}|\psi^{\prime}(s)|^{k+1} \,\mathrm{d}s \right)^{\frac{1}{k+1}} \left(\ln\frac{1+t}{1+T}\right)^{\frac{k}{k+1}} \\
&\leq \psi(T) + \epsilon \left(\ln\frac{1+t}{1+T}\right)^{\frac{n}{n+2}}
\end{align*}
for all $ t \geq T.$ Thus, $\lim_{t\to \infty}\psi(t)/(\ln t)^{\frac{n}{n+2}}=0$ and there exists $\overline{T}$ such that
\begin{align*}
    |\psi(t)|^{\frac{n+2}{n}} &\le\ln t,  \quad \text{for all } t \geq \overline{T}.
\end{align*}
Hence, since we also have $\lim_{t\to 0}\psi(t)=0$
\begin{align*}
     \int_{0}^{\infty}\mathrm{e}^{\mathrm{e}^{|\psi|^{\frac{n+2}{n}}}-nt}\,\mathrm{d}t&=\int_{0}^{\overline{T}}\mathrm{e}^{\mathrm{e}^{|\psi|^{\frac{n+2}{n}}}-nt}\,\mathrm{d}t+\int_{\overline{T}}^{\infty}\mathrm{e}^{\mathrm{e}^{|\psi|^{\frac{n+2}{n}}}-nt}\,\mathrm{d}t\\
     &\leq C+\int_{\overline{T}}^{\infty}\mathrm{e}^{-(n-1)t}\,\mathrm{d}t<\infty.
\end{align*}
We first consider the item $(b)$ for the \textit{subcritical} case $a<n$. Here, the uniform estimate follows from Corollary~\ref{corollary-radial}. Indeed, for $\|v\|_{w_1}\leq 1$ Corollary~\ref{corollary-radial}-$(b)$, and the change \eqref{changeVP2} and \eqref{portal-normP2} with $\sigma=c_n$ yield
\begin{equation*}
    |\psi(t)|^{\frac{n+2}{n}}\leq \ln(1+ t), \;\;\forall\, t>0.
\end{equation*} 
Hence, \eqref{portal-funcionalP2} implies
\begin{align*}
  \int_{0}^{1}r^{n-1}\mathrm{e}^{a\mathrm{e}^{c^{\frac{2}{n}}_n|v|^{\frac{n+2}{n}}}}\,\mathrm{d} r& = \int_{0}^{\infty}\mathrm{e}^{a \mathrm{e}^{|\psi|^{\frac{n+2}{n}}}-nt}\,\mathrm{d}t\\
  & \leq \mathrm{e}^{a}\int_{0}^{\infty}\mathrm{e}^{-(n-a)t}\,\mathrm{d}t=\frac{\mathrm{e}^a}{n-a}
\end{align*}
for $a<n$. To treat the critical case $a=n$, we will use Theorem~\ref{Leckband}. Actually, we set 
\begin{equation*}
    f(t)=\psi^{\prime}(t)(1+t)^{\frac{n}{n+2}}\quad\mbox{and}\quad \varphi(t)=(1+t)^{-\frac{n}{n+2}}
\end{equation*}
where $\psi$ is given by \eqref{changeVP2} with $\sigma=c_n$. Then, from \eqref{portal-normP2} 
\begin{equation*}
    \int_{0}^{\infty}|f(t)|^{k+1}\,\mathrm{d}t=\int_{0}^{\infty}(1+t)^{\frac{n}{2}}|\psi^{\prime}(t)|^{k+1}\,\mathrm{d}t=\|v\|^{k+1}_{w_1}\leq 1,
\end{equation*}
and 
\begin{equation*}
   G(x)=\left(\int_{0}^{x}\varphi^{\frac{k+1}{k}}(t)\,\mathrm{d}t\right)^{\frac{k}{k+1}} =\left(\int_{0}^{x}\frac{1}{1+t}\,\mathrm{d}t\right)^{\frac{n}{n+2}}=(\ln(1+x))^{\frac{n}{n+2}}
\end{equation*}
and also
\begin{equation*}
    F(x)=\int_{0}^{x}f(t)\varphi(t)\,\mathrm{d}t=\int_{0}^{x}\psi^{\prime}(t)\,\mathrm{d}t=\psi(x).
\end{equation*}
Finally, by setting
\begin{equation*}
    N(s)=\mathrm{e}^{s^{\frac{n+2}{n}}}-1\quad\mbox{and}\quad \Phi(s)=\mathrm{e}^{-ns}
\end{equation*}
we obtain
\begin{equation*}
    N(G(t))=t\quad\mbox{and}\quad N(F(t))=\mathrm{e}^{|\psi(t)|^{\frac{n+2}{n}}}-1.
\end{equation*}
Thus, Theorem~\ref{Leckband} with these choices yields 
\begin{align*}
    \int_{0}^{1}r^{n-1}\mathrm{e}^{n\mathrm{e}^{c^{\frac{2}{n}}_n|v|^{\frac{n+2}{n}}}}\,\mathrm{d} r& = \int_{0}^{\infty}\mathrm{e}^{n\mathrm{e}^{|\psi|^{\frac{n+2}{n}}}-nt}\,\mathrm{d}t\\
    & =\mathrm{e}^{n}\int_{0}^{\infty}\mathrm{e}^{n\mathrm{e}^{|\psi|^{\frac{n+2}{n}}}-nt-n}\,\mathrm{d}t
    \leq C\mathrm{e}^{n}\int_{0}^{\infty}\mathrm{e}^{-nt}\,\mathrm{d}t=C_n.
\end{align*}

\paragraph{\textit{Sharpness.}} For each $\ell\in\mathbb N$, define $\widetilde v_\ell\colon(0,1]\to\mathbb R$ by
\begin{equation}\label{velltilde}
\widetilde v_\ell(r)=-\left[c_n\ln\left(1+\ell\right)\right]^{-\frac{1}{k+1}}\left\{\begin{array}{ll}
\ln\left(1+\ell\right),&\mbox{if }0<r\leq \mathrm{e}^{-\ell},\\
\ln\left(\ln\frac{\mathrm{e}}{r}\right),&\mbox{if }\mathrm{e}^{-\ell}\leq r\leq1.
\end{array}\right.
\end{equation}
A straightforward computation shows that
\begin{equation*}
c_n\int_0^1r^{n-k}w_1(r)|\widetilde v_\ell^{\prime}|^{k+1}\,\mathrm{d} r=1.
\end{equation*}
Moreover, for $a>n$, we have
\begin{equation*}
\int_0^1r^{n-1}\mathrm{e}^{a\mathrm{e}^{c_n^{\frac{2}{n}}|\widetilde v_\ell|^{\frac{n+2}{n}}}}\,\mathrm{d} r\geq\int_0^{\mathrm{e}^{-\ell}}r^{n-1}\mathrm{e}^{a\mathrm{e}^{c_n^{\frac{2}{n}}|\widetilde v_\ell|^{\frac{n+2}{n}}}}\,\mathrm{d} r= \mathrm{e}^{a\mathrm{e}^{\ln(1+\ell)}}\int_0^{\mathrm{e}^{-\ell}}r^{n-1}\,\mathrm{d} r=\dfrac{\mathrm{e}^a}{n}\mathrm{e}^{\ell(a-n)}
\end{equation*}
goes to infinity as $\ell\to\infty$. This demonstrates that $a=n$ is the sharp value, thus concluding the proof of the proposition.
\end{proof}
\section{Log-weighted Trudinger-Moser inequalities for \texorpdfstring{$k$}{}-Hessian equation}\label{section4}
In this section we prove Theorems \ref{thm1}, \ref{thm1B}, and \ref{thm2}. As previously mentioned, the main idea is to transport the problem from $\Phi_{0,\mathrm{rad}}^k(B,w)$ to $X^{1,k+1}_{1,w}$ and apply the inequalities established for $X^{1,k+1}_{1,w}$, as given in Proposition~\ref{prop-thm1}, Corollary~\ref{corollary-radial}, and Proposition~\ref{prop-thm2}. For the sharpness condition, since the sequence of functions used for the space $X^{1,k+1}_{1,w}$ either has a singularity at the origin or is not $C^2$, we need to apply a mollification argument to construct a suitable sequence of functions in $\Phi^k_{0,\mathrm{rad}}(B,w)$.
\subsection{Proof of Theorem~\ref{thm1}}
The proof of sufficiency part follows directly from Proposition~\ref{prop-thm1} together with the identities \eqref{gate-norm} and \eqref{gate-fuctional}. In order to show the sharpness conditions we will use a mollification of the functions $v$ in \eqref{v-optimalMoser}, $v_{\ell}$ in \eqref{velloptimalMoser}, and $\overline v_\ell$ in \eqref{vellbar} to build suitable sequences of $k$-admissible functions in $\Phi^{k}_{0,\mathrm{rad}}(B,w)$.

\smallskip

\paragraph{\textit{Sharpness (a).}} Consider the function $u(x):=v(|x|)$, where $v$ is given in \eqref{v-optimalMoser}. This function satisfies
\begin{equation*}
\int_B\mathrm{e}^{|u|^\gamma}\mathrm dx=\omega_{N-1}\int_0^1r^{n-1}\mathrm{e}^{|v|^\gamma}\mathrm dx.
\end{equation*}
By Proposition \ref{prop-thm1}, item $(a)$, it remains to verify that $u\in\Phi_{0,\mathrm{rad}}^k(B,w)$. Observe that $v$ satisfies $v(1)=0$ and the inequality
\begin{align}
\frac{(v^{\prime})^{j-1}}{r^{j-1}}&\left(v^{\prime\prime}+\frac{n-j}{j}\frac{v^{\prime}}{r}\right)\nonumber\\
&=\frac{(v^{\prime})^{j-1}}{r^{j-1}}\left(\dfrac{1}{\gamma_{n,\beta}}-\eta\right)\left(n\ln\frac1r\right)^{\frac{1}{\gamma_{n,\beta}}-\eta-2}\dfrac{n^2}{r^2}\left[-\left(\dfrac{1}{\gamma_{n,\beta}}-\eta-1\right)+\dfrac{n-2j}j\ln\frac1r\right]\nonumber\\
&\geq\dfrac{(v^{\prime}_\ell)^{j-1}}{r^{j-1}}C_{n}>0\label{q1a}
\end{align}
for $r\in(0,\mathrm{e}^{-\frac{1}{n}})$ and $j=1,\ldots,k$, as well as
\begin{equation}
\frac{(v^{\prime})^{j-1}}{r^{j-1}}\left(v^{\prime\prime}+\frac{n-j}{j}\frac{v^{\prime}}{r}\right)=\dfrac{(v^{\prime})^{j-1}}{r^{j-1}}\dfrac{n-2j}j\dfrac{n}{r^2}\geq0\label{q2a}
\end{equation}
for $r\in(\mathrm{e}^{-\frac{1}{n}},1)$ and $j=1,\ldots,k$. Up to a positive constant, the left-hand expressions in \eqref{q1a} and \eqref{q2a} represent the $j$-Hessian operator $S_j(D^2u)$; see \cite[Section 3]{zbMATH06712355}.

Let $\epsilon=\epsilon(n)>0$ be a parameter to be chosen later, and consider a mollifier $\varphi_\epsilon\in C^\infty(\mathbb R)$ with $\mathrm{supp}(\varphi_\epsilon)\subset(-\epsilon,\epsilon)$. We define $\widehat V_\ell$, a smooth approximation of $v$ on $[\mathrm{e}^{-\frac{\ell}{n}},1]$, as
\begin{equation*}
\widehat V_\ell(r)=\left\{\begin{array}{ll}
     v(\mathrm{e}^{-\frac{\ell}n})&\mbox{if }r\in[0,\mathrm{e}^{-\frac\ell n}]  \\
     \varphi_\epsilon*v(\widehat\eta_\epsilon(r))&\mbox{if }r\in[\mathrm{e}^{-\frac{\ell}n},\mathrm{e}^{-\frac1n}]\\
     v(r)&\mbox{if }r\in[\mathrm{e}^{-\frac1n},1],
\end{array}\right.
\end{equation*}
where $\widehat\eta_\epsilon(r)$ is a linear function such that $\widehat\eta_\epsilon(\mathrm{e}^{-\frac{\ell}{n}})=\mathrm{e}^{-\frac{\ell}{n}}-\epsilon$ and $\widehat\eta_{\epsilon}(\mathrm{e}^{-\frac1n})=\mathrm{e}^{-\frac1n}+\epsilon$. Since $v\in X^{1,k+1}_{1,w}$, we can choose $\epsilon>0$ small enough such that $\|\widehat V_\ell-v\|_w\leq\frac1\ell$. Defining $\widehat U_\ell(x):=\widehat V_\ell(|x|)$, we note that $\widehat U_\ell\to u$ with respect to the norm $\|\cdot\|_{\Phi,w}$ given in \eqref{normPhi}. 

To establish that $u\in\Phi^k_{0,\mathrm{rad}}(B,w)$ (recall the definition in \eqref{defPhi}), it remains to prove that $S_j(D^2\widehat U_\ell)\geq0$ for all $j=1,\ldots,k$, which is equivalent to show that
\begin{equation}\label{cl1a}
\widehat V_\ell^{\prime\prime}(r)+\dfrac{\widehat V_\ell^{\prime}(r)}r\geq0,\quad\forall r\in[0,1].
\end{equation}
This inequality follows immediately for $r\in[0,\mathrm{e}^{-\frac{\ell}{n}}]$, and it holds for $r\in[\mathrm{e}^{-\frac1n},1]$ due to \eqref{q2a}. To verify it for for $r\in[\mathrm{e}^{-\frac{\ell}{n}},\mathrm{e}^{-\frac1n}]$, we compute
\begin{align*}
\widehat V_\ell^{\prime\prime}+\dfrac{\widehat V_\ell^{\prime}}r&=\left(1+\dfrac{2\epsilon}{\mathrm{e}^{-\frac1n}-\mathrm{e}^{-\frac\ell n}}\right)\int_{r-\epsilon}^{r+\epsilon}\varphi_{\epsilon}(r-s)\!\left[\widehat v_\ell^{\prime\prime}(\widehat \eta_\epsilon(s))\!\left(1+\dfrac{2\epsilon}{\mathrm{e}^{-\frac1n}-\mathrm{e}^{-\frac{\ell}{n}}}\right)\!+\dfrac{\widehat v_\ell^{\prime}(\widehat \eta_\epsilon(s))}{r}\right]\!\mathrm ds\\
&=\int_{\widehat \eta^{-1}_\epsilon(r-\epsilon)}^{\widehat \eta^{-1}_\epsilon(r+\epsilon)}\!\varphi_\epsilon(r-\widehat \eta_\epsilon^{-1}(t))\left[\widehat v_\ell^{\prime\prime}(t)\dfrac{2\epsilon}{\mathrm{e}^{-\frac1n}-\mathrm{e}^{-\frac\ell n}}+\dfrac{\widehat v_\ell^\prime(t)}t\dfrac{t-r}{r}+\widehat v_\ell^{\prime\prime}(t)+\dfrac{\widehat v_\ell^\prime(t)}{t}\right]\!\mathrm dt>0,
\end{align*}
where we used \eqref{q2a} and assumed $\epsilon=\epsilon(\ell,n)>0$ to be sufficiently small. This implies that \eqref{cl1a} holds, thus completing the proof that $u\in\Phi^k_{0,\mathrm{rad}}(B,w)$.

\smallskip

\paragraph{\textit{Sharpness (b).}}
First, we consider the case $w=w_0(r)=(\ln\frac1r)^{\frac{\beta n}2}$. For now, assume $\beta>0$. Let $(v_\ell)$ be the sequence of functions given by \eqref{velloptimalMoser}, and define $u_\ell(x)=v_\ell(|x|)$. Note that $v_\ell$ satisfies $v_\ell(1)=0$ and the equation
\begin{equation}\label{q11}
\frac{(v_\ell^{\prime})^{j-1}}{r^{j-1}}\left(v^{\prime\prime}_\ell+\frac{n-j}{j}\frac{v^{\prime}_\ell}{r}\right)=0
\end{equation}
for $r\in(0,\mathrm{e}^{-\frac{\ell}{n}})$ and $j=1,\ldots,k$, as well as
\begin{align}
\frac{(v_\ell^{\prime})^{j-1}}{r^{j-1}}\left(v^{\prime\prime}_\ell+\frac{n-j}{j}\frac{v^{\prime}_\ell}{r}\right)&=\dfrac{(v^{\prime}_\ell)^{j-1}}{r^{j-1}}\left(\dfrac{\ell}{\alpha_{n,\beta}}\right)^{\frac1{\gamma_{n,\beta}}}(1-\beta)\left(\dfrac{n}{\ell}\right)^{1-\beta}\nonumber\\
&\qquad\cdot\dfrac{\left(\ln\frac1r\right)^{-\beta-1}}{r^2}\left[\beta+\dfrac{n-2j}{j}\ln\frac1r\right]\geq \dfrac{(v^{\prime}_\ell)^{j-1}}{r^{j-1}}C_{\beta n\ell}>0\label{q21}
\end{align}
for $r\in(\mathrm{e}^{-\frac{\ell}{n}},1)$ and $j=1,\ldots,k$. Up to a positive constant, the left expression in \eqref{q11} and \eqref{q21} represents the $j$-Hessian operator $S_j(D^2u_\ell)$; see \cite[Section 3]{zbMATH06712355}.

However, despite \eqref{q11} and \eqref{q21}, $u_\ell$ is not $k$-admissible because $u_\ell\notin C^2(B)$ due to $v_\ell$ not being smooth at $\mathrm{e}^{-\frac\ell n}$. To address this, let $\epsilon=\epsilon(\ell,\beta,n)>0$ be a parameter to be fixed later, and consider a mollifier $\varphi_\epsilon\in C^\infty(\mathbb R)$ with $\mathrm{supp}(\varphi_\epsilon)\subset(-\epsilon,\epsilon)$. We define $V_\ell$ a smooth approximation of $v_\ell$ by
\begin{equation*}
V_\ell(r)=\left\{\begin{array}{ll}
     v_\ell(0)&\mbox{if }r\in[0,\mathrm{e}^{-\frac\ell n}]  \\
     \varphi_\epsilon*v_\ell(\eta_\epsilon(r))&\mbox{if }r\in[\mathrm{e}^{-\frac{\ell}n},1]\\
     0&\mbox{if }r\in[1,\infty),
\end{array}\right.
\end{equation*}
where $\eta_\epsilon(r)$ is a linear function such that $\eta_\epsilon(\mathrm{e}^{-\frac{\ell}{n}})=\mathrm{e}^{-\frac{\ell}{n}}-\epsilon$ and $\eta_{\epsilon}(1)=1+\epsilon$. We claim that for sufficiently small $\epsilon>0$, we have $\|V_\ell-v_\ell\|_{w_0}\leq\frac1\ell$ and
\begin{equation}\label{cl1}
V_\ell^{\prime\prime}(r)+\dfrac{V_\ell^{\prime}(r)}r\geq0,\quad\forall r\in[0,\infty).
\end{equation}
To establish \eqref{cl1}, it suffices to verify it for $r\in(\mathrm{e}^{-\frac{\ell}{n}},1)$. Note that
\begin{align*}
V_\ell^{\prime\prime}+\dfrac{V_\ell^{\prime}}r&=\left(1+\dfrac{2\epsilon}{1-\mathrm{e}^{-\frac\ell n}}\right)\int_{r-\epsilon}^{r+\epsilon}\varphi_{\epsilon}(r-s)\left[v_\ell^{\prime\prime}(\eta_\epsilon(s))\left(1+\dfrac{2\epsilon}{1-\mathrm{e}^{-\frac{\ell}{n}}}\right)+\dfrac{v_\ell^{\prime}(\eta_\epsilon(s))}{r}\right]\mathrm ds\\
&=\int_{\eta^{-1}_\epsilon(r-\epsilon)}^{\eta^{-1}_\epsilon(r+\epsilon)}\varphi_\epsilon(r-\eta_\epsilon^{-1}(t))\left[v_\ell^{\prime\prime}(t)\dfrac{2\epsilon}{1-\mathrm{e}^{-\frac\ell n}}+\dfrac{v_\ell^\prime(t)}t\dfrac{t-r}{r}+v_\ell^{\prime\prime}(t)+\dfrac{v_\ell^\prime(t)}{t}\right]\mathrm dt>0,
\end{align*}
where we used \eqref{q21} and assumed $\epsilon=\epsilon(\ell,\beta,n)>0$ to be sufficiently small. This establishes \eqref{cl1}.

From \eqref{cl1}, we conclude that $U_\ell(x):=V_\ell(|x|)$ is a $k$-admissible function. Moreover, for any $\alpha>\alpha_{n,\beta}$,
\begin{align*}
\int_B\mathrm e^{\alpha\left|\frac{U_\ell}{\|U_\ell\|_{\Phi,w_0}}\right|^{\gamma_{n,\beta}}}\mathrm dx&\geq\omega_{n-1}\int_0^{e^{-\frac{\ell}{n}}}r^{n-1}\exp\left(\dfrac{\alpha}{\|V_\ell\|_{w_0}^{\gamma_{n,\beta}}}\frac{\ell}{\alpha_{n,\beta}}\right)\mathrm dr\\
&=\dfrac{\omega_{n-1}}n\exp\left[\ell\left(\dfrac{\alpha}{\|V_\ell\|_{w_0}^{\gamma_{n,\beta}}\alpha_{n,\beta}}-1\right)\right]\to\infty
\end{align*}
as $\ell$ goes to $\infty$. This completes the proof for the case $w=w_0$ with $\beta>0$.

For the case $w=w_0$ with $\beta=0$, we consider the same sequence $(v_\ell)$ with $\beta>0$ fixed and note that $\|v_\ell\|_{w_0}=1+o_\beta(1)$ as $\beta\to0$. We then analogously construct $V_\ell$ and $U_\ell$, concluding that the supremum in \eqref{S-attainable} is infinite.

The case $w=w_1=(\ln\frac{e}{r})^{\frac{\beta n}2}$ is very similar to the case $w=w_0$, differing only in the choice of the sequence. For completeness, we include its proof. For now, assume $\beta>0$. Consider the sequence of functions $(\overline v_\ell)$ given by \eqref{vellbar}. To verify that $\overline u_\ell(x):=\overline v_\ell(|x|)$ is $k$-admissible function, it suffices to show that
\begin{equation*}
\overline v_\ell^{\prime\prime}(r)+\dfrac{\overline v_\ell^{\prime}(r)}{r}\geq0,\quad\forall r\in[0,1].
\end{equation*}
This inequality is trivial for $r\in[0,\mathrm{e}^{1-\frac{\ell}{n}})$. On the other hand, we have
\begin{equation}\label{cl2}
\overline v_\ell^{\prime\prime}(r)+\dfrac{\overline v_\ell^{\prime}(r)}{r}=C_0\beta\dfrac{\left(\ln\frac{\mathrm{e}}{r}\right)^{-\beta-1}}{r^2}\geq C_0\beta\min_{r\in(\mathrm{e}^{1-\frac{\ell}{n}},1)}\dfrac{\left(\ln\frac{\mathrm{e}}{r}\right)^{-\beta-1}}{r^2}>0,\quad\forall r\in(\mathrm{e}^{1-\frac{\ell}{n}},1),
\end{equation}
where
\begin{equation*}
C_0=\left(\frac{\ell}{\alpha_{n,\beta}}\right)^{\frac{1}{\gamma_{n,\beta}}}\left(\dfrac{\ell^{1-\beta}}{\ell^{1-\beta}-n^{1-\beta}}\right)^{\frac{1}{k+1}}(1-\beta)\left(\dfrac{n}{\ell}\right)^{1-\beta}.
\end{equation*}

However, $\overline u_\ell$ is not a $k$-admissible function because $\overline u_\ell\notin C^2(B)$, as $\overline v_\ell$ is not smooth at $\mathrm{e}^{1-\frac\ell n}$. To address this, let $\epsilon=\epsilon(\ell,\beta,n)>0$ be a small parameter to be determined later, and let $\varphi_\epsilon\in C^\infty(\mathbb R)$ be a mollifier with $\mathrm{supp}(\varphi_\epsilon)\subset(-\epsilon,\epsilon)$. We define $\overline V_\ell$ a smooth approximation of $\overline v_\ell$ by
\begin{equation*}
\overline V_\ell(r)=\left\{\begin{array}{ll}
     \overline v_\ell(0)&\mbox{if }r\in[0,\mathrm{e}^{1-\frac\ell n}]  \\
     \varphi_\epsilon*\overline v_\ell(\overline \eta_\epsilon(r))&\mbox{if }r\in[\mathrm{e}^{1-\frac{\ell}n},1]\\
     0&\mbox{if }r\in[1,\infty),
\end{array}\right.
\end{equation*}
where $\overline \eta_\epsilon(r)$ is a linear function satisfying $\overline \eta_\epsilon(\mathrm{e}^{1-\frac{\ell}{n}})=\mathrm{e}^{1-\frac{\ell}{n}}-\epsilon$ and $\overline \eta_{\epsilon}(1)=1+\epsilon$. We claim that for sufficiently small $\epsilon>0$, we have $\|\overline V_\ell-\overline v_\ell\|_{w_1}\leq\frac1\ell$ and
\begin{equation}\label{cl11}
\overline V_\ell^{\prime\prime}(r)+\dfrac{\overline V_\ell^{\prime}(r)}r\geq0,\quad\forall r\in[0,\infty).
\end{equation}
To establish \eqref{cl11}, it suffices to check it for $r\in(\mathrm{e}^{-\frac{\ell}{n}},1)$. Note that
\begin{align*}
\overline V_\ell^{\prime\prime}+\dfrac{\overline V_\ell^{\prime}}r&=\left(1+\dfrac{2\epsilon}{1-\mathrm{e}^{1-\frac\ell n}}\right)\int_{r-\epsilon}^{r+\epsilon}\varphi_{\epsilon}(r-s)\left[\overline v_\ell^{\prime\prime}(\overline \eta_\epsilon(s))\left(1+\dfrac{2\epsilon}{1-\mathrm{e}^{1-\frac{\ell}{n}}}\right)+\dfrac{\overline v_\ell^{\prime}(\overline \eta_\epsilon(s))}{r}\right]\mathrm ds\\
&=\int_{\overline \eta^{-1}_\epsilon(r-\epsilon)}^{\overline \eta^{-1}_\epsilon(r+\epsilon)}\varphi_\epsilon(r-\overline \eta_\epsilon^{-1}(t))\left[\overline v_\ell^{\prime\prime}(t)\dfrac{2\epsilon}{1-\mathrm{e}^{1-\frac\ell n}}+\dfrac{\overline v_\ell^\prime(t)}t\dfrac{t-r}{r}+\overline v_\ell^{\prime\prime}(t)+\dfrac{\overline v_\ell^\prime(t)}{t}\right]\mathrm dt>0,
\end{align*}
where we used \eqref{cl2} and assumed $\epsilon=\epsilon(\ell,\beta,n)>0$ to be sufficiently small. This completes the proof of \eqref{cl11}.

From \eqref{cl11}, we conclude that $\overline U_\ell(x):=\overline V_\ell(|x|)$ is a $k$-admissible function. Moreover, for any $\alpha>\alpha_{n,\beta}$,
\begin{align*}
\int_B\mathrm e^{\alpha\left|\frac{\overline U_\ell}{\|\overline U_\ell\|_{\Phi,w_1}}\right|^{\gamma_{n,\beta}}}\mathrm dx&\geq\omega_{n-1}\int_0^{e^{1-\frac{\ell}{n}}}r^{n-1}\exp\left[\dfrac{\alpha}{\|\overline V_\ell\|_{w_1}^{\gamma_{n,\beta}}}\frac{\ell}{\alpha_{n,\beta}}\left(\dfrac{\ell^{1-\beta}}{\ell^{1-\beta}-n^{1-\beta}}\right)^{-\frac{k\gamma_{n,\beta}}{k+1}}\right]\mathrm dr\\
&=\dfrac{\omega_{n-1}\mathrm{e}^n}n\exp\left[\ell\left(\dfrac{\alpha}{\|\overline V_\ell\|_{w_1}^{\gamma_{n,\beta}}\alpha_{n,\beta}}\left(\dfrac{\ell^{1-\beta}}{\ell^{1-\beta}-n^{1-\beta}}\right)^{-\frac{k\gamma_{n,\beta}}{k+1}}-1\right)\right]\to\infty
\end{align*}
as $\ell$ goes to $\infty$. This completes the proof for $w=w_1$ and $\beta>0$.

For the case $w=w_1$ and $\beta=0$, we consider the same $(\overline v_\ell)$ with $\beta>0$ fixed and note that $\|\overline v_\ell\|_{w_1}=1+o_\beta(1)$ as $\beta\to0$. Then, we construct $\overline V_\ell$ and $\overline U_\ell$ analogously, concluding that the supremum in \eqref{S-attainable} is infinite.
\subsection{Proof of Theorem~\ref{thm1B}}
 It follows directly from Corollary~\ref{corollary-radial}. Indeed, if $u\in \Phi^{k}_{0,\mathrm{rad}}(B, w_1)$ by setting $v(r)=u(|x|)$ with $r=|x|$, we have $v\in X^{1,k+1}_{1,w_1} $
 and from Corollary~\ref{corollary-radial}-$(b)$, there exists $C>0$ such that 
\begin{equation*}
    |v(r)|\leq \frac{1}{c^{\frac{2}{n+2}}_n|1-\beta|^{\frac{n}{n+2}}}\left|\left(\ln \frac{\mathrm{e}}{r}\right)^{1-\beta}-1\right|^{\frac{n}{n+2}}\|v\|_{w_1}\leq C\|v\|_{w_1}
\end{equation*}
for all $r\in (0,1)$, provided that $\beta>1$. Hence, from \eqref{gate-norm} we get \begin{equation*}\|u\|_{\infty}\leq C\|v\|_{w_1}=C\|u\|_{w_1}.\end{equation*}

\subsection{Proof of Theorem~\ref{thm2}}
The proof of sufficiency part follows directly from Proposition~\ref{prop-thm2} together with the identities \eqref{gate-norm} and \eqref{gate-fuctional}. In order to show the sharpness conditions we will use a mollification of the functions $\widetilde v_\ell$ in \eqref{velltilde} to build suitable sequences of $k$-admissible functions in $\Phi^{k}_{0,\mathrm{rad}}(B,w)$.

\paragraph{\textit{Sharpness.}} Let $(\widetilde v_\ell)$ be the sequence of functions given by \eqref{velltilde}, and define $\widetilde u_\ell(x)=\widetilde v_\ell(|x|)$. Note that $\widetilde v_\ell$ satisfies $\widetilde v_\ell(1)=0$ and the equation
\begin{equation}\label{q111}
\frac{(\widetilde v_\ell^{\prime})^{j-1}}{r^{j-1}}\left(\widetilde v^{\prime\prime}_\ell+\frac{n-j}{j}\frac{\widetilde v^{\prime}_\ell}{r}\right)=0
\end{equation}
for $r\in(0,\mathrm{e}^{-\ell})$ and $j=1,\ldots,k$, as well as
\begin{align}
\frac{(\widetilde v_\ell^{\prime})^{j-1}}{r^{j-1}}\left(\widetilde v^{\prime\prime}_\ell+\frac{n-j}{j}\frac{\widetilde v^{\prime}_\ell}{r}\right)&=\dfrac{(\widetilde v^{\prime}_\ell)^{j-1}}{r^{j-1}}\dfrac{\left[c_n\ln(1+\ell)\right]^{-\frac1{k+1}}}{r^2\ln\frac{\mathrm{e}}{r}}\left(\dfrac{1}{\ln\frac{\mathrm{e}}{r}}+\frac{n-2j}{j}\right)\nonumber\\
&\geq\dfrac{(\widetilde v^{\prime}_\ell)^{j-1}}{r^{j-1}}C_{n\ell}>0\label{q211}
\end{align}
for $r\in(\mathrm{e}^{-\ell},1)$ and $j=1,\ldots,k$. Up to a positive constant, the left expression in \eqref{q111} and \eqref{q211} represents the $j$-Hessian operator $S_j(D^2\widetilde u_\ell)$; see \cite[Section 3]{zbMATH06712355}.

However, despite \eqref{q111} and \eqref{q211}, $\widetilde u_\ell$ is not $k$-admissible because $\widetilde u_\ell\notin C^2(B)$ due to $\widetilde v_\ell$ not being smooth at $\mathrm{e}^{-\ell}$. To address this, let $\epsilon=\epsilon(\ell,n)>0$ be a parameter to be fixed later, and consider a mollifier $\varphi_\epsilon\in C^\infty(\mathbb R)$ with $\mathrm{supp}(\varphi_\epsilon)\subset(-\epsilon,\epsilon)$. We define $\widetilde V_\ell$ a smooth approximation of $\widetilde v_\ell$ by
\begin{equation*}
\widetilde V_\ell(r)=\left\{\begin{array}{ll}
     \widetilde v_\ell(0)&\mbox{if }r\in[0,\mathrm{e}^{-\ell}]  \\
     \varphi_\epsilon*\widetilde v_\ell(\widetilde \eta_\epsilon(r))&\mbox{if }r\in[\mathrm{e}^{-\ell},1]\\
     0&\mbox{if }r\in[1,\infty),
\end{array}\right.
\end{equation*}
where $\widetilde \eta_\epsilon(r)$ is a linear function such that $\widetilde \eta_\epsilon(\mathrm{e}^{-\ell})=\mathrm{e}^{-\ell}-\epsilon$ and $\widetilde \eta_{\epsilon}(1)=1+\epsilon$. We claim that for sufficiently small $\epsilon>0$, we have $\|\widetilde V_\ell-\widetilde v_\ell\|_{w}\leq\frac1\ell$ and
\begin{equation}\label{cl111}
\widetilde V_\ell^{\prime\prime}(r)+\dfrac{\widetilde V_\ell^{\prime}(r)}r\geq0,\quad\forall r\in[0,\infty).
\end{equation}
To establish \eqref{cl111}, it suffices to verify it for $r\in(\mathrm{e}^{-\ell},1)$. Note that
\begin{align*}
\widetilde V_\ell^{\prime\prime}+\dfrac{\widetilde V_\ell^{\prime}}r&=\left(1+\dfrac{2\epsilon}{1-\mathrm{e}^{-\ell}}\right)\int_{r-\epsilon}^{r+\epsilon}\varphi_{\epsilon}(r-s)\left[\widetilde v_\ell^{\prime\prime}(\widetilde \eta_\epsilon(s))\left(1+\dfrac{2\epsilon}{1-\mathrm{e}^{-\ell}}\right)+\dfrac{\widetilde v_\ell^{\prime}(\widetilde \eta_\epsilon(s))}{r}\right]\mathrm ds\\
&=\int_{\widetilde \eta^{-1}_\epsilon(r-\epsilon)}^{\widetilde \eta^{-1}_\epsilon(r+\epsilon)}\varphi_\epsilon(r-\widetilde \eta_\epsilon^{-1}(t))\left[\widetilde v_\ell^{\prime\prime}(t)\dfrac{2\epsilon}{1-\mathrm{e}^{-\ell}}+\dfrac{\widetilde v_\ell^\prime(t)}t\dfrac{t-r}{r}+\widetilde v_\ell^{\prime\prime}(t)+\dfrac{\widetilde v_\ell^\prime(t)}{t}\right]\mathrm dt>0,
\end{align*}
where we used \eqref{q211} and assumed $\epsilon=\epsilon(\ell,n)>0$ to be sufficiently small. This establishes \eqref{cl111}.

From \eqref{cl111}, we conclude that $\widetilde U_\ell(x):=\widetilde V_\ell(|x|)$ is a $k$-admissible function. Moreover, for any $a>n$,
\begin{align*}
\int_B\mathrm{e}^{a\mathrm{e}^{c_n^{\frac2n}\left|\frac{\widetilde U_\ell}{\|\widetilde U_\ell\|_{\Phi,w}}\right|^{\frac{n+2}{n}}}}\mathrm dx&\geq\omega_{n-1}\int_0^{\mathrm{e}^{-\ell}}r^{n-1}\mathrm{e}^{a\mathrm{e}^{c_n^{\frac2n}\left|\frac{\widetilde V_\ell}{\|\widetilde V_\ell\|_{w}}\right|^{\frac{n+2}{n}}}}\mathrm dr=\mathrm{e}^{a\mathrm{e}^{\ln(1+\ell)\|\widetilde V_\ell\|_{w}^{-\frac{n+2}n}}}\int_0^{\mathrm{e}^{-\ell}}r^{n-1}\mathrm dr\\
&=\frac{1}{n}\exp\left(a(1+\ell)^{\|\widetilde V_\ell\|_{w}^{-\frac{n+2}n}}-n\ell\right)\\
&=\frac{1}{n}\exp\left[n\ell\left(\frac{a}{n}\frac{(1+\ell)^{\|\widetilde V_\ell\|_{w}^{-\frac{n+2}n}}}\ell-1\right)\right]\to\infty
\end{align*}
as $\ell$ goes to $\infty$. This completes the proof of the theorem.
\section{Attainability for transported extremal problem}\label{section5}
In order to prove Theorem~\ref{thm-extremal} we will first deal with the attainability of the Trudinger-Moser-type supremum $\mathcal{MT}(n,\alpha,\beta)$ in \eqref{PI-ub} for the critical case $\alpha=\alpha_{n,\beta}$ and $w(r)=w_{\beta}(r)=(\ln \frac{1}{r})^{\frac{\beta n}{2}}$. The main difficulty here is that the Moser type functional on $X^{1,k+1}_{1,w_{\beta}}$ given by
\begin{equation}\label{Moser-functional}
    J_{n,\beta}(v)=\int_{0}^{1}r^{n-1}\mathrm{e}^{\alpha_{n,\beta}|v|^{\gamma_{n,\beta}}}\,\mathrm{d}r,\quad v\in X^{1,k+1}_{1,w_{\beta}}
\end{equation}
is subject to the loss of compactness arising from critical growth. In fact, for the optimal sequence $(v_\ell)\subset X^{1,k+1}_{1,w_{\beta}} $ defined in \eqref{velloptimalMoser}, we observe that $v_\ell\rightharpoonup 0$ weakly in $X^{1,k+1}_{1,w_{\beta}}$ and, by the item (ii) of Corollary \ref{thm0}, we have $v_\ell\to0$ in $L^{k+1}_{n-1}(0,1)$. However, $J_{n,\beta}(v_{\ell})\not\to J_{n,\beta}(0)$, as $J_{n,\beta}(v_\ell)>(2-\mathrm{e}^{-\ell})/n>1/n=J_{n,\beta}(0)$ for $\ell\geq1$.
To address this difficulty, we establish a Lions-type result (see Lemma~\ref{Lions-improvement} below) which will be combined with both the monotonic property and the upper bound estimate for the concentration levels of the Moser-type functional \eqref{Moser-functional}. Indeed, following \cite{zbMATH07931768,zbMATH07238535} we say that $(v_j)\subset X^{1,k+1}_{1,w_{\beta}}$ is a normalized concentrating sequence at origin, NCS for short, if for each $r_0\in(0,1)$ we have
\begin{equation*}
\|v_j\|_{w_{\beta}}=1,\;\; v_j\rightharpoonup 0\;\;\mbox{weakly in}\;\;X^{1,k+1}_{1,w_{\beta}} \;\;\mbox{and}\;\; \lim_{j\rightarrow \infty}\int_{r_0}^{1}r^{n-k}w_\beta(r)|v^{\prime}_j|^{k+1}\mathrm{d}r=0.
\end{equation*}
The concentration level $J^{\delta}_{n,\beta}(0)$ of the Moser-type functional $J_{n,\beta}$ on the set of all NCS is defined by 
\begin{equation*}
J^{\delta}_{n,\beta}(0)=\sup\left\{\limsup_{j\rightarrow\infty} J_{n,\beta}(v_j)\;:\; (v_j)\;\;\mbox{is NCS}  \right\}.
\end{equation*}
From \cite[Lemma~3]{zbMATH07931768}, we know the following estimate for the concentration level $J^{\delta}_{n,0}(0)$
\begin{equation}\label{b-c} 
J^{\delta}_{n,0}(0) \leq \frac{1}{n}\Big(1+\mathrm{e}^{\Psi(\frac{n}{2}+1)+\gamma}\Big)
\end{equation}
where $\Psi(x)={\Gamma}^{\prime}(x)/\Gamma(x)$ with $\Gamma(x)=\int_{0}^{1}(-\ln t)^{x-1}\,\mathrm{d}t$, $x>0$ is the digamma function and $\gamma=-\Psi(1)$ is the Euler-Mascheroni constant.

On the other hand, by arguing as in \cite{zbMATH07238535,zbMATH07931768}, we are able to obtain the following strict estimate:
\begin{lemma}\label{S-classic} It holds that
    \begin{equation*}
        \mathcal{MT}(n,\alpha_{n,0},0)>\frac{1}{n}\Big(1+\mathrm{e}^{\Psi(\frac{n}{2}+1)+\gamma}\Big).
    \end{equation*}
    \end{lemma}
    \begin{proof}
  Indeed, in the proof of Lemma~6 in \cite{zbMATH07931768}, this inequality arises as part of the strict estimate \cite[Equation (50)]{zbMATH07931768} established therein.
    \end{proof}
    To analyze the behavior of both $\mathcal{MT}(n,\alpha,\beta)$ and $J^{\delta}_{n,\beta}(0)$ with respect to $\beta$, we introduce a clever change of variables, as suggested in \cite{zbMATH07122708}. For $v\in X^{1,k+1}_{1,w_{\beta}}$ with $\|v\|_{w_{\beta}}\leq 1$ and for $0\leq \tilde{\beta}<\beta$, we denote
    \begin{equation}\label{changeV1}
        z(r)=\left(\frac{\alpha_{n,\beta}}{\alpha_{n,\tilde{\beta}}}\right)^{\frac{n(1-\tilde{\beta})}{n+2}}v(r)|v(r)|^{\frac{\beta-\tilde{\beta}}{1-\beta}}.
    \end{equation}
    \begin{lemma}\label{C-change}
       If $v\in X^{1,k+1}_{1,w_{\beta}}$ and $z$ is defined by \eqref{changeV1}, then $\|z\|^{\frac{1}{1-\tilde{\beta}}}_{w_{\tilde{\beta}}}\le\|v\|^{\frac{1}{1-\beta}}_{w_{\beta}}.$
    \end{lemma}
    \begin{proof}
        From Lemma~\ref{r-estimate} with $t=1$ and $w(r)=w_{\beta}(r)=(\ln \frac{1}{r})^{\frac{\beta n}{2}}$, we have
        \begin{equation*}
\begin{aligned}
 |v(r)| &\leq \left(\frac{1}{c^{\frac{2}{n}}_n}\int_{r}^{1}\frac{1}{s}(-\ln s)^{-\beta}\,\mathrm{d}s\right)^{\frac{n}{n+2}}\left(c_n\int_{r}^{1}s^{n-k}|v^{\prime}(s)|^{k+1}w_{\beta}(s)\,\mathrm{d}s\right)^{\frac{2}{n+2}}\\
&=\left(\frac{n}{\alpha_{n,\beta}}\right)^{\frac{n(1-\beta)}{n+2}}w^{\frac{2(1-\beta)}{\beta(n+2)}}_{\beta}(r)\left(c_n\int_{r}^{1}s^{n-k}|v^{\prime}(s)|^{k+1}w_{\beta}(s)\,\mathrm{d}s\right)^{\frac{2}{n+2}}
\end{aligned}
\end{equation*}
for any $r\in (0,1)$. Thus,

\begin{align*}
    |z^{\prime}(r)|^{k+1}&=\left|\left(\frac{\alpha_{n,\beta}}{\alpha_{n,\tilde{\beta}}}\right)^{\frac{n(1-\tilde{\beta})}{n+2}}\frac{1-\tilde{\beta}}{1-\beta}|v(r)|^{\frac{\beta-\tilde{\beta}}{1-\beta}}v^{\prime}(r)\right|^{\frac{n+2}{2}}\\
    &=\left(\frac{\alpha_{n,\beta}}{\alpha_{n,\tilde{\beta}}}\right)^{\frac{n(1-\tilde{\beta})}{2}}\left(\frac{1-\tilde{\beta}}{1-\beta}\right)^{\frac{n+2}{2}}|v^{\prime}(r)|^{k+1}|v(r)|^{\frac{n+2}{2}\frac{\beta-\tilde{\beta}}{1-\beta}}\\
    &\leq \left(\frac{\alpha_{n,\beta}}{\alpha_{n,\tilde{\beta}}}\right)^{\frac{n(1-\tilde{\beta})}{2}}\left(\frac{1-\tilde{\beta}}{1-\beta}\right)^{\frac{n+2}{2}}\left(\left(\frac{n}{\alpha_{n,\beta}}\right)^{\frac{n(1-\beta)}{n+2}}w^{\frac{2(1-\beta)}{\beta(n+2)}}_{\beta}(r)\right)^{\frac{n+2}{2}\frac{\beta-\tilde{\beta}}{1-\beta}}\\
    &\quad\times \left(c_n\int_{r}^{1}s^{n-k}|v^{\prime}(s)|^{k+1}w_{\beta}(s)\,\mathrm{d}s\right)^{\frac{\beta-\tilde{\beta}}{1-\beta}}|v^{\prime}(r)|^{k+1}\\
    &=\left(\frac{1-\tilde{\beta}}{1-\beta}\right)|v^{\prime}(r)|^{k+1}\frac{w_{\beta}(r)}{w_{\tilde{\beta}}(r)}\left(c_n\int_{r}^{1}s^{n-k}|v^{\prime}(s)|^{k+1}w_{\beta}(s)\,\mathrm{d}s\right)^{\frac{\beta-\tilde{\beta}}{1-\beta}}.
\end{align*}    
Consequently,
\begin{equation}\nonumber
    \begin{aligned}
\|z\|^{k+1}_{w_{\tilde{\beta}}}&=c_{n}\int_{0}^{1}r^{n-k}|z^{\prime}(r)|^{k+1}w_{\tilde{\beta}}(r)\,\mathrm{d} r\\
&\leq \left(\frac{1-\tilde{\beta}}{1-\beta}\right)c_n\int_{0}^{1}r^{n-k}|v^{\prime}(r)|^{k+1}w_{\beta}(r)\left(c_n\int_{r}^{1}s^{n-k}|v^{\prime}(s)|^{k+1}w_{\beta}(s)\,\mathrm{d}s\right)^{\frac{\beta-\tilde{\beta}}{1-\beta}}\,\mathrm{d} r\\
&=-c^{\frac{1-\tilde{\beta}}{1-\beta}}_n\int_{0}^{1}\frac{d}{\,\mathrm{d} r}\left(\left(\int_{r}^{1}s^{n-k}|v^{\prime}(s)|^{k+1}w_{\beta}(s)\,\mathrm{d}s\right)^{\frac{1-\tilde{\beta}}{1-\beta}}\right)\,\mathrm{d} r\\
&=c^{\frac{1-\tilde{\beta}}{1-\beta}}_n\left(\int_{0}^{1}r^{n-k}|v^{\prime}(r)|^{k+1}w_{\beta}(r)\,\mathrm{d} r\right)^{\frac{1-\tilde{\beta}}{1-\beta}}\\
&=\|v\|^{(k+1)\frac{1-\tilde{\beta}}{1-\beta}}_{w_{\beta}}.
    \end{aligned}
\end{equation}
    \end{proof}   
    \begin{lemma}\label{monotonicityMT} The function $\beta\mapsto \mathcal{MT}(n,\alpha_{n,\beta},\beta)$ is decreasing on $[0,1)$.
\end{lemma}
This completes the proof of the lemma.
\begin{proof}
For $0\le\tilde{\beta}<\beta$ and $v\in X^{1,k+1}_{1,w_{\beta}}$ with $\|v\|_{w_{\beta}}\leq 1$, let $z$ be the function given by \eqref{changeV1}. From Lemma~\ref{C-change} we have $\|z\|_{w_{\tilde{\beta
}}}\leq 1$. In addition, 
\begin{equation}\nonumber
\int_{0}^{1}r^{n-1}\mathrm{e}^{\alpha_{n,\beta}|v|^{\gamma_{n,\beta}}}\,\mathrm{d} r=\int_{0}^{1}r^{n-1}\mathrm{e}^{\alpha_{n,\tilde{\beta}}|z|^{\gamma_{n,\tilde{\beta}}}}\,\mathrm{d} r\leq  \mathcal{MT}(n,\alpha_{n,\tilde{\beta}},\tilde{\beta}).
\end{equation}
By the arbitrariness of the choice of $v\in X^{1,k+1}_{1,w_{\beta}}$ with $\|v\|_{w_{\beta}}\leq 1$, it follows that
\begin{equation*}
\mathcal{MT}(n,\alpha_{n,\beta},\beta)\leq \mathcal{MT}(n,\alpha_{n,\tilde{\beta}},\tilde{\beta}). 
\end{equation*}
This completes the proof of the lemma.
\end{proof}
\begin{lemma}\label{Lions-improvement} Let $(v_j)$ be a sequence in $X^{1,k+1}_{1,w_{\beta}}$ such that $\|v_j\|_{w_{\beta}}=1$ and $v_j\rightharpoonup v_0$ weakly in $ X^{1,k+1}_{1,w_{\beta}}$. Then
\begin{equation}\label{D-improvement}
    \limsup_{j\to\infty} \int_{0}^{1}r^{n-1}\mathrm{e}^{p\alpha_{n,\beta}|v_j|^{\gamma_{n,\beta}}}\,\mathrm{d}r<\infty
\end{equation}
for any $0<p<P_{v_0}:=(1-\|v_0\|^{k+1}_{w_{\beta}})^{-\frac{2}{n(1-\beta)}}$.
\end{lemma}
\begin{proof}
Taking \eqref{PI-ub} into account, we may assume $v_0\not\equiv 0$. In addition, as a consequence of the item (ii) of Corollary~\ref{thm0} we have the compact embedding 
\begin{equation*}
 X^{1,k+1}_{1,w_{\beta}}\hookrightarrow L^q_{n-1},\mbox{ for all }1\leq q<\infty.  
\end{equation*} 
Thus, up to a subsequence, we can suppose that $v_j\to v_0$ in $L^{q}_{n-1}$ and $v_j(r)\to v_0(r)$ a.e in $(0,1)$. Now, given $L>0$, let us define the truncations given by
\begin{equation*}
T^{L}v_j=\min\left\{|v_j|,L\right\}\mbox{sign}(v_j)\quad\mbox{and}\quad
T_{L}{v_j}= v_j-T^{L}v_j,\;\;\; j=0,1,2\cdots.
\end{equation*}
Note that 
\begin{equation}\label{I1}
\|v_j\|^{k+1}_{w_{\beta}}=\|T^{L}v_j\|^{k+1}_{w_{\beta}}+\|T_{L}v_j\|^{k+1}_{w_{\beta}}, \;\;\; j=0,1,2, \cdots .
\end{equation}
In particular, $\|T^{L}v_0\|^{k+1}_{w_{\beta}}\to \|v_0\|^{k+1}_{w_{\beta}}$, as $L\to\infty$. So, for any $p<P_{v_0}$, we can choose \begin{equation*}p_0:=p(1-\|T^{L}v_0\|^{k+1}_{w_{\beta}})^{\frac{2}{n(1-\beta)}}<1,\end{equation*} for $L$ large enough. Since $
T^{L}v_j(r)\rightarrow T^{L}v_0(r)$ a.e in $(0,1)$ and, from \eqref{I1}, $(T^{L}v_j)$ is bounded in $X^{1,k+1}_{1,w_{\beta}}$, up to a subsequence, we can assume that
$T^{L}v_j\rightharpoonup T^{L}v_0$ weakly in $X^{1,k+1}_{1,w_{\beta}}$. Then, the weak lower semi-continuity of the norm provides 
\begin{equation*}
\|T^{L}v_0\|^{k+1}_{w_{\beta}}\le\liminf_{j\to\infty} \|T^{L}v_j\|^{k+1}_{w_{\beta}}
\end{equation*}
and 
\begin{equation*}
\limsup_{j\to\infty}\|T_{L}v_j\|^{k+1}_{w_{\beta}}=1-\liminf_{j\to\infty} \|T^{L}v_j\|^{k+1}_{w_{\beta}}\leq 1-\|T^{L}v_0\|^{k+1}_{w_{\beta}}.
\end{equation*}
Then, there exists $j_0\in\mathbb{N}$ such that 
\begin{equation}\label{TL-norm}
p\|T_{L}v_j\|^{\gamma_{n,\beta}}_{w_{\beta}}\leq p\left(1-\|T^{L}v_0\|^{k+1}_{w_{\beta}}\right)^\frac{\gamma_{n,\beta}}{k+1}\leq \frac{p_0+1}{2}<1,
\end{equation}
for all $j\ge j_0$. Of course, we have $v_j=T_{L}v_j+T^{L}v_j$ and $|T^{L}v_j|\leq L$. Then, using the convexity of the function $x\mapsto x^q$, $x\ge 0$ with $q>1$ (cf. \cite[Lemma~2.1]{zbMATH07714682}), for $\epsilon >0$
		\begin{equation*}
			(x+y)^q \leq (1+\epsilon)^{\frac{q-1}{q}}x^q+(1-(1+\epsilon)^{-\frac{1}{q}})^{1-q}y^q,\;\; x, y\ge 0.    
		\end{equation*}
It follows that 
\begin{equation}\label{e-convex}
    \begin{aligned}
        |v_j|^{\gamma_{n,\beta}} &\leq (|T_{L}v_j|+|T^{L}v_j|)^{\gamma_{n,\beta}}\\
        &\leq (1+\epsilon)^{\frac{\gamma_{n,\beta}-1}{\gamma_{n,\beta}}}\|T_{L}v_j\|^{\gamma_{n,\beta}}_{w_{\beta}}\Big|\frac{T_{L}v_j}{\|T_{L}v_j\|_{w_{\beta}}}\Big|^{\gamma_{n,\beta}}+(1-(1+\epsilon)^{-\frac{1}{\gamma_{n,\beta}}})^{1-\gamma_{n,\beta}}L^{\gamma_{n,\beta}}.
    \end{aligned}
\end{equation} 
Hence, by choosing $\epsilon>0$ small such that $(1+\epsilon)^{\frac{\gamma_{n,\beta}-1}{\gamma_{n,\beta}}}\frac{p_0+1}{2}<1$, from \eqref{TL-norm} and \eqref{e-convex} we obtain
\begin{equation}\label{ee-convex}
    \begin{aligned}
        \alpha_{n,\beta}p|v_j|^{\gamma_{n,\beta}} &\leq \alpha_{n,\beta}\Big|\frac{T_{L}v_j}{\|T_{L}v_j\|_{w_{\beta}}}\Big|^{\gamma_{n,\beta}}+C(p,n,\beta, L).
    \end{aligned}
\end{equation} 
Then, \eqref{D-improvement} follows from \eqref{PI-ub} and \eqref{ee-convex}. This completes the proof of the lemma.
\end{proof}

\begin{lemma}\label{monotonicity} The function $\beta\mapsto J ^{\delta}_{n,\beta}(0)$ is decreasing on $[0,1)$.
\end{lemma}
\begin{proof}
   Note that for any $\beta\in [0,1)$ and any normalized concentrating sequence $(v_j)\subset X^{1,k+1}_{1,w_{\beta}}$ at origin, we have 
   \begin{equation}\label{eq413}
       J^{\delta}_{n,\beta}(0)\ge \limsup_{j\to\infty}\int_{0}^{1}r^{n-1}\mathrm{e}^{\alpha_{n,\beta}|v_j|^{\gamma_{n,\beta}}}\,\mathrm{d}r\ge \frac{1}{n}.
   \end{equation}
   Without loss of generality, assume $\displaystyle\lim_{j\to\infty} J_{n,\beta}(v_j)$ exists. Fix $\widetilde\beta\in[0,\beta)$. Our goal is to prove that $J^\delta_{n,\beta}(0)\leq J^\delta_{n,\widetilde\beta}(0)$. More specifically, it suffices to show that
   \begin{equation}\label{ineqenough}
       \lim_{j\to\infty}J_{n,\beta}(v_j)\leq J^\delta_{n,\widetilde\beta}(0).
   \end{equation}
   Let $(z_j)$ be the sequence defined from $(v_j)$ via \eqref{changeV1}. By Lemma \ref{r-estimate}, $(z_j)$ and $(v_j)$ converge uniformly to zero on $(r_0,1)$ for any $r_0\in(0,1)$. Furthermore, by Lemma \ref{C-change}, we have $\|z_j\|_{w_{\widetilde\beta}}\leq1$. Therefore, up to subsequence, $z_j\rightharpoonup z$ in $X^{1,k+1}_{1,\omega_{\widetilde \beta}}$ for some $z\in X^{1,k+1}_{1,\omega_{\widetilde \beta}}$. From the item (ii) of Corollary~\ref{thm0}, it follows that $z_j\to z$ in $L^q_{n-1}(0,1)$ for all $1\leq q<\infty$. Since $(z_j)$ converges uniformly to zero on $(r_0,1)$, we conclude $z=0$. Thus, $z_j\rightharpoonup0$ in $X^{1,k+1}_{1,\omega_{\widetilde \beta}}$. 
   
   The proof will now be divided into two cases.

\smallskip

\noindent\underline{Case 1:} $\displaystyle\limsup_{j\to\infty}\|z_j\|_{w_{\widetilde\beta}}<1$.

\noindent In this case, there exist $a\in(0,1)$ and $j_0\in\mathbb N$ such that $\|z_j\|_{w_{\widetilde\beta}}\leq a$ for all $j\geq j_0$. Using the item (b) of Proposition \ref{prop-thm1}, we conclude that
\begin{equation}\label{asfoknkasd}
\sup_{j}\int_0^1r^{n-1}\mathrm{e}^{p\alpha_{n,\widetilde\beta}|z_j|^{\gamma_{n,\widetilde\beta}}}\mathrm \,\mathrm{d} r<\infty
\end{equation}
for all $1<p\leq 1/a^{\gamma_{n,\widetilde\beta}}$. Fix $1<p_0\leq 1/a^{\gamma_{n,\widetilde\beta}}$. Using the inequality $\mathrm{e}^x-1\leq x\mathrm{e}^x$, we obtain
\begin{align}
&\left|\int_0^1r^{n-1}\mathrm{e}^{\alpha_{n,\widetilde \beta}|z_j|^{\gamma_{n,\widetilde\beta}}}-\int_0^1r^{n-1}\mathrm \,\mathrm{d} r\right|\leq\int_0^1r^{n-1}\alpha_{n,\widetilde \beta}|z_j|^{\gamma_{n,\widetilde\beta}}\mathrm{e}^{\alpha_{n,\widetilde \beta}|z_j|^{\gamma_{n,\widetilde\beta}}}\mathrm \,\mathrm{d} r\nonumber\\
&\qquad\leq\alpha_{n,\widetilde\beta}\left(\int_0^1r^{n-1}|z_j|^{\frac{p_0\gamma_{n,\widetilde\beta}}{p_0-1}}\mathrm \,\mathrm{d} r\right)^{\frac{p_0-1}{p_0}}\left(\int_0^1r^{n-1}\mathrm{e}^{p_0\alpha_{n,\widetilde\beta}|z_j|^{\gamma_{n,\widetilde\beta}}}\mathrm \,\mathrm{d} r\right)^{\frac{1}{p_0}}\overset{j\to\infty}\longrightarrow 0,\label{expineq}
\end{align}
where we used that $z_j\to0$ in $L^q_{n-1}(0,1)$ for all $1\leq q<\infty$ and \eqref{asfoknkasd}. Thus,
\begin{equation*}
\lim_{j\to\infty}\int_0^1r^{n-1}\mathrm{e}^{\alpha_{n,\beta}|v_j|^{\gamma_{n,\beta}}}\mathrm \,\mathrm{d} r=\lim_{j\to\infty}\int_0^1r^{n-1}\mathrm{e}^{p_0\alpha_{n,\widetilde\beta}|z_j|^{\gamma_{n,\widetilde\beta}}}\mathrm \,\mathrm{d} r=\frac{1}{n}\leq J_{n,\widetilde\beta}^{\delta}(0).
\end{equation*}
Since $(v_j)$ was arbitrarily chosen, we conclude $J^\delta_{n,\beta}(0)\leq J^\delta_{n,\widetilde\beta}(0)$. This completes the proof of \eqref{ineqenough} for Case 1.

\smallskip

\noindent\underline{Case 2:} $\displaystyle\limsup_{j\to\infty}\|z_j\|_{w_{\widetilde\beta}}=1$.

\noindent In this case, up to subsequence, we assume $\displaystyle\lim_{j\to\infty}\|z_j\|_{w_{\widetilde\beta}}=1$. Define $\widetilde z_j=z_j/\|z_j\|_{w_{\widetilde\beta}}$. Note that $\widetilde z_j\rightharpoonup 0$ in $X^{1,k+1}_{1,w_{\widetilde\beta}}$ and $|z_j|\leq|\widetilde z_j|$. It suffices to consider the case in which there exist $a,r_0\in(0,1)$ such that for every $j_0$, there exist $j\geq j_0$ satisfying
\begin{equation}\label{eq415}
\int_0^{r_0}r^{n-k}w_{\widetilde\beta}(r)|\widetilde z^{\prime}_j|^{k+1}\mathrm \,\mathrm{d} r\leq a.
\end{equation}
Indeed, if \eqref{eq415} does not hold, then $(\widetilde z_j)$ is a normalized concentrating sequence (NCS). Using $|z_j|\leq|\widetilde z_j|$, we get
\begin{equation*}
\lim_{j\to\infty}J_{n,\beta}(v_j)=\lim_{j\to\infty}J_{n,\widetilde\beta}(z_j)\leq \limsup_{j\to\infty}J_{n,\widetilde\beta}(\widetilde z_j)\leq J_{n,\widetilde\beta}^\delta(0),
\end{equation*}
which concludes \eqref{ineqenough}.

Define the function $\overline z_j$ as
\begin{equation*}
\overline z_j(r)=\left\{\begin{array}{ll}
     \widetilde z_j(r)-\widetilde z_j(r_0),&\mbox{if }0<r<r_0,  \\
     0,&\mbox{if }r_0\leq r<1. 
\end{array}\right.
\end{equation*}
Note that $\overline z_j\in X^{1,k+1}_{1,w_{\widetilde\beta}}$ and by \eqref{eq415}, together with a passage to a subsequence, we have $\|\overline z_j\|^{k+1}_{w_{\widetilde\beta}}\leq a<1$. Fix $\epsilon>0$ small enough such that $(1+\epsilon)a^{\frac{1}{k(1-\widetilde\beta)}}<1$. Using the inequality
\begin{equation*}
|\widetilde z_j(r)|^{\gamma_{n,\widetilde\beta}}\leq (1+\epsilon)|\overline z_j(r)|^{\gamma_{n,\widetilde\beta}}+C(n,\widetilde\beta,\epsilon)|\widetilde z_j(r_0)|^{\gamma_{n,\widetilde\beta}},\quad\forall 0<r<r_0,
\end{equation*}
where $C(n,\widetilde\beta,\epsilon)=\left(1-(1+\epsilon)^{\frac{1}{1-\gamma_{n,\widetilde\beta}}}\right)^{1-\gamma_{n,\widetilde\beta}}$, and following an analogous argument to \eqref{expineq}, we deduce
\begin{equation}\label{eq417}
\lim_{j\to\infty}\int_0^{r_0}r^{n-1}\mathrm{e}^{\alpha_{n,\widetilde\beta}|\widetilde z_j|^{\gamma_{n,\widetilde\beta}}}\mathrm \,\mathrm{d} r=\frac{r_0^n}n.
\end{equation}
Since $(\widetilde z_j)$ converges uniformly to zero on $(r_0,1)$, it follows that
\begin{equation}\label{eq418}
\lim_{j\to\infty}\int_{r_0}^1r^{n-1}\mathrm{e}^{\alpha_{n,\widetilde\beta}|\widetilde z_j|^{\gamma_{n,\widetilde\beta}}}\mathrm \,\mathrm{d} r=\frac{1}n-\frac{r_0^n}n.
\end{equation}
By summing \eqref{eq417} and \eqref{eq418}, we obtain
\begin{equation*}
\lim_{j\to\infty}\int_0^{1}r^{n-1}\mathrm{e}^{\alpha_{n,\widetilde\beta}|\widetilde z_j|^{\gamma_{n,\widetilde\beta}}}\mathrm \,\mathrm{d} r=\frac{1}n.
\end{equation*}
Consequently, using \eqref{eq413}, we have
\begin{equation*}
\lim_{j\to\infty}J_{n,\beta}(v_j)=\lim_{j\to\infty}J_{n,\widetilde\beta}(z_j)\leq\lim_{j\to\infty}J_{n,\widetilde\beta}(\widetilde z_j)=\frac{1}{n}\leq J^\delta_{n,\widetilde\beta}(0).
\end{equation*}
This concludes \eqref{ineqenough} and, therefore, completes the proof of the lemma.
\end{proof}

\begin{lemma}\label{APPROX} It holds that
$
\mathcal{MT}(n,\alpha_{n,0},0)=\displaystyle\lim_{\beta\to 0}\mathcal{MT}(n,\alpha_{n,\beta},\beta).
$
\end{lemma}
\begin{proof}
From Lemma \ref{monotonicityMT}, we know that
\begin{equation*}
\limsup_{\beta\to0}\mathcal{MT}(n,\alpha_{n,\beta},\beta)\leq \mathcal{MT}(n,\alpha_{n,0},0).
\end{equation*}
To complete the proof, it suffices to show that
\begin{equation}\label{eq420}
\liminf_{\beta\to0}\mathcal{MT}(n,\alpha_{n,\beta},\beta)\geq\mathcal{MT}(n,\alpha_{n,0},0).
\end{equation}

Let $(v_{j,0})$ be a maximizing sequence for $\mathcal{MT}(n,\alpha_{n,0},0)$, i.e., $\|v_{j,0}\|_{X^{1,k+1}_1}=1$ and
\begin{equation}\label{eq421}
\lim_{j\to\infty}\int_0^1r^{n-1}\mathrm{e}^{\alpha_{n,0}|v_{j,0}|^{\gamma_{n,0}}}\mathrm \,\mathrm{d} r=\mathcal{MT}(n,\alpha_{n,0},0).
\end{equation}
For each $j$ and $\beta\in(0,1)$, define
\begin{equation*}
v_{j,\beta}(r)=-\int_r^1v_{j,0}^{\prime}(s)\left(\ln\frac1s\right)^{-\frac{\beta k}{k+1}}\,\mathrm ds.
\end{equation*}
To show that $v_{j,\beta}$ is well-defined, we estimate its integrand in the following way
\begin{align}
\int_r^1|v_{j,0}^{\prime}(s)|\left(\ln\frac1s\right)^{-\frac{\beta k}{k+1}}\,\mathrm ds&\leq\left(\int_{r}^1s^{n-k}|v_{j,0}^{\prime}|^{k+1}\,\mathrm ds\right)^{\frac1{k+1}}\left(\int_r^1s^{-1}\left(\ln\frac1s\right)^{-\beta}\,\mathrm ds\right)^{\frac{k}{k+1}}\nonumber\\
&= c_n^{-\frac1{k+1}}\left(\int_0^{\ln\frac1r}t^{-\beta}\mathrm \,\mathrm{d}t\right)^{\frac{k}{k+1}}=\frac{c_n^{-\frac1{k+1}}}{(1-\beta)^{\frac{k}{k+1}}}\left(\ln\frac{1}{r}\right)^{\frac{(1-\beta)k}{k+1}}.\label{eq422}
\end{align}
This ensures $v_{j,\beta}$ is well-defined. Moreover, we verify that $v_{j,\beta}\in X^{1,k+1}_{1,w_\beta}$ with $\|v_{j,\beta}\|_{w_\beta}=\|v_{j,0}\|_{X^{1,k+1}_1}=1$.

\begin{figure}[h]
    \centering
    \includegraphics[width=0.8\textwidth]{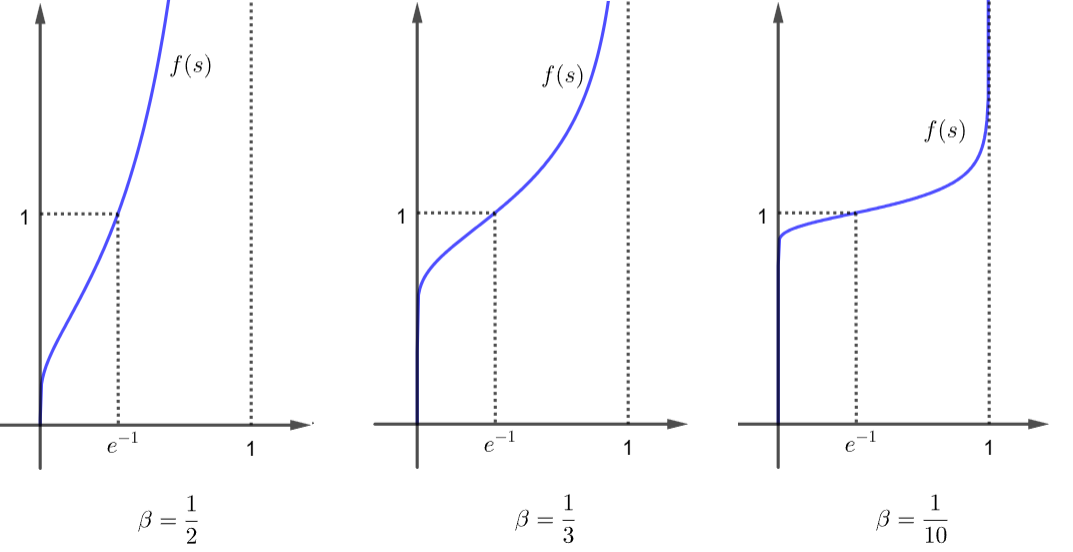}
    \caption{Intuition for the choice of $F$}
    \label{figure1}
    \end{figure}

We claim that, for each $j$,
\begin{equation}\label{eq423}
\lim_{\beta\to0}v_{j,\beta}(r)=v_{j,0}(r),\quad\forall r\in(0,1).
\end{equation}
Using the Dominated Convergence Theorem, it suffices to check that $|v_{j,0}^{\prime}(s)|F(s)$ is in $L^1(r,1)$ for each $r\in(0,1)$, where
\begin{equation*}
F(s)=\left\{\begin{array}{ll}
     1,&\mbox{if }s\leq \mathrm{e}^{-1},  \\
     \left(\ln\frac1s\right)^{-\frac{k}{2k+2}},&\mbox{if }s>\mathrm{e}^{-1}. 
\end{array}\right.
\end{equation*}
Figure \ref{figure1} highlights the motivation behind the definition of $F$, showing the behavior of $f(s)=(\ln\frac1s)^{-\frac{\beta k}{k+1}}$ as $\beta\to0$.
    
\noindent Applying \eqref{eq422}, we have
\begin{align*}
\int_r^1|v^{\prime}_{j,0}(s)|F(s)\,\mathrm ds&=\|v^{\prime}_{j,0}\|_{L^1(r,\mathrm{e}^{-1})}+\int_{\mathrm{e}^{-1}}^1|v^{\prime}_{j,0}(s)|\left(\ln\frac1s\right)^{-\frac{k}{2k+2}}\,\mathrm ds\\
&\leq \|v^{\prime}_{j,0}\|_{L^1(r,\mathrm{e}^{-1})}+\frac{c_n^{-\frac1{k+1}}}{(1-\beta)^{\frac{k}{k+1}}}.
\end{align*}
Since $v^{\prime}_{j,0}\in L^{k+1}_{n-k}(r,\mathrm{e}^{-1})\subset L^1(r,\mathrm{e}^{-1})$, we conclude that $|v^{\prime}_{j,0}(s)|F(s)\in L^1(r,1)$, establishing \eqref{eq423}.

Using Fatou's Lemma and \eqref{eq423}, we obtain
\begin{align*}
\liminf_{\beta\to0}\mathcal{MT}(n,\alpha_{n,\beta},\beta)&\geq\liminf_{\beta\to0}\int_0^1r^{n-1}\mathrm{e}^{\alpha_{n,\beta}|v_{j,\beta}|^{\gamma_{n,\beta}}}\mathrm \,\mathrm{d} r\geq \int_0^1r^{n-1}\liminf_{\beta\to0}\mathrm{e}^{\alpha_{n,\beta}|v_{j,\beta}|^{\gamma_{n,\beta}}}\mathrm \,\mathrm{d} r\\
&=\int_0^1r^{n-1}\mathrm{e}^{\alpha_{n,0}|v_{j,0}|^{\gamma_{n,0}}}\mathrm \,\mathrm{d} r.
\end{align*}
Since \eqref{eq421} and the inequality above holds for all $j$, we conclude \eqref{eq420}, thereby completing the proof of the lemma.
\end{proof}
Now, we are in a position to prove that the Trudinger-Moser-type supremum $\mathcal{MT}(n,\alpha,\beta)$ in \eqref{PI-ub} is attained for $\alpha=\alpha_{n,\beta}$ and $w=w_{\beta}$, at least for $\beta\geq0$ small enough.
\begin{prop}\label{extremalX} Let $w=w_{\beta}(r)=(\ln \frac{1}{r})^{\frac{\beta n}{2}}$ in \eqref{PI-ub}. There exists $\beta_0\in (0,1)$ such that $\mathcal{MT}(n,\alpha_{n,\beta},\beta)$ is attained for any $\beta\in (0,\beta_0)$.
\end{prop}
\begin{proof}
 From Lemma \ref{S-classic} and \eqref{b-c}, we have
 \begin{equation*}
\mathcal{MT}(n,\alpha_{n,0},0)>\dfrac1n\Big(1+\mathrm{e}^{\Psi(\frac{n}{2}+1)+\gamma}\Big)\geq J_{n,0}^\delta(0).
 \end{equation*}
 Using Lemma \ref{APPROX} and Lemma \ref{monotonicity}, there exists $\beta_0\in(0,1)$ such that
 \begin{equation}\label{eq424}
\mathcal{MT}(n,\alpha_{n,\beta},\beta)>J^\delta_{n,0}(0)\geq J^\delta_{n,\beta}(0),
 \end{equation}
 for all $\beta\in[0,\beta_0)$.
 
 By combining 
 \eqref{b-c}, Lemma~\ref{S-classic}, Lemma~\ref{monotonicity}, and Lemma~\ref{APPROX}, we conclude that there exists $\beta_0\in (0,1)$ such that 
    \begin{equation*}
       \mathcal{MT}(n,\alpha_{n,\beta},\beta) > \frac{1}{n}\Big(1+\mathrm{e}^{\Psi(\frac{n}{2}+1)+\gamma}\Big)
       \ge J^{\delta}_{n,0}(0)
       \ge J^{\delta}_{n,\beta}(0)
    \end{equation*}
    for all $\beta\in [0,\beta_0)$.

    Fix $\beta\in[0,\beta_0)$ and let $(v_j)$ be a maximizing sequence for $\mathcal{MT}(n,\alpha_{n,\beta},\beta)$. By passing to a subsequence, we have $v_j\rightharpoonup v_0$ in $X^{1,k+1}_{1,w_\beta}$ for some $v_0\in X^{1,k+1}_{1,w_\beta}$. We claim that
\begin{equation}\label{eq425}
v_0\not\equiv0.
\end{equation}

    Assume, for the sake of contradiction, that $v_0\equiv0$. If $(v_j)$ is a NCS, then
    \begin{equation*}
    \mathcal{MT}(n,\alpha_{n,\beta},\beta)=\lim_{j\to\infty}J_{n,\beta}(v_j)\leq J^\delta_{n,\beta}(0),
    \end{equation*}
    which contradicts \eqref{eq424}. Thus, $(v_j)$ is not a NCS. Consequently, there exist $a,r_0\in(0,1)$ such that for every $j_0$, there exists $j\geq j_0$ satisfying
    \begin{equation*}
    \int_0^{r_0}r^{n-k}w_\beta(r)|v_j^{\prime}|^{k+1}\mathrm \,\mathrm{d} r\leq a.
    \end{equation*}
    Using similar arguments as in the proof of Lemma \ref{monotonicity} (specifically, the steps used following \eqref{eq415} and \eqref{eq413}), we obtain
    \begin{equation*}
    \mathcal{MT}(n,\alpha_{n,\beta},\beta)=\lim_{j\to\infty}J_{n,\beta}(v_j)=\frac1n\leq J^\delta_{n,\beta}(0),
    \end{equation*}
    which again contradicts \eqref{eq424}. Therefore, the proof of \eqref{eq425} is complete.

    By Lemma \ref{Lions-improvement}, we have
    \begin{equation*}
    \limsup_{j\to\infty}\int_0^1r^{n-1}\mathrm{e}^{p\alpha_{n,\beta}|v_j|^{\gamma_{n,\beta}}}\mathrm \,\mathrm{d} r<\infty
    \end{equation*}
    for all $0<p<(1-\|v_0\|^{k+1}_{w_\beta})^{-\frac{2}{n(1-\beta)}}$. Note that $(1-\|v_0\|^{k+1}_{w_\beta})^{-\frac{2}{n(1-\beta)}}>1$ by \eqref{eq425}. On the other hand, by Corollary \ref{thm0}, $v_j\to v_0$ in $L^q_{n-1}(0,1)$ for all $1\leq q<\infty$. Following similar ideas as in \eqref{expineq}, we obtain
    \begin{equation*}
    \mathcal{MT}(n,\alpha_{n,\beta},\beta)=\lim_{j\to\infty}J_{n,\beta}(v_j)=J_{n,\beta}(v_0).
    \end{equation*}

This implies that $\mathcal{MT}(n,\alpha_{n,\beta},\beta)$ is attained by $v_0$. To prove that $\|v_0\|_{w_\beta}=1$, observe that $v_j\rightharpoonup v_0$ in $X^{1,k+1}_{1,w_\beta}$ and that $J_{n,\beta}(v_0/\|v_0\|_{w_\beta})>J_{n,\beta}(v_0)$ if $\|v_0\|_{w_\beta}<1$. This concludes the proof of the proposition.
    \end{proof}
\section{Existence of the \texorpdfstring{$k$}{}-admissible maximizers}\label{section6}
This section aims to use the extremal function $v\in X^{1,k+1}_{1,w_{\beta}}$ for $\mathcal{MT}(n,\alpha_{n,\beta},\beta)$, ensured by Proposition~\ref{extremalX}, to build an extremal function $u\in\Phi^{k}_{0,\mathrm{rad}}(B, w_{\beta})$ for the supremum $\mathrm{MT}(n,\alpha_{n,\beta},\beta)$ and then we complete the prove of Theorem~\ref{thm-extremal}.
 
\subsection{Regularity for auxiliary maximizers} 
Next, we prove that if $v\in X^{1,k+1}_{1,w_{\beta}}$ is a maximizers $\mathcal{MT}(n,\alpha_{n,\beta},\beta)$ then $v\in C^{2}[0,1]$. Taking into account Corollary~\ref{corollary-radial}, it is reasonable to be concerned with the behavior of $v$ near $r = 0$. The next result provides a good insight into this matter.

\begin{lemma}\label{0-beha} Let $v\in X^{1,k+1}_{1,w_{\beta}}$ with $k=n/2$ be a non-increasing function. Then, for any $\vartheta,\alpha>0$ and $0\leq \beta<1$ it holds
\begin{equation*}
\lim_{r\to 0^+} r^{\vartheta}|v(r)|^{\gamma_{n,\beta}-1}\mathrm{e}^{\alpha|v(r)|^{\gamma_{n,\beta}}}=0,\;\;\mbox{where}\;\; \gamma_{n,\beta}=\frac{1}{1-\beta}\frac{n+2}{n}
\end{equation*}
\end{lemma}
\begin{proof}
We claim that there exist $m>1$ and $C>0$ depending only on $m, n, \beta$ and $\vartheta$ such that 
\begin{equation}\label{emb1}
\|v^{\prime}\|_{L^{m}_{\vartheta}}\leq C\|v\|_{w_{\beta}}.
\end{equation}
Indeed, taking $m, q>1$ such that
$(\vartheta+1)q>k+1\ge mq$, the H\"{o}lder inequality yields 
\begin{equation}\nonumber
\begin{aligned}
\int_{0}^{1}r^{n-k}|v^{\prime}|^{mq}w_{\beta}(r)\,\mathrm{d} r&\leq \left(\int^{1}_{0}r^{n-k}|v^{\prime}|^{k+1}w_{\beta}(r)\,\mathrm{d} r\right)^{\frac{mq}{k+1}}\left(\int^{1}_{0}r^{n-k}w_{\beta}(r)\,\mathrm{d} r\right)^{\frac{\frac{k+1}{mq}-1}{\frac{k+1}{mq}}}\\
&= C\left(\int^{1}_{0}r^{n-k}w_{\beta}(r)\,\mathrm{d} r\right)^{\frac{\frac{k+1}{mq}-1}{\frac{k+1}{mq}}}\|v\|^{mq}_{w_{\beta}},
\end{aligned}
\end{equation}
with $C=c_n^{-mq/(k+1)}$.
In addition, using the above estimate and our choices for $m$ and $q$, we can write (recall $k=n/2$ and $(\vartheta-\frac{k}{q})\frac{q}{q-1}>-1$)
 \begin{align*}
\int_{0}^1 r^{\vartheta}|v^{\prime}|^{m}\mathrm{d} r&=\int_{0}^1\big( 
|v^{\prime}|^{m}r^{\frac{n-k}{q}}w^{\frac{1}{q}}_{\beta}(r)\big)\big(r^{\vartheta-\frac{k}{q}}w^{-\frac{1}{q}}_{\beta}(r)\big)\mathrm{d}
r\\
&\leq \left(\int_{0}^{1}r^{n-k}|v^{\prime}|^{mq}w_{\beta}(r)\mathrm{d}
r\right)^{\frac{1}{q}}\left(\int_{0}^{1}r^{(\vartheta-\frac{k}{q})\frac{q}{q-1}}w^{-\frac{1}{q-1}}_{\beta}(r)\mathrm{d}
r\right)^{\frac{q-1}{q}}\\
&\leq  C\|v\|^{m}_{w_{\beta}},
 \end{align*}
where we have use that
\begin{equation}\nonumber
    \begin{aligned}
      \int_{0}^{1}r^{(\vartheta-\frac{k}{q})\frac{q}{q-1}}w^{-\frac{1}{q-1}}_{\beta}(r)\mathrm{d}
r& =\int_{0}^{1}r^{(\vartheta-\frac{k}{q})\frac{q}{q-1}}(-\ln r)^{-\frac{\beta n}{2(q-1)}}\mathrm{d}
r  \\
&= \int_{0}^{\infty}\mathrm{e}^{-[(\vartheta-\frac{k}{q})\frac{q}{q-1}+1]t} t^{-\frac{\beta n}{2(q-1)}}\mathrm{d}t\\
&= \left[\left(\vartheta-\frac{k}{q}\right)\frac{q}{q-1}+1\right]^{\frac{\beta n}{2(q-1)}-1}\int_{0}^{\infty}x^{-\frac{\beta n}{2(q-1)}}\mathrm{e}^{-x}\mathrm{d}x\\
&=\left[\left(\vartheta-\frac{k}{q}\right)\frac{q}{q-1}+1\right]^{\frac{\beta n}{2(q-1)}-1}\Gamma\Big(1-\frac{\beta n}{2(q-1)}\Big)
    \end{aligned}
\end{equation}
provided that $q$ is chosen large such that $1-\frac{\beta n}{2(q-1)}>0$. This proves \eqref{emb1}. Now, we set 
\begin{equation*}
\varphi(r)=|v(r)|^{\gamma_{n,\beta}-1}\mathrm{e}^{\alpha|v(r)|^{\gamma_{n,\beta}}}.
\end{equation*}
It remains to show that
\begin{equation*}
\lim_{r\rightarrow 0^{+}}r^{\vartheta}\varphi(r)=0.
\end{equation*}
Since $v$ is a non-increasing function, without loss of generality we can assume $v>0$ in $(0,1/2)$ and $\lim_{r\rightarrow 0^{+}}v(r)=\infty$. Hence, there exists $C=C(n, \alpha,\beta)>0$ such that (use that $1/v$ is bounded near $0$)
\begin{eqnarray}
|\varphi^{\prime}(r)|&\leq &(\gamma_{n,\beta}-1)|v(r)|^{\gamma_{n,\beta}-2}
|v^{\prime}(r)| \mathrm{e}^{\alpha |v(r)|^{\gamma_{n,\beta}}}+\alpha\gamma_{n,\beta}|v(r)|^{2(\gamma_{n,\beta}-1)}|v^{\prime}(r)|
\mathrm{e}^{\alpha |v(r)|^{\gamma_{n,\beta}}}\nonumber\\
&\leq & C\left(
|v(r)|^{\gamma_{n,\beta}}|v^{\prime}(r)| \mathrm{e}^{\alpha |v(r)|^{\gamma_{n,\beta}}}+ |v(r)|^{2\gamma_{n,\beta}}|v^{\prime}(r)|
\mathrm{e}^{\alpha |v(r)|^{\gamma_{n,\beta}}}\right)\label{11fev201301},
\end{eqnarray}
for any $r\in(0,1/2)$. Now, we claim that 
\begin{equation}\label{phi-claim}
\varphi^{\prime}\in L^{1}_{\vartheta}(1,1/2).
\end{equation}
Firstly, by arguing as in steps \eqref{changeV}, \eqref{portal-norm}, \eqref{portal-funcional}, \eqref{psi-Int} and \eqref{Holder-radial} in the proof of Proposition~\ref{prop-thm1} (see also Theorem~1.1 in the recent paper \cite{XueZhangZhu2025}), we can see that 
\begin{equation}\label{Log-tmG}
    \int_{0}^{1}r^{\vartheta}\mathrm{e}^{\alpha p|v|^{\gamma_{n,\beta}}}\mathrm{d}r<\infty
\end{equation}
for any $p>1$ and $\alpha, \vartheta>0$.
Thus, for $m>1$ such as in \eqref{emb1}, $p>1$ and $q>1$ (large) such that
\begin{equation*}
\frac{2\gamma_{n,\beta}}{q}+\frac{1}{m}+\frac{1}{p}\leq 1
\end{equation*}
the H\"{o}lder inequality yields
\begin{align*}\nonumber
& \int_{0}^{1}r^{\vartheta}|v(r)|^{2\gamma_{n,\beta}}|v^{\prime}(r)|
\mathrm{e}^{\alpha |v(r)|^{\gamma_{n,\beta}}}\mathrm{d}r\\
& \le\left(\int_{0}^{1}r^{\vartheta}|v|^{q}\mathrm{d}r\right)^{\frac{2\gamma_{n,\beta}}{q}}\left(\int_{0}^{1}r^{\vartheta}|v^{\prime}|^{m}\mathrm{d}r\right)^{\frac{1}{m}}\left(\int_{0}^{1}r^{\vartheta}\mathrm{e}^{\alpha p|v|^{\gamma_{n,\beta}}}\mathrm{d}r\right)^{\frac{1}{p}}.
\end{align*}
The above estimate combined with \eqref{emb1}, Corollary~\ref{thm0}-(ii) and \eqref{Log-tmG} gives
\begin{equation}\label{phi-part1}
|v(r)|^{2\gamma_{n,\beta}}|v^{\prime}(r)|
\mathrm{e}^{\alpha |v(r)|^{\gamma_{n,\beta}}}\in L^{1}_{\vartheta}(0,1).
\end{equation}
Analogously, we can show that 
\begin{equation}\label{phi-part2}
|v(r)|^{\gamma_{n,\beta}}|v^{\prime}(r)|
\mathrm{e}^{\alpha |v(r)|^{\gamma_{n,\beta}}}\in L^{1}_{\vartheta}(0,1).
\end{equation}
Hence, combining \eqref{11fev201301}, \eqref{phi-part1}, and \eqref{phi-part2}, we obtain \eqref{phi-claim}. For $0<r<s<1/2$ we have 
\begin{equation}\nonumber
|\varphi(r)|\leq \varphi\Big(\frac{1}{2}\Big)+\int_{r}^{s}|\varphi^{\prime}(\tau)|\mathrm{d}\tau+\int_{s}^{\frac{1}{2}}|\varphi^{\prime}(\tau)|\mathrm{d}\tau.
\end{equation}
Therefore, 
\begin{equation}\label{phi-calculusF}
r^{\vartheta}|\varphi(r)|\leq r^{\vartheta}\varphi\Big(\frac{1}{2}\Big)+\int_{r}^{s}\tau^{\vartheta}|\varphi^{\prime}(\tau)|\mathrm{d}\tau+\frac{r^{\vartheta}}{s^{\vartheta}}\int_{s}^{\frac{1}{2}}\tau^{\vartheta}|\varphi^{\prime}(\tau)|\mathrm{d}\tau.
\end{equation}
From \eqref{phi-claim}, for each $\epsilon>0$ there exists $s>0$ (small) such that
\begin{equation*}\int_{r}^{s}\tau^{\vartheta}|\varphi^{\prime}(\tau)|\mathrm{d}\tau<\frac{\epsilon}{3}.
\end{equation*}
Also, there exists $0<\delta< s$ such that
\begin{equation*}\delta^{\vartheta}\left|\varphi\Big(\frac{1}{2}\Big)\right|<\frac{\epsilon}{3}\;\;\;\mbox{and}\;\;\;
\frac{\delta^{\vartheta}}{s^{\vartheta}}\int_{s}^{\frac{1}{2}}\tau^{\vartheta}|\varphi^{\prime}(\tau)|\mathrm{d}\tau<\frac{\epsilon}{3}.\end{equation*}
Hence, using \eqref{phi-calculusF}, we conclude that 
\begin{equation*}
r^{\vartheta}|\varphi(r)|<\epsilon, \;\;\;\mbox{for}\;\;\; 0<r<\delta
\end{equation*}
which completes the proof.
\end{proof}
\begin{lemma}\label{vc2}
Let $v\in X^{1,k+1}_{1,w_{\beta}}$ with $k=n/2$ be a non-negative maximizer for $\mathcal{MT}(n,\alpha_{n,\beta},\beta)$. Then, $v$ is a decreasing function and $v\in C^{2}[0,1]$.
\end{lemma}
\begin{proof}
Let $I:X^{1,k+1}_{1,w_{\beta}}\to \mathbb{R}$ be the logarithmic Trudinger-Moser functional 
\begin{equation*}
    I(v)=\int_{0}^{1} r^{n-1}\mathrm{e}^{\alpha_{n,\beta}|v|^{\gamma_{n,\beta}}}\mathrm{d} r.
\end{equation*}
A standard argument shows that $I \in C^1 $ and such that 
\begin{equation}\nonumber
  I^{\prime}(v).h =   \alpha_{n,\beta}\gamma_{n,\beta}\int_{0}^{1}r^{n-1}|v|^{\gamma_{n,\beta}-2}\mathrm{e}^{\alpha_{n,\beta}|v|^{\gamma_{n,\beta}}} v h dr.
\end{equation}
Then, if $v\in X^{1,k+1}_{1,w_{\beta}}$ is a non-negative maximizer for $\mathcal{MT}(n,\alpha_{n,\beta},\beta)$, the Lagrange multipliers theorem yields, for any $h\in X^{1,k+1}_{1,w_{\beta}}$
\begin{equation}\label{weaksolution}
c_{n}\int_{0}^{1}r^{n-k}|v^{\prime}|^{k-1}v^{\prime}h^{\prime}w_{\beta}\,\mathrm{d}
r=\lambda \int_{0}^{1}r^{n-1}|v|^{\gamma_{n,\beta}-1}\mathrm{e}^{\alpha_{n,\beta}|v|^{\gamma_{n,\beta}}}h\, dr
\end{equation}
where
\begin{equation*}
\lambda=\Big(\int_{0}^{1}r^{n-1}|v|^{\gamma_{n,\beta}}\mathrm{e}^{\alpha_{n,\beta}|v|^{\gamma_{n,\beta}}}\, dr\Big)^{-1}.
\end{equation*}
Following \cite{zbMATH01127685}, for each $r\in(0,1)$ and $\rho>0$ let
$h_{\rho}\in X^{1,k+1}_{1,w_{\beta}}$ be given by
\begin{equation*}
h_{\rho}(s)=\left\{\begin{aligned}&\;1 &\mbox{if}\quad &0\leq s\leq r,&\\
&\;1+\frac{1}{\rho}(r-s)&\mbox{if}\quad &r\leq s\leq r+\rho,&\\
&\;0 &\mbox{if}\quad &s\ge r+\rho.&
\end{aligned}\right.
\end{equation*}
By using $h_{\rho}$ in \eqref{weaksolution} and letting
$\rho\rightarrow 0$, we deduce
\begin{equation}\label{integralsolution}
c_{n}r^{n-k}(-|v^{\prime}|^{k-1}v^{\prime})w_{\beta}(r)=\lambda\int_{0}^{r}s^{n-1}|v|^{\gamma_{n,\beta}-1}\mathrm{e}^{\alpha_{n,\beta}|v|^{\gamma_{n,\beta}}}\,\mathrm{d}s,\;\;\mbox{a.e
on}\;\; (0,1).
\end{equation}
It follows that $v$ is a decreasing function such that $v\in C^{1}(0,1]$. In addition, from \eqref{integralsolution}
\begin{equation}\label{gradiente}
(-v^{\prime}(r))^{k}w_{\beta}(r)=\frac{\lambda}{c_nr^{n-k}}\int_{0}^{r}s^{n-1}|v|^{\gamma_{n,\beta}-1}\mathrm{e}^{\alpha_{n,\beta}|v|^{\gamma_{n,\beta}}}\,\mathrm{d}s.
\end{equation}
Hence, from Lemma~\ref{0-beha} and L'Hospital's rule we get 
\begin{equation*}
\lim_{r\to 0}(-v^{\prime}(r))^{k}\frac{w_{\beta}(r)}{r^{\sigma}}= 0,\;\;\;\mbox{ for all }\;\; \sigma<k.
\end{equation*}
Thus, since $w_{\beta}(r)\to\infty$ as $r\to 0^{+}$, we obtain $\lim_{r\rightarrow 0^{+}}v^{\prime}(r)=0$ and then $v\in C^{1}[0,1]$. In addition, taking into account that we already know $v \in C^{1}[0,1]$, the identity \eqref{gradiente} gives 
\begin{equation}\label{gradienteRWK}
\begin{aligned}
\lim_{r\to 0}\Big(-\frac{v^{\prime}(r)}{r}\Big)^{k}w_{\beta}(r)&=\frac{\lambda}{c_n}\lim_{r\to 0}\frac{1}{r^{n}}\int_{0}^{r}s^{n-1}|v|^{\gamma_{n,\beta}-1}\mathrm{e}^{\alpha_{n,\beta}|v|^{\gamma_{n,\beta}}}\,\mathrm{d}s\\
&=\frac{\lambda}{nc_n}|v(0)|^{\gamma_{n,\beta}}\mathrm{e}^{\alpha_{n,\beta}|v(0)|^{\gamma_{n,\beta}}}>0.
\end{aligned}
\end{equation}
In particular, we also have 
\begin{equation}\label{gradienteRR00}
\begin{aligned}
\lim_{r\to 0}\Big(-\frac{v^{\prime}(r)}{r}\Big)=0.
\end{aligned}
\end{equation}
Now, by using \eqref{gradiente} again, we get $v\in C^{2}(0,1]$ and we can write
\begin{equation}\nonumber
\begin{aligned}
-k(-v^{\prime}(r))^{k-1}v^{\prime\prime}(r)w_{\beta}(r) -\frac{n\beta}{2}(-v^{\prime}(r))^{k}\frac{w_{\beta}(r)}{r(-\ln r)} &=-\frac{n-k}{r}(-v^{\prime}(r))^{k}w_{\beta}(r)\\
& + \frac{\lambda}{c_n}\frac{1}{r^{1-k}}|v(r)|^{\gamma_{n,\beta}-1}\mathrm{e}^{\alpha_{n,\beta}|v(r)|^{\gamma_{n,\beta}}}
\end{aligned}
\end{equation}
which can be rewritten as

\begin{equation}\nonumber
\begin{aligned}
v^{\prime\prime}(r)+\beta\Big(-\frac{v^{\prime}(r)}{r}\Big)\frac{1}{(-\ln r)} &=\frac{n-k}{k}\Big(-\frac{v^{\prime}(r)}{r}\Big)\\
& - \frac{\lambda}{kc_n}\Big(-\frac{v^{\prime}(r)}{r}\Big)\frac{1}{\Big(-\frac{v^{\prime}(r)}{r}\Big)^{k}w_{\beta}(r)}|v(r)|^{\gamma_{n,\beta}-1}\mathrm{e}^{\alpha_{n,\beta}|v(r)|^{\gamma_{n,\beta}}}.
\end{aligned}
\end{equation}
By using \eqref{gradienteRWK} and \eqref{gradienteRR00}, we obtain $\lim_{r\to 0}v^{\prime\prime}(r)=0$ and consequently we conclude that
$v\in C^2[0,1]$ as desired.
\end{proof}
\subsection{Constructing \texorpdfstring{$k$}{}-admissible maximizers: Proof of Theorem~\ref{thm-extremal}}
Let $v\in X^{1,k+1}_{1,w_{\beta}}$ be a non-negative extremal function for $\mathcal{MT}(n,\alpha_{n,\beta},\beta)$, for $\beta>0$ small enough, ensured by Proposition~\ref{extremalX}. Now, we consider $u: \overline{B}\rightarrow\mathbb{R}$ given by
\begin{equation}\label{u0difinition}
u(x)=-v(r),\;\; \mbox{with}\;\; r=|x|, \;\; \mbox{for all}\;\; x\in \overline{B}.
\end{equation}
\begin{lemma}\label{k-admi} Let $u: \overline{B}\rightarrow\mathbb{R}$ be defined by \eqref{u0difinition}. Then $u\in \Phi^{k}_{0,\mathrm{rad}}(B, w_{\beta})$ and it is an extremal function for $\mathrm{MT}(n,\alpha_{n,\beta},\beta)$.
\end{lemma}
\begin{proof}
From Lemma~\ref{vc2} we have $v\in C^{2}[0,1]$ and consequently $u\in C^{2}(B)$ with $u_{|\partial B}=0$. To prove $u\in \Phi^{k}_{0,\mathrm{rad}}(B, w_{\beta})$, we will first show that $u$ is a $k$-admissible function, that is, $S_j(D^2u)\ge 0$ for $j=1,2,\dots, k$ holds. By definition, $u$ is a radially symmetric function. Then the radial form of the $k$-Hessian operator yields
\begin{equation*}
S_{j}(D^2u)=\frac{1}{j}\binom{n-1}{j-1}r^{1-n}\left(r^{n-j}(u^{\prime})^{j}\right)^{\prime}.
\end{equation*}
Thus, to conclude the $k$-admissibility of $u$, it is sufficient to show that
\begin{equation*}
\left(r^{n-j}(u^{\prime})^{j}\right)^{\prime}\ge 0, \quad \mbox{for}\;\; j=1,\dots, k,
\end{equation*}
or equivalently
\begin{equation}\label{v-kadmi}
\left(r^{n-j}(-v^{\prime})^{j}\right)^{\prime}\ge 0, \quad \mbox{for}\;\; j=1,\dots, k
\end{equation}
Using \eqref{gradiente} and the assumption $n=2k$, it is easy to show that 
\begin{equation}\nonumber
\begin{aligned}
r^{n-j}(-v^{\prime})^{j}
&=r^{n-2j}\Theta^{\frac{i}{k}}(r),
\end{aligned}
\end{equation}
where 
\begin{equation*}
\Theta(r)=\frac{\lambda}{c_n w_{\beta}(r)}\int_{0}^{r}s^{n-1}|v|^{\gamma_{n,\beta}-1}\mathrm{e}^{\alpha_{n,\beta}|v|^{\gamma_{n,\beta}}}\,\mathrm{d}s.
\end{equation*}
Hence,
\begin{equation}\label{HessianAdss}
\begin{aligned}
  (r^{n-j}(-v^{\prime})^{j})^{\prime}&= (n-2j)r^{n-2j-1}\Theta^{\frac{i}{k}}(r)+r^{n-2j}\frac{j}{k}\Theta^{\frac{i}{k}-1}(r)\Theta^{\prime}(r)\\
   &=r^{n-2j}\Theta^{\frac{j}{k}}(r)\left[\frac{n-2j}{r}+\frac{j}{k}\frac{\Theta^{\prime}(r)}{\Theta(r)}\right].
\end{aligned}
\end{equation}
Note that $\Theta>0$ on $(0,1]$ and 
\begin{equation*}\Theta^{\prime}(r)=\frac{n\beta}{2}\frac{(-\ln r)}{r}\Theta(r)+\frac{\lambda}{c_n w_{\beta}(r)}\left(r^{n-1}|v|^{\gamma_{n,\beta}-1}\mathrm{e}^{\alpha_{n,\beta}|v|^{\gamma_{n,\beta}}}\right).\end{equation*}
Thus, $\Theta^{\prime}\ge 0$ on $(0, 1]$ and \eqref{HessianAdss} forces \eqref{v-kadmi} which proves that $u\in \Phi^{k}_{0,\mathrm{rad}}(B)$. From \eqref{gate-norm}, we have $\|u\|_{\Phi, w_{\beta}}=\|v\|_{w_{\beta}}=1$ and consequently $u\in \Phi^{k}_{0,\mathrm{rad}}(B, w_{\beta}) $. In addition, from \eqref{gate-fuctional} we can write
\begin{equation*}
\int_{B}\mathrm{e}^{\alpha_{n,\beta}|u|^{\gamma_{n,\beta}}}\, \mathrm{d} x=\omega_{n-1}\int_{0}^{1}r^{n-1}\mathrm{e}^{\alpha_{n,\beta}|v|^{\gamma_{n,\beta}}}\,\mathrm{d} r=\omega_{n-1}\mathcal{MT}(n,\alpha_{n,\beta},\beta)\geq \mathrm{MT}(n,\alpha_{n,\beta},\beta).
\end{equation*}
Thus, 
\begin{equation*}
\int_{B}\mathrm{e}^{\alpha_{n,\beta}|u|^{\gamma_{n,\beta}}}\, \mathrm{d} x = \mathrm{MT}(n,\alpha_{n,\beta},\beta),
\end{equation*}
which is the desired conclusion.
\end{proof}

\section{Potential Further Developments}
Our analysis to obtain optimal estimates provides new insights into the variational study of 
$k$-Hessian equations in the extreme regime. In the following, we highlight key aspects that complement and further refine the analysis presented in this work.
\begin{itemize}
\item Based on the sharp inequalities obtained in Theorems \ref{thm1} and \ref{thm2}, one can variationally study the $k$-Hessian equation with a logarithmic weight in the critical regime - where the energy functional experiences a loss of compactness - similar to the approach taken for the Laplacian equation in \cite{CALANCHI20151967}.
    \item In Theorem~\ref{thm-extremal}, we establish attainability for sufficiently small values of $\beta>0$. Unfortunately, as in the analogous results obtained by Nguyen~\cite{zbMATH07122708} and Roy~\cite{zbMATH07074659,zbMATH06562450}, we do not provide an optimal constant $\beta_0$ for which attainability is ensured when $\beta$ is smaller than $\beta_0$. It is important to note that this issue remains open even in the Laplacian case, i.e., when $k=1$.
    \item Propositions~\ref{prop-thm1} and \ref{prop-thm2} address the first-derivative case, while Jiang, Xu, Zhang, and Zhu~\cite{MR4870665} studied a similar problem in the space $X^{2,p}_{1,w}$ for second-order derivatives. The higher-order derivative case remains to be studied.
    \item In Section~\ref{section2}, we established Hardy inequalities for $v\in AC_R(0,R)$ with $0<R<\infty$, as this case was essential to our study. However, by applying analogous results from Opic and Kufner~\cite{zbMATH00046945}, it is also possible to study the remaining three cases, where $v\in AC_L(0,R)$ or $R=\infty$. Here $AC_L(0,R)$ denotes the space of locally absolutely continuous functions that tend to zero as $r\to0$.
\end{itemize}

\subsection*{Funding}
\begin{sloppypar}
J. M. do \'O acknowledges partial support from CNPq through grants 312340/2021-4, 409764/2023-0, 443594/2023-6, CAPES MATH AMSUD grant 88887.878894/2023-00 and Para\'iba State Research Foundation (FAPESQ), grant no 3034/2021. J. F de Oliveira acknowledges partial support from CNPq through grant 309491/2021-5. R. C. Ponciano acknowledges partial support from São Paulo Research Foundation (FAPESP) grant 2023/07697-9.\\
{\bf Data availability statement:} Our manuscript has no data associated or further material. \\
 {\bf Ethical Approval:}  All data generated or analyzed during this study are included in this article.\\
 {\bf Conflict of interest:} The authors declare no conflict of interest. \\
\end{sloppypar}

\end{document}